\documentclass [leqno, spanish, 10pt]{amsart}

\usepackage[T1]{fontenc}
\usepackage[utf8]{inputenc}

\usepackage{etoolbox} 	
\usepackage{graphicx}
\usepackage{epstopdf}

\usepackage{amssymb, amsmath, amsfonts, amsthm, amsbsy, amscd, stmaryrd}
\usepackage[bbgreekl]{mathbbol} 
\usepackage[mathscr]{euscript}
\usepackage{upgreek}
\usepackage[bbgreekl]{mathbbol} 
\usepackage{esint}

\usepackage{xypic}
\usepackage[all]{xy}
\usepackage{tikz}
\usepackage{etex} 

\usepackage{hyperref}
\hypersetup{
	colorlinks=true, linktocpage=true, 
}
\usepackage{filecontents}

\usepackage{mathtools}
\usepackage{enumerate}
\usepackage{url}
\usepackage{etoolbox}
\apptocmd{\sloppy}{\hbadness 10000\relax}{}{}

\usepackage{thmtools} 
\newtheorem{theorem}{Theorem}[section]
\newtheorem{corollary}[theorem]{Corollary}
\newtheorem{lemma}[theorem]{Lemma}
\newtheorem{definition}[theorem]{Definition}

\theoremstyle{definition}
\newtheorem{remark}[theorem]{Remark}

\makeatletter
\def\l@subsection{\@tocline{2}{0pt}{1pc}{4.6em}{}}
\renewcommand{\tocsubsection}[3]{%
  \indentlabel{\@ifnotempty{#2}{\hspace*{2.3em}\makebox[2.3em][l]{%
    \ignorespaces#1 #2.\hfill}}}#3}
\makeatother

\numberwithin{equation}{section}

\newcommand{\cp}{\mathbb{CP}}
\newcommand{\bcp}{\overline{\mathbb{CP}}}
\newcommand{\W}{\mathscr{W}}

\newcommand{\ste}{\mathbb{V}}
\newcommand{\std}{\mathbb{V}^{\ast}}
\newcommand{\wt}{\mathscr{C}}
\newcommand{\bric}{\overline{\ric}}
\renewcommand{\pmod}{\mathscr{RP}^{2}}
\DeclareMathOperator\arccot{arccot}
\DeclareMathOperator\arccoth{arccoth}
\DeclareMathOperator\arctanh{arctanh}
\newcommand{\sphere}{\mathbb{S}^{2}}
\newcommand{\onesphere}{\mathbb{S}^{1}}
\newcommand{\degen}{\mathsf{dg}}
\newcommand{\crit}{\mathsf{crit}}
\newcommand{\zero}{\mathsf{Z}}
\newcommand{\imt}{\iota}
\newcommand{\mx}{\kappa}
\newcommand{\yam}{\mathscr{Y}}

\newcommand{\jfib}{\mathbb{J}}

\newcommand{\inc}{\iota}
\newcommand{\sinc}{\mathsf{s}}

\newcommand{\cm}{\mathscr{C}}

\newcommand{\X}{\mathscr{X}}

\newcommand{\lop}{\mathscr{L}}
\newcommand{\dn}{d_{\nabla}}
\newcommand{\dbn}{d_{\bnabla}}
\newcommand{\dnw}{d_{\nabla}^{\circ}}
\newcommand{\sd}{\delta}

\newcommand{\mf}{\Psi}

\newcommand{\dop}{\mathscr{P}}

\newcommand{\sig}{\si}
\newcommand{\rf}{\rho}
\newcommand{\rfc}{\varrho}

\newcommand{\pon}{\mathsf{p}}
\newcommand{\curv}{\mathcal{C}}

\newcommand{\ex}{\mathscr{E}}

\newcommand{\sOm}{\mathbb{\Omega}}
\newcommand{\sOmega}{\mathbb{\Omega}}
\renewcommand{\H}{\mathscr{H}}
\newcommand{\Hs}{\mathscr{M}}
\newcommand{\U}{\mathscr{U}}
\newcommand{\acom}{\mathbb{J}}
\newcommand{\intcom}{\acom_{\text{int}}}
\newcommand{\ldc}{\jmath}

\newcommand{\symcon}{\mathbb{S}}
\newcommand{\projcon}{\mathbb{P}}
\newcommand{\projkil}{\mathcal{P}}

\newcommand{\hm}{\mathsf{H}}

\newcommand{\Ham}{\mathbb{Ham}}
\newcommand{\Symplecto}{\mathbb{Symp}}
\newcommand{\ham}{\mathbb{ham}}
\newcommand{\symplecto}{\mathbb{symp}}

\newcommand{\Om}{\Omega}

\newcommand{\lc}{\mathsf{L}}

\newcommand{\sflat}{\curlyvee}
\newcommand{\ssharp}{\curlywedge}

\newcommand{\prin}{\mathsf{P}}

\newcommand{\weylmod}{\mathscr{W}}

\newcommand{\vr}{\updelta}
\newcommand{\vc}{\tau}
\newcommand{\dad}{d^{\prime}}

\newcommand{\vol}{\mathsf{vol}}

\newcommand{\affcon}{\mathbb{A}}

\newcommand{\sR}{\mathscr{R}}
\newcommand{\sRo}{\mathscr{S}}
\newcommand{\sWo}{\mathscr{T}}

\newcommand{\dum}{\,\cdot\,\,}
\newcommand{\Ga}{\Gamma}

\newcommand{\ric}{\mathsf{Ric}}
\newcommand{\riem}{\mathsf{Riem}}
\newcommand{\Q}{\mathcal{Q}}
\newcommand{\lap}{\Delta}
\newcommand{\prj}{\mathcal{P}}

\newcommand{\Omk}[1]{\tfrac{\Omega^{#1}}{#1!}}

\renewcommand{\sp}{\mathfrak{sp}}
\newcommand{\hop}{\mathscr{H}}

\newcommand{\la}{\lambda}
\newcommand{\ep}{\epsilon}
\newcommand{\reat}{\mathbb{R}^{\times}}

\newcommand{\ext}{\Lambda}
\newcommand{\cinf}{C^{\infty}}

\newcommand{\ben}{[\Bar{\nabla}]}

\newcommand{\eno}{\operatorname{End}}

\newcommand{\si}{\sigma}
\newcommand{\pr}{\partial}
\newcommand{\ctm}{T^{\ast}M}

\newcommand{\bnabla}{\bar{\nabla}}

\newcommand{\en}{[\nabla]}

\newcommand{\lie}{\mathfrak{L}}

\newcommand{\F}{\mathcal{F}}

\newcommand{\C}{\mathcal{C}}
\newcommand{\lb}{\langle}
\newcommand{\ra}{\rangle}
\newcommand{\A}{\mathcal{A}}
\newcommand{\al}{\alpha}
\newcommand{\be}{\beta}
\newcommand{\ga}{\gamma}

\newcommand{\emf}{\mathcal{E}}

\newcommand{\tnabla}{\tilde{\nabla}}

\newcommand{\sll}{\mathfrak{sl}}

\newcommand{\jac}{\mathscr{J}}

\newcommand{\vect}{\mathfrak{vec}}

\DeclareMathOperator{\diff}{Diff}

\newcommand{\g}{\mathfrak{g}}

\newcommand{\tensor}{\otimes}

\newcommand{\rea}{\mathbb R}

\newcommand{\tr}{\operatorname{\mathsf{tr}}}

\newcommand{\K}{\mathscr{K}}

\newcommand{\ct}{T^{\ast}}

\renewcommand{\L}{\mathscr{L}}

\renewcommand{\P}{\mathcal{P}}

\newcommand{\dens}{\mathcal{V}}

\begin{document}
\title{Critical symplectic connections on surfaces}

\author{Daniel J. F. Fox} 
\address{Departamento de Matemática Aplicada a la Ingeniería Industrial \\ Escuela Técnica Superior de Ingeniería y Diseño Industrial\\ Universidad Politécnica de Madrid\\Ronda de Valencia 3\\ 28012 Madrid España}
\email{daniel.fox@upm.es}

\begin{abstract}
The space of symplectic connections on a symplectic manifold is a symplectic affine space. M. Cahen and S. Gutt showed that the action of the group of Hamiltonian diffeomorphisms on this space is Hamiltonian and calculated the moment map. This is analogous to, but distinct from, the action of Hamiltonian diffeomorphisms on the space of compatible almost complex structures that motivates study of extremal Kähler metrics. In particular, moment constant connections are critical, where a symplectic connection is \textit{critical} if it is critical, with respect to arbitrary variations, for the $L^{2}$-norm of the Cahen-Gutt moment map. This occurs if and only if the Hamiltonian vector field generated by its moment map image is an infinitesimal automorphism of the symplectic connection. This paper develops the study of moment constant and critical symplectic connections, following, to the extent possible, the analogy with the similar, but different, setting of constant scalar curvature and extremal Kähler metrics. 

It focuses on the special context of critical symplectic connections on surfaces, for which general structural results are obtained, although some results about the higher-dimensional case are included as well. For surfaces, projectively flat and preferred symplectic connections are critical, and the relations between these and other related conditions are examined in detail. The relation between the Cahen-Gutt moment map and the Goldman moment map for projective structures is explained.
\end{abstract}

\maketitle
\setcounter{tocdepth}{1}  
{\footnotesize }
\tableofcontents

\section{Introduction and summary of results}\label{introsection}
A torsion-free affine connection on a $2n$-dimensional symplectic manifold $(M, \Om)$ is \textit{symplectic} if $\nabla_{i}\Om_{jk} = 0$ (in this article a symplectic connection is always torsion-free). The space $\symcon(M, \Om)$ of symplectic connections on $(M, \Om)$ is itself a symplectic affine space, and in \cite{Cahen-Gutt} (see also \cite{Bieliavsky-Cahen-Gutt-Rawnsley-Schwachhofer} or \cite{Gutt-remarks}), M. Cahen and S. Gutt showed that the action of the group $\Ham(M, \Om)$ of Hamiltonian diffeomorphisms on $\symcon(M, \Om)$ is Hamiltonian, with moment map $\nabla \to \K(\nabla) \in \cinf(M)$ (see\eqref{kdefined0} for its definition). 

This recalls the Hamiltonian action of $\Ham(M, \Om)$ on the space $\acom(M, \Om)$ of almost complex structures compatible with $\Om$, that plays an important role in the study of extremal and constant scalar curvature Kähler metrics. This analogy suggests studying the moment constant and critical symplectic connections, where a symplectic connection $\nabla \in \symcon(M, \Om)$ is \textit{critical} if it is critical with respect to arbitrary compactly supported variations for the $L^{2}$-norm 
\begin{align}\label{edefined}
\ex(\nabla) = \int_{M}\K(\nabla)^{2}\,\Omk{n}.
\end{align}
(In earlier versions of this article, such connections were called \textit{extremal symplectic}. The terminology has been changed because the word \textit{extremal} apparently misled some readers into thinking that critical symplectic connections were a special case of the usual almost Hermitian picture.) 

This article initiates the study of critical symplectic connections. It has three main subdivisions. The first, comprising Sections \ref{momentmapsection} and \ref{extremalsymplecticsection}, describes the general framework. Most of the results are valid in all dimensions $2n$. 

The second, constituting Sections \ref{preferredsection}-\ref{cohomsection}, focuses on the case $2n = 2$ of surfaces, for which there are relations with the theory of flat projective structures. In this setting general structural results are obtained. Basic lines of inquiry include: the construction of moment flat and critical connections; the description of the spaces of moment flat and critical connections on a fixed symplectic manifold; the analysis of to what extent critical connections are more general than moment constant connections; and, to what extent moment flat connections are more general than other classes of distinguished symplectic connections, such as preferred connections or projectively flat connections.

Although one reason for interest in critical symplectic connections is the applicability of these notions in the setting of symplectic manifolds admitting no Kähler metrics, the most studied symplectic connections are the Levi-Civita connections of Kähler metrics. The third part of the paper, comprising Sections \ref{criticalkahlersection}-\ref{higherdimensionalkahlersection} addresses this special case. The results support the idea that Kähler connections that are critical symplectic are uncommon, and must in some cases be equivalent to local products of locally symmetric spaces.  For example, it is shown that the Levi-Civita connection of the Ricci-flat Yau metric on a K3 surface is not critical symplectic.

The results reported here suggest that the higher-dimensional case $2n >2$ is interesting, and, although that context is not the focus of this article, Section \ref{concludingsection} contains brief discussion of analogues for symplectic connections of the Einstein condition and Bach tensor associated with a metric.

The remainder of the introduction describes the contents in detail. For technical reasons the ordering of the discussion in the introduction and the presentation of the corresponding material in the main text do not always coincide.

\subsection{}
Every symplectic manifold admits a symplectic connection, for if $\bnabla$ is any torsion-free affine connection then 
\begin{align}\label{lichid}
\nabla = \bnabla + \tfrac{2}{3}\Om^{kp}\bnabla_{(i}\Om_{j)p}
\end{align}
is symplectic. Here the abstract index conventions (see \cite{Penrose-Rindler} or \cite{Wald}) are used, and indices are raised and lowered (respecting horizontal position) using the symplectic form $\Om_{ij}$ and the dual antisymmetric bivector $\Om^{ij}$, subject to the conventions $X_{i} = X^{p}\Om_{pi}$ and $X^{i} = \Om^{ip}X_{p}$ (so that $\Om^{ip}\Om_{pj} = -\delta_{j}\,^{i}$, where $\delta_{j}\,^{i}$ is the canonical pairing between the tangent space and its dual). Enclosure of indices in square brackets (resp. parentheses) indicates complete antisymmetrization (resp. symmetrization) over the enclosed indices. Among the labels enclosed by delimiters, those further delimited by vertical bars $|\cdot|$ are omitted from the indicated (anti)symmetrization. For example, $2\nabla_{[i}h_{|j|k]} = \nabla_{i}h_{jk} - \nabla_{k}h_{ji}$. A label is in either \textit{up} position or \textit{down}, and a label appearing as both an up and a down index indicates the trace pairing (summation convention).

The affine space $\symcon(M, \Om)$ of symplectic connections on $(M, \Om)$ is modeled on the vector space $\Ga(S^{3}(\ctm))$ of completely symmetric covariant cubic tensors, for if $\nabla, \bnabla \in \symcon(M, \Om)$ and $\bnabla = \nabla + \Pi_{ij}\,^{k}$, then $\Pi_{[ij]}\,^{k} = 0$, because $\nabla$ and $\bnabla$ are torsion free, and, with $0 = \bnabla_{i}\Om_{jk} = -2\Pi_{i[jk]}$, this implies $\Pi_{ijk} = \Pi_{(ijk)}$. (For a smooth vector bundle $E \to M$, $\Ga(E)$ denotes the vector space of smooth sections of $E$, and $S^{k}(E)$ denotes the $k$th symmetric power of $E$.)

The symplectomorphism group $\Symplecto(M, \Om)$ of $(M, \Om)$ comprises compactly supported diffeomorphisms of $M$ that preserve $\Om$. 
The bilinear  $\Symplecto(M, \Om)$-invariant pairing of $\al, \be \in \Ga(S^{k}(\ctm))$ defined by 
\begin{align}\label{pairing}
\lb \al, \be \ra = \int_{M} \al_{i_{1}\dots i_{k}}\be^{i_{1}\dots i_{k}} \, \Omk{n} 
\end{align}
is graded symmetric in the sense that $\lb \al, \be \ra = (-1)^{|\al||\be|}\lb \be, \al \ra$, where $|\al| = k$. If a function is regarded as a $0$-tensor, then $\lb \dum, \dum \ra$ agrees with the $L^{2}$ inner product on $\cinf(M)$. Using a compatible almost complex structure it is straightforward to show that the pairing $\lb \al, \be \ra$ is (weakly) nondegenerate in the sense that if $ \lb \al, \be \ra = 0$ for all compactly supported $\be \in \Ga(S^{q}(\ctm))$ then $\al = 0$. Hence, the pairing \eqref{pairing} determines on $\symcon(M, \Om)$ a weakly nondegenerate antisymmetric bilinear form defined by $\sOm_{\nabla}(\al, \be) =  \lb \al, \be \ra$ for $\al, \be \in T_{\nabla}\symcon(M, \Om) = \Ga(S^{3}(\ctm))$. The translation invariance, $\sOm_{\nabla + \Pi} = \sOm_{\nabla}$, for $\Pi \in T_{\nabla}\symcon(M, \Om)$ means $\sOm$ is parallel, so closed, and so $\sOm$ is a symplectic form. 

The diffeomorphism group $\diff(M)$ of $M$ acts by pullback on the affine space $\affcon(M)$ of torsion-free affine connections on $M$; for $\phi \in \diff(M)$, $\phi^{\ast}(\nabla)_{X}Y = T\phi^{-1}(\nabla_{T\phi(X)}T\phi(Y))$. The induced action of $\Symplecto(M, \Om)$ on $\affcon(M)$ preserves the subspace $\symcon(M, \Om)$ and the form $\sOm$. This suggests that interesting classes of symplectic connections can be identified in terms of the symplectic geometry of the symplectic affine space $(\symcon(M, \Om), \sOm)$ and the action on it of $\Symplecto(M, \Om)$ and its subgroup $\Ham(M, \Om)$ consisting of compactly supported Hamiltonian diffeomorphisms; these are the elements of the path connected component of the identity $\Symplecto(M, \Om)_{0} \subset \Symplecto(M, \Om)$ that can be realized as the time one flow of a normalized time-dependent Hamiltonian on $M \times [-1, 1]$, where \textit{normalized} means mean-zero or compactly supported as $M$ is compact or noncompact. 

In general, interesting classes of connections are defined in terms of curvature and as critical points of functionals constructed from the curvature. In particular, it is natural to consider functionals on $\symcon(M, \Om)$ invariant with respect to the action of some subgroup of $\Symplecto(M, \Om)$. The most basic classes of functionals are those quadratic in the curvature, or those determined by some special property of the group action, for example that it be Hamiltonian. As will be explained, such considerations focus attention on the classes of preferred, moment constant, and critical symplectic connections. In order to define these classes, some more definitions are needed. 

The curvature $R_{ijk}\,^{l}$ of $\nabla \in \affcon(M)$ is defined by $2\nabla_{[i}\nabla_{j]}X^{k} = R_{ijp}\,^{k}X^{p}$ for $X \in \Ga(TM)$. The Ricci curvature is $R_{ij} = R_{pij}\,^{p}$. (Sometimes, for readability, there will be written $\ric$ or $\ric_{ij}$ instead of $R_{ij}$.) The basic properties of the curvature of a symplectic connection are treated in many sources; good starting points include \cite{Bieliavsky-Cahen-Gutt-Rawnsley-Schwachhofer}, \cite{Fedosov}, \cite{Gelfand-Retakh-Shubin}, and \cite{Vaisman}. 
If $\nabla \in \symcon(M, \Om)$, then, since $\nabla$ preserves $\Om$, it has symmetric Ricci tensor, for $2R_{[ij]} = - R_{ijp}\,^{p} = 0$. By the Ricci identity, $0 = 2\nabla_{[i}\nabla_{j]}\Om_{kl} = -2R_{ij[kl]}$, where $R_{ijkl} = R_{ijk}\,^{p}\Om_{pl}$. With the Bianchi identity this yields $R_{p}\,^{p}\,_{ij} = -2R_{ip}\,^{p}\,_{j} = 2R_{pij}\,^{p} = 2R_{ij}$; it follows that every nontrivial trace of $R_{ijkl}$ is a constant multiple of $R_{ij}$. From this observation together with the differential Bianchi identity it is straightforward to obtain $\nabla^{p}R_{pijk}= \nabla_{i}R_{jk}$.

The space of tensors on a vector space $\ste$ having curvature tensor symmetries can be decomposed into irreducibles with respect to the action of a subgroup of $GL(\ste)$. The decomposition with respect to an orthogonal group yields the usual conformal Weyl tensor, the traceless Ricci curvature, and the scalar curvature. The analogous decomposition with respect to the linear symplectic group yields only two irreducibles, one, $\weylmod(\std, \Om) = \{\al_{ijkl} \in \tensor^{4}\std: \al_{[ij]kl} = \al_{ijkl}, \al_{ij[kl]} = 0, \al_{[ijk]l} = 0, \al_{pij}\,^{p}  = 0\}$, comprising completely trace-free tensors with the symmetries of a symplectic curvature tensor, and one isomorphic to $S^{2}(\std)$, corresponding to the Ricci curvature. 
The part of $R_{ijkl}$ that vanishes when contracted with $\Om^{ij}$ on any pair of indices is the \textit{symplectic Weyl tensor}
\begin{align}\label{symplecticweyltensor}
W_{ijkl} = R_{ijkl} - \tfrac{1}{n+1}\left(\Om_{i(k}R_{l)j} - \Om_{j(k}R_{l)i} + \Om_{ij}R_{kl} \right) \in \Ga(\weylmod(\ctm, \Om)).
\end{align}

Particularly in the case $2n = 2$, an important role is played by the \textit{curvature one-form} $\rf_{i} = \rf(\nabla)_{i} = 2\nabla^{p}R_{ip}$ of $\nabla \in \symcon(M, \Om)$.
Absent some auxiliary metric structure, there is no reasonable analogue for symplectic connections of the scalar curvature of a metric, but the symplectic dual of the curvature one-form is a reasonable analogue of the Hamiltonian vector field generated by the scalar curvature of a Kähler metric. 
An almost complex structure $J_{i}\,^{j}$ is \textit{compatible} with $\Om_{ij}$ if $g_{ij} = -J_{i}\,^{p}\Om_{pj} = -J_{ij}$ is symmetric and positive definite. In this case the Riemannian volume element $d\vol_{g}$ equals $\Omk{n}$. The triple $(g, J, \Om)$ is \textit{Kähler} if $J$ is moreover integrable; to specify it, it suffices to specify any two of $g$, $J$, and $\Om$. In particular, the Levi-Civita connection $D$ of $g$ is symplectic. For a Kähler structure $(\Om, g, J)$ with Levi-Civita connection $D$ and scalar curvature $\sR_{g} = g^{ij}R_{ij}$, by the traced differential Bianchi identity there holds $2g^{pq}D_{p}R_{iq} = D_{i}\sR_{g}$, and so 
\begin{align}\label{rfkahler}
\rf_{i} = -2\Omega^{pq}D_{p}R_{iq} = -2g^{pb}D_{p}(J_{b}\,^{q}R_{iq}) = 2g^{pb}D_{p}(J_{i}\,^{q}R_{bq}) = J_{i}\,^{q}D_{q}\sR_{g} = -g_{iq}\hm_{\sR_{g}}^{q},
\end{align}
where $\hm_{f}^{i} = -df^{i} = \Om^{pi}df_{p}$ is the Hamiltonian vector field associated with $f \in \cinf(M)$.
(For a nondegenerate covariant symmetric two-tensor $g_{ij}$, $g^{ij}$ always denotes the inverse symmetric bivector, and not the tensor obtained by raising indices with $\Om^{ij}$. When $g_{ij}$ is determined by a compatible complex structure, the two possible meanings for $g^{ij}$ coincide.)
Equivalently, the vector field $\rf^{\sharp\, i}$ metrically dual to $\rf_{i}$ is the negative of the Hamiltonian vector field generated by the scalar curvature: $\rf^{\sharp \,i} = g^{ip}\rf_{p} = -\hm_{\sR_{g}}^{i}$. In particular, a Kähler metric has constant scalar curvature if and only if $\rf = 0$, and is extremal if and only if $\rf^{\sharp\, i}$ is a real holomorphic vector field. For a general symplectic connection the curvature one-form $\rf_{i}$ serves as a substitute for the rotated differential of the scalar curvature, and the vanishing of $\rf$ is a reasonable substitute for the condition of constant scalar curvature. 

The curvature one-form also admits an interpretation as a multiple of the trace of the projective Cotton tensor of the projective structure generated by $\nabla$.
Two torsion-free affine connections are \textit{projectively equivalent} if the image of each geodesic of one is contained in the image of a geodesic of the other. This is the case if and only if their difference tensor has the form $2 \ga_{(i}\delta_{j)}\,^{k}$ for some one-form $\ga_{i}$. A \textit{projective structure} $\en$ is a projective equivalence class of torsion-free affine connections. Let $P_{ij} = \tfrac{1}{1-2n}R_{(ij)} - \tfrac{1}{2n+1}R_{[ij]}$. The \textit{projective Weyl tensor} $B_{ijk}\,^{l} = R_{ijk}\,^{l} + 2\delta_{[i}\,^{l}P_{j]k} - 2 \delta_{k}\,^{l}P_{[ij]} = R_{ijk}\,^{l} + \tfrac{2}{1-2n}\delta_{[i}\,^{l}R_{j]k}$, depends only on the projective equivalence class of $\nabla$. 
The \textit{projective Cotton tensor} is $C_{ijk} = 2\nabla_{[i}P_{j]k} = \tfrac{2}{1-2n}\nabla_{[i}R_{j]k} - \tfrac{2}{4n^{2} - 1}\nabla_{k}R_{[ij]}$. When $2n > 2$, by the differential Bianchi identity, $C_{ijk} = 2(1-n)\nabla_{p}B_{ijk}\,^{p}$. When $2n = 2$, the projective Weyl tensor vanishes and $C_{ijk} = -2\nabla_{[i}R_{j]k} - \tfrac{2}{3}\nabla_{k}R_{[ij]}$ does not depend on the choice of representative $\nabla \in \en$. The projective structure $\en$ is \textit{projectively flat} if $B_{ijk}\,^{l} = 0$ and $C_{ijk} = 0$; this is equivalent to the existence of an atlas of charts with transition functions in $PGL(2n, \rea)$. For $\nabla \in \symcon(M, \Om)$, $(2n-1)B_{p}\,^{p}\,_{ij} = 4(n-1)R_{ij}$, so if $2n > 2$ a projectively flat symplectic connection is Ricci flat, so flat. 
Skew-symmetrizing the differential Bianchi identity $\nabla^{p}R_{pijk} = \nabla_{i}R_{jk}$ yields $(2n-1)C_{ijk} = \nabla^{p}R_{ijkp}$, while differentiating \eqref{symplecticweyl} yields $2(n+1)\nabla^{p}R_{ijkp} = 2(n+1)\nabla^{p}W_{ijkp} + (2n-1)C_{ijk} + \Om_{ij}\rf_{k} - \Om_{k[i}\rf_{j]}$, so that
\begin{align}\label{gi52n}
 (2n+1)(2n-1)C_{ijk} = 2(n+1)\nabla^{p}W_{ijkp} + \rf_{k}\Om_{ij} - \Om_{k[i}\rf_{j]}.
\end{align}
Tracing \eqref{gi52n} in $ij$ yields $(2n-1)C_{p}\,^{p}\,_{i} = \rf_{i}$. 

The simplest $\Symplecto(M, \Om)$-invariant functional on the symplectic affine space $(\symcon(M, \Om), \sOm)$ associates to $\nabla \in \symcon(M, \Om)$ the integral $\int_{M}R_{i}\,^{j}R_{j}\,^{i}\,\Omk{n}$ of the square of the Ricci endomorphism. Its critical points are characterized by the equation $\nabla_{(i}R_{jk)} = 0$ (this follows from \eqref{ricvar}) and are called \textit{preferred}. For surfaces preferred symplectic connections were introduced by F. Bourgeois and M. Cahen  in \cite{Bourgeois-Cahen-preferred} and in \cite{Bourgeois-Cahen}, where they were called \textit{solutions to the field equations}. 
One reason for interest in preferred symplectic connections, in addition to their variational character, is that, since decomposing $\nabla_{i}R_{jk}$ by symmetries yields $3\nabla_{i}R_{jk} = 3\nabla_{(i}R_{jk)} -2(2n-1)C_{i(jk)}$, the only possible extensions of the condition that a symplectic connection have parallel Ricci tensor are that it have vanishing projective Cotton tensor or that it be preferred.

Another reason for interest in the preferred symplectic connections is that any functional on $\symcon(M, \Om)$ that is constructed by integrating expressions quadratic in the curvature of $\nabla \in \symcon(M, \Om)$ has as its critical points the preferred symplectic connections. If the curvature $R_{ijk}\,^{l}$ of $\nabla \in \symcon(M, \Om)$ is regarded as an adjoint bundle valued two-form $\curv_{ij}$, the $4$-form $\tr( \curv \wedge \curv)_{ijkl}$ equals $6R_{[ij|p|}\,^{q}R_{kl]q}\,^{p}$. The first Pontryagin class $\pon_{1}(M)$ of $M$ is represented by the \textit{first Pontryagin form}
\begin{align}\label{pon}
\pon_{1\,ijkl} = \pon_{1}(\nabla)_{ijkl} = -\tfrac{1}{8\pi^{2}}\tr (\curv \wedge \curv)_{ijkl} =  -\tfrac{3}{4\pi^{2}} R_{[ij|p|}\,^{q}R_{kl]q}\,^{p}  =  -\tfrac{3}{4\pi^{2}} B_{[ij|p|}\,^{q}B_{kl]q}\,^{p}
\end{align}
of $\nabla \in \symcon(M, \Om)$. Since 
\begin{align}\label{ponom}
8\pi^{2}\pon_{1}(\nabla)\wedge \Omk{(n-2)} = \pi^{2}\pon_{1}(\nabla)_{p}\,^{p}\,_{q}\,^{q}\Omk{n} = (R^{ij}R_{ij} - \tfrac{1}{2}R^{ijkl}R_{ijkl})\Omk{n}
\end{align}
(see \eqref{dlefpon} and \eqref{pontryagin}), the integral 
\begin{align}\label{intpon}
\int_{M}(R^{ij}R_{ij} - \tfrac{1}{2}R^{ijkl}R_{ijkl})\,\Omk{n} = 8\pi^{2}\int_{M} \pon_{1}\wedge \Omk{(n-2)} = 8\pi^{2}(n-2)!\lb [\pon_{1}] \cup [\Om]^{n-2}, [M]\ra,
\end{align}
depends only on $\pon_{1}(M)$ and the cohomology class $[\Om]$. Since all functionals on $\symcon(M, \Om)$ quadratic in the curvature of $\nabla$ must be integrals of linear combinations of $R^{ij}R_{ij}$ and $R^{ijkl}R_{ijkl}$, it follows that all such functionals have the same critical points, which are the preferred symplectic connections. This is similar to, although simpler than, the situation for quadratic curvature functionals on the space of metrics in a fixed Kähler class as discussed in \cite{Calabi-extremal}.

An action of a Lie group $G$ on a symplectic manifold $(M, \Om)$ is \textit{symplectic} if $G$ acts by symplectic diffeomorphisms.  The Lie algebra homomorphism from the Lie algebra $\g$ of $G$ to $\symplecto(M, \Om)$ defined by $x \to \X^{x}_{p} = \tfrac{d}{dt}\big|_{t = 0}\exp(-tx)\cdot p$ is \textit{Hamiltonian} if there is a map $\mu:M \to \g^{\ast}$, equivariant with respect to the action of $G$ on $M$ and the coadjoint action of $G$ on $\g^{\ast}$, such that for each $x \in \g$, the Hamiltonian vector field $\hm_{\mu(x)}$ equals $\X^{x}$; $\mu$ is called a \textit{moment map}.

It is natural to ask if the action of $\Symplecto(M, \Om)$ or its subgroup $\Ham(M, \Om)$ on $\symcon(M, \Om)$ is Hamiltonian.
The Lie algebra $\symplecto(M, \Om)$ of $\Symplecto(M, \Om)$ comprises compactly supported vector fields $X$ such that $\lie_{X}\Om = 0$; equivalently the one-form $X^{\sflat} = \Om(X, \,\cdot\,)$ is closed.
Since $\hm_{\{u, v\}} = [\hm_{u}, \hm_{v}]$ for the Poisson bracket $\{u, v\} = \hm_{u}^{i}\hm_{v}^{j}\Om_{ij} = -du^{p}dv_{p} = dv(\hm_{u})$ of $u, v \in \cinf(M)$, the compactly supported Hamiltonian vector fields constitute a subalgebra $\ham(M, \Om)\subset \symplecto(M, \Om)$. By a theorem of A.~Banyaga, the infinitesimal generator of a flow by Hamiltonian diffeomorphisms is a Hamiltonian vector field, so  $\ham(M, \Om)$ can be regarded as the Lie algebra of $\Ham(M, \Om)$. Because $\{u, v\}\Omk{n} = du \wedge dv \wedge \Omk{(n-1)} = d(udv\wedge \Omk{(n-1)})$ is always exact, if $M$ is compact, $\int_{M}\{u, v\} \,\Omk{n} = 0$ for $u, v \in \cinf(M)$, so the subspace $\cinf_{0}(M) \subset \cinf(M)$ of mean zero functions is a Lie subalgebra of $\cinf(M)$, isomorphic to $\ham(M, \Om)$ via the map $u \to \hm_{u}$. If $M$ is noncompact, then $\ham(M, \Om)$ is isomorphic with the Lie algebra $\cinf_{c}(M)$ of compactly supported smooth functions. 
It follows that, for any $u\in \cinf(M)$, integration against $u\Omk{n}$ defines a linear functional on $\ham(M, \Om)$, so $\cinf(M)$ can be identified with a subspace of $\ham(M, \Om)^{\ast}$, however $\ham(M, \Om)^{\ast}$ is topologized.
\begin{theorem}[M. Cahen and S. Gutt \cite{Cahen-Gutt}; see also \cite{Bieliavsky-Cahen-Gutt-Rawnsley-Schwachhofer} or \cite{Gutt-remarks}]\label{momentmaptheorem}
On a symplectic manifold $(M, \Om)$, the map $\K:\symcon(M, \Om) \to \ham(M, \Om)^{\ast}$ defined by 
\begin{align}\label{kdefined0}
\begin{split}
\K(\nabla) &= \nabla^{i}\nabla^{j}R_{ij} - \tfrac{1}{2}R^{ij}R_{ij}+ \tfrac{1}{4}R^{ijkl}R_{ijkl} = \tfrac{1}{2}\nabla^{i}\rf_{i} - \tfrac{\pi^{2}}{2}\pon_{1}(\nabla)_{p}\,^{p}\,_{q}\,^{q}
\end{split}
\end{align}
is a moment map for the action of $\Ham(M, \Om)$ on the symplectic affine space $(\symcon(M, \Om), \sOm)$, equivariant with respect to the natural actions of the group $\Symplecto(M, \Om)$ of compactly supported symplectomorphisms of $(M, \Om)$. 
\end{theorem}
Theorem \ref{momentmaptheorem} is reproved in section \ref{momentmapsection}. Although here there is no real novelty, the proof given is structured differently than that given in \cite{Bieliavsky-Cahen-Gutt-Rawnsley-Schwachhofer} or \cite{Gutt-remarks}, in a manner intended to parallel the proof given by S. Donaldson in \cite{Donaldson-remarks} that the Hermitian scalar curvature of the associated Hermitian metric is a moment map for the action of $\Ham(M, \Om)$ on the space $\acom(M, \Om)$ of almost complex structures compatible with the symplectic form $\Om$ (in the Kähler case this was obtained also by A. Fujiki in \cite{Fujiki}). This approach has the side benefit of simplifying subsequent arguments, for example the computation of the second variation of $\emf$.

It is reasonable to ask if the action of $\Symplecto(M, \Om)$ on $\symcon(M, \Om)$ is Hamiltonian. Although there is an obstruction when $2n > 2$, for surfaces ($2n = 2$), the answer is affirmative. To state the claim precisely it is necessary to specify where the moment map takes values. it is claimed that the dual space $\symplecto(M, \Om)^{\ast}$ can be identified with the space $\ext^{1}(M)/\dad\ext^{2}(M)$ of smooth $1$-forms modulo coexact $2$-forms, where the \textit{symplectic codifferential} $\dad\al$ of a $k$-form $\al$, defined by $\dad \al_{i_{1}\dots i_{k-1}} = -\nabla^{p}\al_{pi_{1}\dots i_{k-1}}$, does not depend on the choice of $\nabla \in \symcon(M, \Om)$ and satisfies $\dad\circ \dad = 0$. Because $dX^{\sflat} = 0$, integration determines a pairing $\symplecto(M, \Om) \times \ext^{2n-1}(M)/d\ext^{2n-2}(M) \to \rea$ defined by $(X, [\al]) \to \int_{M}X^{\sflat} \wedge \al$, for any $\al \in [\al]$. Because the symplectic $\star$-operator, defined on a $k$-form $\be$ by $k!(\star \be)_{i_{1}, \dots, i_{2n-k}} = \be^{j_{1}\dots j_{k}}(\Omk{n})_{j_{1}\dots, j_{k} i_{1} \dots i_{2n-k}}$, satisfies $\star \dad = (-1)^{k+1}d \star$, it induces, via $[\be] \to [\star \be]$, an identification of $\ext^{2n-1}(M)/d\ext^{2n-2}(M)$ with the space $\ext^{1}(M)/\dad\ext^{2}(M)$ of one-forms modulo coexact one-forms. Note that, because $\dad(f\Omk{n}) = -df \wedge \Omk{(n-1)} = d(-f\Omk{(n-1)})$ and $\dad(f\Om) = -df$, there hold $\dad\ext^{2n}(M) \subset d\ext^{2n-2}(M)$ and $d\ext^{0}(M) \subset \dad \ext^{2}(M)$. The dual of the space of compactly supported divergence-free vector fields is naturally identified with $\ext^{1}(M)/d\ext^{0}(M)$ and so the dual of its subspace comprising symplectic vector fields is identified with the possibly smaller quotient $\ext^{1}(M)/\dad\ext^{2}(M)$. Alternatively, $\symplecto(M, \Om)^{\ast}$ can be identified directly with $\ext^{1}(M)/\dad\ext^{2}(M)$ using the pairing \eqref{pairing}, since this pairing satisfies $\al \wedge \star \be = \lb \al, \be \ra \Omk{n}$ for $\al, \be \in \ext^{k}(M)$. Consequently, a moment map for $\Symplecto(M, \Om)$ must take values in $\ext^{1}(M)/\dad\ext^{2}(M) \simeq \ext^{2n-1}(M)/d\ext^{2n-2}(M)$. Moreover, its codifferential should be $\K(\nabla)$, so that when consideration is restricted to the action of $\ham(M, \Om)$, it recovers $\K(\nabla)$. This requirement makes sense because the codifferential of a coexact one-form is zero. 

\begin{theorem}\label{rhomomentmaptheorem}
If $M$ is a surface ($2n = 2$), the map $\symcon(M, \Om) \to  \ext^{1}(M)/\dad\ext^{2}(M)$ sending $\nabla \in \symcon(M, \Om)$ to the equivalence class $-\tfrac{1}{2}[\rf]$ is a moment map for the action of the group $\Symplecto(M, \Om)$ of symplectomorphisms on $\symcon(M, \Om)$, equivariant with respect to the natural actions of $\Symplecto(M, \Om)$. 
\end{theorem}

In dimension $2n > 2$ there is a curvature term that obstructs $-\tfrac{1}{2}[\rf]$ being a moment map. It is not clear whether by modifying $\rf$ this obstruction can be overcome or even if the action of $\Symplecto(M, \Om)$ is Hamiltonian. This issue is discussed in section \ref{pontryaginsection}.

A symplectic connection $\nabla$ is \textit{moment constant} or \textit{moment flat} if $\K(\nabla)$ is constant or zero. Since the curvature tensor of a locally symmetric $\nabla \in \symcon(M, \Om)$ is parallel, any scalar quantity formed from it and $\nabla$ is constant, and so, in this case, $\nabla$ is moment constant. More generally, the same is true if the group of automorphisms of $\nabla \in \symcon(M, \Om)$ acts on $M$ transitively by symplectomorphisms. For example, a reductive homogeneous space $G/H$ that is \textit{symplectic}, meaning it admits a $G$-invariant symplectic form, carries a canonical symplectic connection, namely the connection $\nabla$ obtained by applying \eqref{lichid} with $\bnabla$ being Nomizu's canonical connection of the first kind (see Theorem $10.1$ of \cite{Nomizu-invariant}); this $\nabla$ is moment constant. There are examples of solvable symplectic Lie groups for which this canonical symplectic connection is neither preferred nor has vanishing symplectic Weyl tensor; the simplest are constructed on the group of affine transformations of the complex line in \cite{Fox-sympsec}. 

When $M$ is compact the integral of $\K(\nabla)$ depends only on the cohomology class of $\Om^{n-2}$ and the first Pontryagin class, for, by \eqref{intpon},
\begin{align}\label{intk}
\int_{M}\K(\nabla)\,\Omk{n} = -4\pi^{2}\int_{M} \pon_{1}\wedge \Omk{(n-2)} = -4\pi^{2}(n-2)!\lb [\pon_{1}] \cup [\Om]^{n-2}, [M]\ra.
\end{align}
For example, when $2n = 4$, $\int_{M}\K(\nabla)\,\Om_{2} = -12\pi^{2}\sig(M)$, where, by the Hirzebruch signature theorem, $\sig(M) = \tfrac{1}{3}\pon_{1}(M)$ is the signature of $M$. By \eqref{intk}, when $\nabla$ is moment constant, the constant value of $\K(\nabla)$ is expressible in terms of the first Pontryagin class and the symplectic volume of $M$.

The critical points with respect to variations within a fixed Kähler class of the Calabi functional, the squared $L^{2}$-norm of the scalar curvature, are the extremal Kähler metrics first studied in \cite{Calabi-extremal}. 
Here $\K$ is regarded as an analogue for symplectic connections of the Hermitian scalar curvature of an almost Kähler metric. This analogy suggests studying connections $\nabla \in \symcon(M, \Om)$, that are critical for $\emf(\nabla)$ with respect to arbitrary compactly supported variations. 
 
\begin{definition}
On a symplectic manifold $(M, \nabla)$, a symplectic connection $\nabla \in \symcon(M, \Om)$ is \textbf{critical} if it is a critical point, with respect to arbitrary compactly supported variations, of the functional $\ex:\symcon(M, \Om) \to \rea$ defined by \eqref{edefined}.
\end{definition}
By the $\Symplecto(M, \Om)$-equivariance of $\K$, $\ex$ is constant along $\Symplecto(M, \Om)$ orbits in $\symcon(M, \Om)$, so can be viewed as a functional on the quotient $\symcon(M, \Om)/\Symplecto(M, \Om)$. 
In the analogy with extremal Kähler metrics, $\ex$ plays the role of the Calabi functional. 

Although the analogy is incomplete, the results obtained suggest that it provides a reasonable guide to expectations. For example, a Kähler metric is extremal if and only if the $(1,0)$ part of the metric gradient of its scalar curvature is holomorphic. Analogously:
\begin{theorem}\label{criticaltheorem}
A symplectic connection $\nabla \in \symcon(M, \Om)$ is critical if and only if the Hamiltonian vector field $\hm_{\K(\nabla)}$ generated by $\K(\nabla)$ is an infinitesimal automorphism of $\nabla$. 
\end{theorem}

In particular moment constant symplectic connections are critical. 
On a compact manifold, constant scalar curvature metrics are absolute minimizers of the Calabi functional. 
In fact, by Theorem $1.5$ of \cite{Chen-spacekahleriii}, every extremal Kähler metric is an absolute minimizer of the Calabi functional. It is not clear to what extent the analogous statements for critical symplectic connections should be expected to be valid because the space over which $\emf$ is varied is larger than the fixed Kähler class over which the Calabi functional is varied. 

The second variation of $\ex$ is computed in Lemma \ref{secondvariation}, and while the result formally resembles the second variation in the Kähler setting, it yields useful information less readily, because ellipticity of the differential operators involved is not immediate as it is in the metric setting.  Nonetheless, by Corollary \ref{kconstantcorollary}, the second variation is nonnegative at a moment constant connection. The consequence, that, on a compact manifold, moment constant symplectic connections are absolute minimizers is in any case obvious since $\emf$ is nonnegative. 

A complete development of the analogy between extremal Kähler metrics and critical symplectic connections would include a class (or classes) of symplectic connections analogous to Kähler Einstein metrics. In this regard, there are available various possible notions, in which the first Pontryagin class plays role like that of the first Chern class, but discussion here is limited to some brief remarks relegated to the final Sections \ref{momentsection} and \ref{symplecticbachsection}. 

\subsection{Surfaces equipped with a parallel volume form}
The focus of sections \ref{preferredsection}-\ref{cohomsection} is the consideration of the simplest nontrivial context for symplectic connections, namely connections preserving a volume form on a surface. The results obtained are sufficiently rich to suggest that the higher-dimensional case merits further study.
On the other hand, the case of surfaces is special in several respects. In particular, for a surface there are close relations between moment flat symplectic connections and projectively flat connections, fundamentally because of the identification $\sll(2, \rea) \simeq \sp(1, \rea)$; these are detailed in Section \ref{cohomsection} and discussed further below. Since, on a surface $M$, $2\al_{ij} = \al_{p}\,^{p}\Om_{ij}$ for any $\al_{ij} \in \Ga(\ext^{2}(\ctm))$, and $W_{ijkl} = 0$, \eqref{symplecticweyltensor} yields
\begin{equation}\label{symplecticweyl}
2R_{ijkl} = \Om_{i(k}R_{l)j} - \Om_{j(k}R_{l)i} + \Om_{ij}R_{kl} = 2\Om_{ij}R_{kl},
\end{equation}
showing that the curvature of $\nabla \in \symcon(M, \Om)$ is completely determined by the Ricci curvature. 
When $2n = 2$, \eqref{gi52n} specializes to
\begin{align}\label{gi5}
 2C_{ijk} =2\nabla^{p}R_{ijkp} = \rf_{k}\Om_{ij}.
\end{align}
Consequently, a projective structure $\en$ on an orientable surface is projectively flat if and only if for some, and hence every, volume form $\Om_{ij}$, there vanishes the curvature one-form $\rf_{i}$ of the unique representative $\nabla \in \en$ for which $\Om_{ij}$ is parallel. 

Since on a surface the first Pontryagin form vanishes, relations between the geometry of a critical symplectic connection and the topology of the surface are less direct than in higher dimensions.  The part of \eqref{kdefined0} quadratic in the curvature is a constant multiple of the contraction of $\Om \wedge \Om$ with the first Pontryagin form of the connection $\nabla$ and so vanishes identically on a symplectic $2$-manifold. Hence, on a surface, $2\K(\nabla) = \nabla^{p}\rf_{p}$. Alternatively, \eqref{symplecticweyl} yields the identities $2R^{pq}R_{pijq} = -2R_{ip}R_{j}\,^{p} = - R^{pq}R_{pq}\Om_{ij}  = \tfrac{1}{2}R_{i}\,^{abc}R_{jabc}$, and $R^{ijkl}R_{ijkl} = 2R^{ij}R_{ij}$, and in \eqref{kdefined0} these yield $2K(\nabla) = \nabla^{p}\rf_{p}$. 
Since $\K(\nabla)$ is a divergence, when $M$ is compact $\int_{M}\K(\nabla)\,\Om = 0$. Hence on a compact symplectic $2$-manifold, $\K(\nabla)$ is constant if and only if it is $0$. 

\begin{lemma}\label{projflatlemma} 
On a surface $(2n = 2)$, a symplectic connection is moment flat if and only if the vector field symplectically dual to the curvature one-form is symplectic. In particular, on a surface, a projectively flat symplectic connection is moment flat.
\end{lemma}
\begin{proof}
Since $d\rf$ must be a multiple of $\Om$, 
\begin{align}
\label{cdiv}
&d\rf_{ij} = 2\nabla_{[i}\rf_{j]} = \nabla_{p}\rf^{p}\Om_{ij} = -2\K(\nabla)\Om_{ij},
\end{align}
so that $\K(\nabla) = 0$ if and only if $\rf$ is closed, or, equivalently, the vector field $\rf^{i}$ is symplectic. By \eqref{gi5}, if $\nabla$ is projectively flat, then $\rf_{i}$ vanishes, so by \eqref{cdiv}, $\nabla$ is moment flat. 
\end{proof}
 Together \eqref{gi5} and \eqref{cdiv} show that $\K$ can be viewed as a sort of derivative of the projective Cotton tensor. 

Purely local computations yield Theorem \ref{preferredtheorem}.
\begin{theorem}\label{preferredtheorem}
On a symplectic $2$-manifold a preferred symplectic connection is critical.
\end{theorem}
Since there are preferred symplectic connections that are not moment flat, Theorem \ref{preferredtheorem} demonstrates there are critical symplectic connections that are not moment constant.  By Lemma \ref{projflatlemma}, projectively flat implies moment flat and, by definition and Theorem \ref{preferredtheorem}, moment constant and preferred imply critical. Basic issues are determining under what conditions these implications are reversible, and, when they are not, characterizing the failure:
\begin{enumerate}
\item Determine when there exist and describe moment flat symplectic connections that are not projectively flat.
\item Determine when there exist and describe critical symplectic connections that are not moment constant.
\item Determine when there exist and describe critical symplectic connections that are not preferred.
\end{enumerate}
As will be explained, in each case there are conditions under which the implication is not reversible, while, on the other hand, there are conditions on surfaces that guarantee that a critical connection is moment constant or projectively flat.

The essential content of Theorem \ref{preferredtheorem} (the criticality of preferred connections), namely that the Hamiltonian flow generated by $\K$ preserves $\nabla$, is contained in Proposition $6.1$ of \cite{Bourgeois-Cahen}, and its proof there uses (quantities equivalent to) $\K$, $\rf$, and identities relating them and their differentials that continue to be valid in the more general setting of critical symplectic connections considered here. This motivated much of section \ref{localstructuresection}, where the key technical results that facilitate the description of a critical symplectic connection $\nabla$ are recorded. 

The two key technical observations underlying structural results for critical symplectic connections that are not moment flat are the following.
First, for a critical $\nabla \in \symcon(M, \Om)$, since $\hm_{\K}$ preserves $\nabla$, it preserves the curvature tensors associated with $\nabla$, and, from $\lie_{\hm_{\K}}\rf = 0$, it follows that there is a constant $\vc$ such that $\K^{2} +  \rf(\hm_{\K}) = \vc$. Using this identity it can be concluded that when $\K$ is not constant each connected component of its critical set is an isolated point or an isolated closed $\nabla$-geodesic. Second, the zero set of $\vc - \K^{2}$ is a union of isolated points and geodesic circles, and on its complement $\bar{M}$, the one-form $\si = (\vc - \K^{2})^{-1}\rf$ is closed and $\Om = d\K \wedge \si$. This means $\K$ and a local primitive of $\si$ constitute canonical action-angle coordinates on $\bar{M}$; precisely the action coordinate $\K$ is a moment map for the action of the flow of the symplectically dual vector field $\si^{\ssharp}$. These observations and related technical claims are detailed in Lemma \ref{extremalstructurelemma}. In the special setting of preferred connections, some form of both these observations plays a key role in \cite{Bourgeois-Cahen}, and this signaled their importance in the more general context of critical symplectic connections. They are useful both for constructing examples and for proving that under certain conditions a critical symplectic connection must be moment constant.
 
By Theorem $7.2$ of \cite{Bourgeois-Cahen} a preferred symplectic connection on a compact surface has parallel Ricci tensor; in particular it has $\rf = 0$ so is projectively flat. Consequently, Theorem \ref{preferredtheorem} yields no interesting examples of critical symplectic connections on compact surfaces. By Theorem $3.1$ of Calabi's \cite{Calabi-extremal}, on a compact surface an extremal Kähler metric has constant scalar curvature. These results both suggest that on a compact surface any critical symplectic connection must be moment flat. Theorem \ref{noextremaltheorem} mostly confirms this expectation.
\begin{theorem}\label{noextremaltheorem}
On a compact symplectic $2$-manifold $(M, \Om)$ of genus at least two, any critical symplectic connection $\nabla \in \symcon(M, \Om)$ is moment flat.
\end{theorem}
The canonical action-angle coordinates on $\bar{M}$ are used to show (see Lemma \ref{compactextremalstructurelemma}) that, when $M$ is compact, each connected component of $\bar{M}$ carries a complete flat Kähler structure preserved by $\hm_{\K}$. This is enough to conclude that the connected components of $\bar{M}$ are diffeomorphic to cylinders and so $M$ is obtained by gluing together disks and cylinders, so must have nonnegative Euler characteristic. This argument leaves open the possibility that on the sphere or torus there are critical symplectic connections that are not moment flat. 

 In trying to construct examples of symplectic connections that are critical but not moment flat, by Lemma \ref{extremalstructurelemma}, it can be assumed that $\Om$ has the standard Darboux form $dx \wedge dy$, that $\K(\nabla)$ equals $x + a$ for some constant $a$, and hence that $\rf = (\vc - (x + a)^{2})^{-1}dy$. The connection $\nabla$ can be written as $\pr + \Pi$ where $\pr$ is the standard flat affine connection preserving $dx$ and $dy$, and $\rf$, $\K$, and the equations $\lie_{\hm_{\K}}\nabla = 0$ can be computed explicitly in terms of the components of $\Pi$. 
The results of this approach are stated in Lemmas \ref{examplelemma} and \ref{extlemma}. 

 The expressions for $\K$ and $\hm_{\K}$ in Darboux coordinates recounted in Lemma \ref{examplelemma} are sufficiently complicated that a complete analysis of them has not been made, but with various simplifying assumptions they yield tractable equations that can be solved to yield several classes of examples. The first such simplifying assumption is to seek a critical $\nabla$ that satisfies additionally $\nabla_{(i}\rf_{j)} = 0$. This supposition determines a family of critical symplectic connections on $\rea^{2}$, that are not moment constant, and that are geodesically complete for certain choices of parameters. These example are not preferred except for a particular choice of parameter, in which case they specialize to the preferred connections constructed in Proposition $11.4$ of \cite{Bourgeois-Cahen}. The observation that the preferred condition implies $\nabla_{(i}\rf_{j)} = 0$ is what suggested that the examples from \cite{Bourgeois-Cahen} could be generalized to yield connections critical but not preferred. 

Attempts to patch together critical symplectic connections on Darboux charts on a sphere $\sphere$ or torus have failed, although on $\sphere$ this approach yields an interesting example of singular critical connections. Precisely, on the complement of two poles in the two-sphere $\sphere$ there is (see Theorem \ref{sphereextremaltheorem} and the surrounding discussion) a one-parameter family $\nabla(t)$ of rotationally symmetric critical symplectic connections such that $\nabla(0)$ is the Levi-Civita connection of the round metric on $\sphere$, that are neither moment flat nor preferred for $t \neq 0$, and which extend continuously but not differentiably at the poles (where $\K^{2} = \vc$) in the sense that the difference tensor $\nabla(t) - \nabla(0)$ extends continuously but not differentiably at the poles when $t \neq 0$. These connections satisfy $\emf(\nabla(t)) = 3\pi t^{2}$, so are not absolute minimizers of $\emf$ on $\sphere$, except when $t = 0$. These observations suggest that a critical symplectic connection smooth on all of $\sphere$ must be equivalent to $\nabla(0)$, but this has not been proved.

Finally, the explicit expressions for $\K$ and $\rf$ in Lemma \ref{examplelemma} simplify considerably if the difference tensor $\Pi_{ij}\,^{k} = \nabla - \pr$ is decomposable, meaning that $\Pi_{ijk} = X_{i}X_{j}X_{k}$ for some one-form $X_{i}$. If $X$ is moreover closed, explicit expressions are obtained, and, when $X = df$ is exact, $\K$ and $\rf$ are expressible in terms of $f$, $df$, and the Hessian of $f$. The conclusion, stated in Theorem \ref{cubeexampletheorem}, is that for a function $f \in \cinf(\rea^{2})$ the graph of which is an improper affine sphere, the connection $\nabla = \pr + df_{i}df_{j}df^{k}$ is moment flat but not projectively flat. A conceptual explanation for the appearance of affine spheres in these examples is lacking, but they suggest that in seeking examples it is useful to examine symplectic connections whose difference tensor $\Pi$ with some particularly nice fixed reference symplectic connection (e.g. the Levi-Civita connection of a constant curvature metric) has a simple form, e.g. is decomposable, is the symmetric product of a fixed metric with a one-form, etc. As is explained in Section \ref{momentflatexamplesection}, the particular case where $\Pi_{ijk} = X_{(i}g_{jk)}$ for a constant curvature metric $g$ and a harmonic one-form $X$ yields a general way of constructing moment flat connections that are not projectively flat. In particular, Theorem \ref{hyperbolicexampletheorem} shows that from a compact hyperbolic surface and a harmonic one-form there can be constructed in this way a moment flat connection that is not projectively flat. 

When $\K$ vanishes, $\rf$ is closed so determines a de Rham cohomology class. This raises the question of which cohomology classes are represented by the one-form $\rf$ of some moment flat symplectic connection. This is resolved by Theorem \ref{cohomtheorem}, that shows that every class in $H^{1}(M; \rea)$ is represented by the curvature one form of some symplectic connection. 
\begin{theorem}\label{cohomtheorem}
Let $(M, \Om)$ be a compact symplectic $2$-manifold with nonzero Euler characteristic. Let $[\al] \in H^{1}(M; \rea)$ be a de Rham cohomology class. There exists $\nabla \in \symcon(M, \Om)$ such that $\K(\nabla) = 0$ and $\rf(\nabla) \in [\al]$.
\end{theorem}
The representative $\nabla$ of Theorem \ref{cohomtheorem} is constructed using Theorem \ref{hyperbolicexampletheorem}.
\begin{remark}
It is not clear whether the connection $\nabla$ constructed in the proof of Theorem \ref{cohomtheorem} is complete. There remains unresolved the question: \textit{is every class in $H^{1}(M; \rea)$ represented by the curvature one-form of some \textbf{complete} symplectic connection?}

More generally, although a handful of the connections constructed in the various examples are shown to be complete, many probably are not, and the role of geodesic completeness in this context is not clear. In general, for affine connections, useful criteria for determining the completeness of a connection are not available, and, even in the case of flat affine connections, it is not clear that the completeness property has the same importance as it has in the metric setting.
\end{remark}

In \cite{Goldman}, W. Goldman showed that the projective Cotton tensor is a moment map for the action of the connected component of the identity, $\diff_{0}(M)$, of the group of diffeomorphisms of $M$ on the space $\projcon(M)$ of projective structures on $M$. 
In section \ref{cohomsection} it is explained in detail how the moment map of Theorem \ref{rhomomentmaptheorem}, for the action of $\Symplecto(M, \Om)$ on $\symcon(M, \Om)$, is related to the Goldman moment map, with the aim of addressing the following general issues.
First it is recalled that projectively flat connections abound. 
Moreover, that a projectively flat $\nabla \in \symcon(M, \Om)$ is moment flat leads to the simplest examples of moment flat symplectic connections that are not the Levi-Civita connections of Kähler metrics, because, as a consequence of a theorem of W. Goldman (recalled here as Theorem \ref{goldmandeformationtheorem}), on a compact orientable surface $M$ of genus at least $2$ there are many flat projective structures which are not represented by the Levi-Civita connection of any Riemannian metric. Suppose given, on such a surface, a flat projective structure $\en$ and suppose that there is a (necessarily Ricci symmetric) representative $\nabla \in \en$ which is the Levi-Civita connection of some Riemannian metric $g_{ij}$. Then $g_{ij}$ and the given orientation determine a unique Kähler structure, and because $\nabla$ is projectively flat it must be that $g_{ij}$ is conformally flat, so has constant negative curvature. That is $g_{ij}$ is hyperbolic. However, by Goldman's theorem, the deformation space of convex flat real projective structures on $M$ is homeomorphic to a ball of dimension $16(g -1)$. Since by standard Teichmüller theory the deformation space of flat conformal structures is homeomorphic to a ball of dimension $6(g-1)$, there are many flat projective structures admitting no representative connection that is a Levi-Civita connection. By a theorem proved independently by J. Loftin and F. Labourie (see \cite{Loftin-survey} or \cite{Labourie-flatprojective}), the deformation space of convex flat real projective structures on $M$ is parameterized by the bundle over the Teichmüller space of $M$ a fiber of which comprises cubic differentials holomorphic with respect to the conformal structure that is the base point; up to a constant fact, the real part of the holomorphic cubic differential is the difference tensor of a particular representative of the conformal structure with the representative of the flat real projective structure preserving its volume form. 

In the setting of symplectic connections, the work of Goldman, Labourie, and Loftin gives a good understanding of the convex flat projective structure on compact surfaces similar to that afforded by the uniformization theorem, and the principal remaining questions are understanding the difference between moment flat connections and projectively flat connections, and understanding the difference between critical symplectic connections and moment constant connections. For both questions relevant examples have been constructed and some substantial partial answers have been obtained, but neither is completely resolved. Theorems \ref{2dkahlertheorem} and \ref{momentflattheorem} show that, on a compact surface, a moment flat connection that is not projectively flat necessarily has a nonmetric character; precisely, Theorem \ref{momentflattheorem} shows that a moment flat connection that differs from a Levi-Civita connection by the real part of a cubic holomorphic differential is necessarily projectively flat. This is discussed in more detail in Section \ref{metricintrosection}.

\subsection{Restriction to the case of a Levi-Civita connection of a Kähler metric}\label{metricintrosection}
The Levi-Civita connections of Kähler metrics are among the most tractable examples of symplectic connections, and it is natural to ask when they are moment constant or critical. 
More generally, while there is no reason to suppose that a critical symplectic connection is nicely related to any particular compatible almost complex structure, once such a compatible structure has been fixed, it is natural to ask what are the critical symplectic connections, if any, among those related to it in some way. The simplest such setting is that of Levi-Civita connections of Kähler metrics.

When it is nonempty, the subspace $\intcom(M, \Om) \subset \acom(M, \Om)$ comprising integrable complex structures compatible with $\Om$ admits a $\Symplecto(M, \Om)$-equivariant map $\ldc:\intcom(M, \Om) \to \symcon(M, \Om)$ defined by $\ldc(J) = D$ where $D$ is the Levi-Civita connection of the metric $-J_{ij}$. On the other hand, there is no obvious canonical map of $\acom(M, \Om)$ into $\symcon(M, \Om)$ (the Hermitian connection has torsion). A version of Lemma \ref{kdlemma} was stated independently by L. La Fuente-Gravy as Proposition $4.4$ of \cite{LaFuente-Gravy}. (Although Lemma \ref{kdlemma} was not included in the first version of this article, it was known to the author then.)

\begin{lemma}\label{kdlemma}
For the Levi-Civita connection $D$ of a Kähler metric $(g, J, \Om)$ with scalar curvature $\sR_{g}$ on a $2n$-dimensional manifold, the Cahen-Gutt moment map is given by 
\begin{align}
\label{kkahlern}
\begin{split}
2\K\circ \ldc(J) &= 2\K(D) =   \lap_{g}\sR_{g} - |\ric|_{g}^{2} + \tfrac{1}{2}|\riem|_{g}^{2}\\
& = \lap_{g}\sR_{g} + \tfrac{1}{2}|A|^{2}_{g} - \tfrac{n-2}{n-1}|E|^{2}_{g} - \tfrac{n-1}{n(2n-1)}\sR_{g}^{2}\\
& = L_{g}\sR_{g} + \tfrac{1}{2}|A|^{2}_{g} - \tfrac{n-2}{n-1}|E|^{2}_{g} + \tfrac{(n-1)(n-2)}{2n(2n-1)}\sR_{g}^{2},
\end{split}
\end{align}
where $\lap_{g}f = g^{ij}D_{i}df_{j}$ is the Laplacian, $L_{g} = \lap_{g} - \tfrac{n-1}{2(2n-1)}\sR_{g}$ is the conformal Laplacian,  $E_{ij} = R_{ij} - \tfrac{1}{2n}\sR_{g}g_{ij}$ is the trace-free Ricci tensor, $A_{ijkl}$ is the conformal Weyl tensor of $g$ defined by
\begin{align}\label{conformalweyl}
\begin{split}
A_{ijkl} &= R_{ijk}\,^{p}g_{pl} + \tfrac{1}{n-1}\left(g_{k[i}R_{j]l} - g_{l[i}R_{j]k}\right) - \tfrac{1}{(n-1)(2n-1)}\sR_{g}g_{k[i}g_{j]l},
\end{split}
\end{align}
and the tensor norms are given by complete contraction with the metric.
\end{lemma}
\begin{proof}
The curvature of $D$ satisfies $J_{i}\,^{p}J_{j}\,^{q}R_{pqk}\,^{l} = R_{ijk}\,^{l}$ and $J_{i}\,^{p}J_{j}\,^{q}R_{pq} = R_{ij}$. It follows that $R_{ijkl}R^{ijkl} = |\riem|^{2}_{g}$ and $R^{ij}R_{ij} = |\ric|^{2}_{g}$. From the twice contracted differential Bianchi identity there follows $2D^{i}D^{j}R_{ij} = \lap_{g}\sR_{g}$. Together these observations yield the second equality of \eqref{kkahlern}. The final equality of \eqref{kkahlern} follows from the identity
\begin{align}
|A|^{2}_{g} = |\riem|^{2}_{g} - \tfrac{2}{n-1}|\ric|^{2}_{g} + \tfrac{1}{(n-1)(2n-1)}\sR_{g}^{2} = |\riem|^{2}_{g} - \tfrac{2}{n-1}|E|^{2}_{g} - \tfrac{1}{n(2n-1)}\sR_{g}^{2},
\end{align}
that is a consequence of \eqref{conformalweyl}.
\end{proof}
Lemma \ref{kdlemma} states that $\K(D)$ is the Laplacian of the scalar curvature plus terms quadratic in the curvature tensor, related to the first Pontryagin form. So, while the Levi-Civita connection of a Kähler metric is symplectic, the two moment maps associated with the actions of Hamiltonian diffeomorphisms on the spaces $\symcon(M, \Om)$ of symplectic connections and $\acom(M, \Om)$ compatible almost complex structures are different, though related (via a differential operator). 

 A Kähler structure $(g, J, \Om)$ is \textit{critical symplectic} if its Levi-Civita connection $D$ is critical symplectic. That a Kähler structure be critical symplectic means that its Levi-Civita connection is critical for $\emf$ with respect to variations through arbitrary symplectic connections, that in particular need not be Levi-Civita connections of Kähler structures. Other notions are possible, and also interesting. For example, \cite{LaFuente-Gravy} studies the critical points of the restriction of $\emf$ to fixed a Kähler class. This restricts the class of variations considered, and so potentially yields more critical points. An intermediate problem is to consider the critical points of the restriction of $\emf$ to the image $\ldc(\intcom(M, \Om)) \subset \symcon(M, \Om)$. 
None of these problems reduces to the previously studied problem of finding \textit{extremal almost Kähler structures}, that by definition are the critical points of the squared norm of the Hermitian scalar curvature on $\acom(M, \Om)$. That problem is surveyed in the last section of \cite{Apostolov-Draghici} and studied in \cite{Lejmi-einstein, Lejmi-almostkahler}. There are potentially interesting connections between these various problems, but elucidating them is not the purpose of this article.

Because the class of variations involved in its definition is large, the condition that a Kähler metric be critical symplectic is very strong, so much so as to raise the following question: \textit{If the Levi-Civita connection $D$ of an irreducible (in the Riemannian sense) Kähler metric on a simply-connected manifold is moment constant, must the metric be locally symmetric?}
Some evidence weakly supporting an affirmative answer is provided in Sections \ref{criticalkahlersection}, \ref{metricsection}, and \ref{higherdimensionalkahlersection}.

In the Kähler setting the uniformization theorem classifies the constant scalar curvature metrics, while Calabi's theorem shows that on compact surfaces the extremal Kähler metrics are simply constant curvature metrics (for simplicity the more complicated situation in the noncompact case is not discussed here; see \cite{Chen-extremalhermitian} and \cite{Wang-Zhu-extremal}). Theorem \ref{2dkahlertheorem} shows the analogous result for symplectic connections.
\begin{theorem}\label{2dkahlertheorem}
On a compact oriented surface, the Levi-Civita connection of a Riemannian metric $g$ is critical symplectic with respect to the symplectic structure determined by $g$ and the given orientation if and only if $g$ has constant curvature.
\end{theorem}

The difference tensor $\nabla - D$ of the Levi-Civita connection $D$ of the Cheng-Yau Riemannian metric on an oriented compact surface with a convex flat real projective structure and the unique connection $\nabla$ representing the projective structure and making parallel the volume induced by the metric is a cubic holomorphic differential with respect to the complex structure induced by the metric and the given orientation (this underlies the Loftin-Labourie parameterization of the moduli space of convex flat real projective structures by cubic holomorphic differentials). This means that $\nabla$ has a sort of metric character and motivates Theorem \ref{momentflattheorem}, that shows: \textit{on a compact oriented surface of genus at least one, a symplectic connection that differs from the Levi-Civita connection of a Riemannian metric by the real part of a cubic differential holomorphic with respect to the complex structure determined by the metric is moment flat if and only if it is projectively flat}. 

In conjunction with Theorem \ref{noextremaltheorem}, Theorems \ref{2dkahlertheorem} and \ref{momentflattheorem} can be interpreted as saying that critical symplectic connections of metric origin are projectively flat. This contrasts with Theorem \ref{cohomtheorem} that shows that moment flat connections that are not projectively flat abound. 

For a $4$-dimensional Kähler manifold, \eqref{kkahlern} simplifies (see \eqref{kks}) because in the decomposition of the conformal Weyl tensor into its self-dual and anti-self-dual parts, the self-dual part is a multiple of the scalar curvature. Section \ref{higherdimensionalkahlersection} presents some results in this setting. The main one is:
\begin{theorem}\label{k3theorem}
The Levi-Civita connection of the Ricci-flat Yau Kähler metric on a K3 surface is not critical symplectic.
\end{theorem}
This implies that the Levi-Civita connection of a compact $4$-dimensional Ricci-flat Kähler metric is critical symplectic if and only if the metric is flat and the underlying manifold is a torus. Theorem \ref{k3theorem} destroys naive ideas such as that the Levi-Civita connections of extremal Kähler metrics might be critical symplectic. 

\begin{remark}
Although it was indicated in the first version of this paper that many of the results make sense in any dimension, their statements were formulated only for the case $2n = 2$. In the current version, the more general statements have been included when they require little extra development. Since the first version was posted there appeared L. La Fuente-Gravy's \cite{LaFuente-Gravy}, which studies the critical points of the restriction to the the space of Levi-Civita connections of Kähler metrics of the functional called here $\emf$  with respect to variations within a fixed Kähler class. The objectives here and in \cite{LaFuente-Gravy} are different. Some material presented here that intersects with \cite{LaFuente-Gravy} has been included with the intention of distinguishing clearly between that context and the one considered here. As the proofs and point of view are different, the newly included material hopefully provides a useful complement to \cite{LaFuente-Gravy}.
\end{remark}

\section{Variation of the moment map}\label{momentmapsection}
In this section the variations of $\rf(\nabla)$ and $\K(\nabla)$ are computed and this is used to show that $\K$ is a moment map. The computations are structured to parallel the proof given in \cite{Donaldson-remarks} that the Hermitian scalar curvature is a moment map for the action of $\Ham(M, \Om)$ on $\acom(M, \Om)$. This manner of organizing the computations has the side benefit that the proof of Theorem \ref{criticaltheorem} characterizing the critical symplectic connections is a trivial formal computation.

Because there are only minor differences between the case $2n = 2$ of surfaces and the general case $2n > 2$, it is convenient to make the computations assuming $(M, \Om)$ is a $2n$-dimensional symplectic manifold, signaling, where appropriate, the specializations to the case $2n = 2$.

\subsection{}
Because $M$ is symplectic the canonical Poisson algebra structure on functions on $\ctm$ polynomial in the fibers can be transferred via the symplectic form to the algebra of covariant symmetric tensors. The result is the Schouten pairing $\{\dum,\dum\}:\Ga(S^{k}(\ctm))\times \Ga(S^{l}(\ctm)) \to \Ga(S^{k+l-1}(\ctm))$, expressible in terms of any $\nabla \in \symcon(M, \Om)$ by 
\begin{align}\label{schouten}
\{\al, \be\}_{i_{1}\dots i_{k+l-1}} = -k\al_{p(i_{1}\dots i_{k-1}}\nabla^{p}\be_{i_{k}\dots i_{k+l-1})} + l\be_{p(i_{1}\dots i_{l-1}}\nabla^{p}\al_{i_{l}\dots i_{k+l-1})}.
\end{align}
In \eqref{schouten} the sign is chosen so that $\{X^{\sflat}, Y^{\sflat}\} = [X, Y]^{\sflat}$ for $X, Y \in \Ga(TM)$. More generally, if $X \in \symplecto(M, \Om)$ then 
\begin{align}\label{xsflatal}
\begin{split}
\{X^{\sflat}, \al\}_{i_{1}\dots i_{k}} &  = -X_{p}\nabla^{p}\al_{i_{1}\dots i_{k}} + k\al_{p(i_{1}\dots i_{k-1}}\nabla^{p}X_{i_{k} ) }\\
& = X^{p}\nabla_{p}\al_{i_{1} \dots i_{k}} + k\al_{p(i_{1}\dots i_{k-1} }\nabla_{i_{k} ) }X^{p} = (\lie_{X}\al)_{i_{1}\dots i_{k}},
\end{split}
\end{align}
the penultimate equality because $\nabla_{[i}X_{j]} = 0$. In particular, $\lie_{\hm_{f}}\al = \{\al, df\}$ for $f \in \cinf(M)$. 
The Schouten bracket of functions is trivial. The Poisson bracket $\{f, g\}$ is related to the Schouten bracket of vector fields (which is the ordinary Lie bracket of vector fields) via the operator $f \to \hm_{f}$; precisely $\{df, dg\} = - d\{f, g\}$. 

A fiberwise endomorphism $A_{i}\,^{j} \in \Ga(\eno(TM))$ is infinitesimally symplectic if and only if $A_{[ij]} = 0$, and so the algebraic commutator of endomorphisms $[A, B]_{i}\,^{j} = A_{p}\,^{j}B_{i}\,^{p} - B_{p}\,^{j}A_{i}\,^{p}$ induces an algebraic commutator of symmetric covariant two-tensors $\al, \be \in \Ga(S^{2}(\ctm))$ given by $[\al, \be]_{ij} = 2\al_{p(i}\be_{j)}\,^{p}$. This algebraic bracket extends to an algebraic Poisson bracket $(\dum, \dum):\Ga(S^{k}(\ctm))\times \Ga(S^{l}(\ctm)) \to \Ga(S^{k+l-2}(\ctm))$ defined for $\al \in \Ga(S^{k}(\ctm))$ and $\be \in \Ga(S^{l}(\ctm))$ by
\begin{align}\label{algebraicbracket}
\begin{split}
(\al, \be)_{i_{1}\dots i_{k+l-2}} &= kl\al_{p(i_{1}\dots i_{k-1}}\be_{i_{k} \dots i_{k+l-2})}\,^{p}\\
& = \tfrac{kl}{k+l-2}\left((k-1)\al_{pi_{1}(i_{2}\dots i_{k-1}}\be_{i_{k}\dots i_{k+l-2})}\,^{p} - (l-1)\be_{pi_{1}(i_{2}\dots i_{l-1}}\al_{i_{l}\dots i_{k+l-2})}\,^{p}\right).
\end{split}
\end{align}
That $(\be, \al) = -(\al, \be)$ is apparent. Let $\ga \in \Ga(S^{m}(\ctm))$. That $(\dum, \dum)$ is a Lie bracket follows by summing cyclic permutations of the identity
\begin{align}
\begin{split}
((\al, \be), \ga)_{i_{1}\dots i_{k+l+m-4}} & = klm\left((k-1)\al_{pq(i_{1}\dots i_{k-2}}\be_{i_{k-1}\dots i_{k+l-3}}\,^{p}\ga_{i_{k+l-2}\dots i_{k+l+m-4})}\,^{q} \right.\\
&\left.- (l-1)\be_{pq(i_{1}\dots i_{l-2}}\al_{i_{k-1}\dots i_{k+l-3}}\,^{p}\ga_{i_{k+l-2}\dots i_{k+l+m-4})}\,^{q}\right).
\end{split} 
\end{align}
Note that the second equality of \eqref{algebraicbracket} makes sense without further interpretation provided at least one of $k$ and $l$ is greater than one. In the case $l = 1$ and $k > 1$, then $(\al, \be) = k\be^{p}\al_{pi_{1}\dots i_{k-1}}$ is simply $k$ times interior multiplication of the vector field $\be^{i}$ in $\al$. 

\subsection{}
For a linear operator $\dop:\Ga(S^{q}(\ctm)) \to \Ga(S^{p}(\ctm))$ write $|\dop| = p - q$. Define the (formal) \textit{adjoint} $\dop^{\ast}:\Ga(S^{p}(\ctm)) \to \Ga(S^{q}(\ctm))$ of $\dop$ by $\lb \dop \al, \be\ra = (-1)^{|\al||\dop|}\lb \al, \dop^{\ast}\be\ra$, where $|\al| = p$. The sign conforms with the rule of signs and guarantees that $(\dop^{\ast})^{\ast} = \dop$ and $(\P\Q)^{\ast} = (-1)^{|\P||\Q|}\Q^{\ast}\P^{\ast}$.

Define $\dn:\Ga(S^{k}(\ctm)) \to \Ga(S^{k+1}(\ctm))$ by $\dn\al_{i_{1}\dots i_{k+1}}= 2\nabla_{[i_{1}}\al_{i_{2}]i_{3}\dots i_{k+1}}$ and define $\sd:\Ga(S^{k}(\ctm)) \to \Ga(S^{k-1}(\ctm))$ by
\begin{align}
\sd \al_{i_{1}\dots i_{k-1}} = (-1)^{k-1}\tfrac{1}{2}(\dn\al)_{p}\,^{p}\,_{i_{1}\dots i_{k-1}} = (-1)^{k-1}\nabla_{p}\al_{i_{1}\dots i_{k-1}}\,^{p}.
\end{align} 
For example, $\rf = 2\sd \ric$. The formal adjoint $\sd^{\ast}$ of $\sd$ is given by $\sd^{\ast}\al_{i_{1}\dots i_{k+1}} = -\nabla_{(i_{1}}\al_{i_{2}\dots i_{k+1})}$. 
Straightforward computations using the Ricci identity show that, for $\al \in \Ga(S^{k}(\ctm))$,
\begin{align}
\label{sddecompose2n}
\nabla_{i}\al_{i_{1}\dots i_{k}} &= -\sd^{\ast}\al_{ii_{1}\dots i_{k}} + \tfrac{k}{k+1}\dn\al_{i(i_{1}\dots i_{k})}.
\end{align}

\begin{lemma}\label{sdcommutationlemma}
On a symplectic manifold $(M, \Om)$, for $\nabla \in \symcon(M, \Om)$ and $\al \in \Ga(S^{k}(\ctm))$ there holds
\begin{align}
\label{sdsdast2n}
\begin{split}
(k+1)\sd\sd^{\ast} \al_{i_{1}\dots i_{k}} + k\sd^{\ast}\sd\al_{i_{1} \dots i_{k}} &= (-1)^{k}\left(2kR_{(i_{1}}\,^{p}\al_{i_{2}\dots i_{k})p} - k(k-1)R^{p}\,_{(i_{1}i_{2}}\,^{q}\al_{i_{3}\dots i_{k})pq}\right)\\
& = (-1)^{k}\left((\al, \ric)_{i_{1}\dots i_{k}} - k(k-1)R^{p}\,_{(i_{1}i_{2}}\,^{q}\al_{i_{3}\dots i_{k})pq}\right).
\end{split}
\end{align}
\end{lemma}

\begin{proof}
The Ricci identity yields
\begin{align}\label{presd}
\begin{split}
(-1)^{k-1}\sd^{\ast}\sd\al_{i_{1}\dots i_{k}} & = \nabla_{(i_{1}}\nabla^{p}\al_{i_{2}\dots i_{k})p} \\
&= \nabla^{p}\nabla_{(i_{1}}\al_{i_{2}\dots i_{k})p} + (k-1)R^{p}\,_{(i_{1}i_{2}}\,^{q}\al_{i_{3}\dots i_{k})pq} - R^{p}\,_{(i_{1}}\al_{i_{2}\dots i_{k})p}.
\end{split}
\end{align}
Combining \eqref{presd} with 
\begin{align}
\begin{split}
(-1)^{k}(k+1)&\sd \sd^{\ast}\al_{i_{1}\dots i_{k}}  = (k+1)\nabla^{p}\nabla_{(i_{1}}\al_{i_{2}\dots i_{k}p)} \\
& = k\nabla^{p}\nabla_{(i_{1}}\al_{i_{2}\dots i_{k})p} + \nabla^{p}\nabla_{p}\al_{i_{1}\dots i_{k}}  = k\nabla^{p}\nabla_{(i_{1}}\al_{i_{2}\dots i_{k})p} + k R^{p}\,_{(i_{1}}\al_{i_{2}\dots i_{k})p},
\end{split}
\end{align}
yields \eqref{sdsdast2n}.
\end{proof}
\begin{remark}
When $2n = 2$, \eqref{sddecompose2n} and \eqref{sdsdast2n} take the simpler forms
\begin{align}
\label{sddecompose}
\nabla_{i}\al_{i_{1}\dots i_{k}} = -\sd^{\ast}\al_{ii_{1}\dots i_{k}} + (-1)^{k+1}\tfrac{k}{k+1}\Om_{i(i_{1}}\sd\al_{i_{2}\dots i_{k})},\\
\label{sdsdast}
(k+1)\sd\sd^{\ast} \al_{i_{1}\dots i_{k}} + k\sd^{\ast}\sd\al_{i_{1} \dots i_{k}} = (-1)^{k}k(k+1)R_{(i_{1}}\,^{p}\al_{i_{2}\dots i_{k})p}.
\end{align}
\end{remark}

Differentiating the pullback of $\nabla \in \affcon(M)$ along the flow of $X \in \Ga(TM)$ defines the Lie derivative $\lie_{X}\nabla$ of $\nabla$ along $X$. Explicitly,
\begin{align}\label{lienabla1}
\begin{split}
(\lie_{X}\nabla)_{ij}\,^{k} &= \nabla_{i}\nabla_{j}X^{k} + X^{p}R_{pij}\,^{k}.
\end{split}
\end{align}
On a surface, \eqref{lienabla1} simplifies to $(\lie_{X}\nabla)_{ij}\,^{k} = \nabla_{i}\nabla_{j}X^{k} + R_{ij}X^{k} - X^{p}R_{pj}\delta_{i}\,^{k}$, because $R_{ijk}\,^{l} = 2\delta_{[i}\,^{l}R_{j]k}$.

For $\nabla \in \symcon(M, \Om)$, define $\lop:\Ga(\ctm) \to T_{\nabla}\symcon(M, \Om)$ by $\lop(X^{\sflat})_{ijk} = (\lie_{X}\nabla)_{(ijk)}$ for $X \in \Ga(TM)$. Since $2(\lie_{X}\nabla)_{i[jk]} = \nabla_{i}dX^{\sflat}_{jk}$,
\begin{align}\label{lienablasym}
(\lie_{X}\nabla)_{ijk} = (\lie_{X}\nabla)_{(ijk)} + \tfrac{2}{3}\nabla_{(i}dX^{\sflat}_{j)k} = \lop(X^{\sflat})_{ijk} +  \tfrac{2}{3}\nabla_{(i}dX^{\sflat}_{j)k}.
\end{align}
If $X \in \symplecto(M, \Om)$, then $(\lie_{X}\nabla)_{ijk}$ is completely symmetric, so, in this case, 
\begin{align}\label{lopsymp}
\lop(X^{\sflat})_{ijk} = (\lie_{X}\nabla)_{(ijk)} = (\lie_{X}\nabla)_{ijk}.
\end{align}
The induced action of $\ham(M, \Om)$ on $\symcon(M, \Om)$ is given by the differential operator $\hop:\cinf(M) \to T_{\nabla}\symcon(M, \Om)$ defined by 
\begin{align}\label{hfdefined}
\begin{split}
\hop(f) & = \lie_{\hm_{f}}\nabla = \lop(-df) = \lop(\sd^{\ast}f). 
\end{split}
\end{align}

\begin{lemma}
Let $(M, \Om)$ be a symplectic manifold. For $\nabla \in \symcon(M, \Om)$ 
\begin{align}
\label{lophop} 
&\lop = (\sd^{\ast})^{2} - \sRo^{\ast},& &\hop = \lop \sd^{\ast} = (\sd^{\ast})^{3} - \sRo^{\ast}\sd^{\ast},\\
\label{lopast}
&\lop^{\ast} = -\sd^{2} - \sRo,& &\hop^{\ast} = (\lop \sd^{\ast})^{\ast} = \sd \lop^{\ast} = -\sd^{3} - \sd\sRo,&
\end{align}
where $\sRo:\Ga(S^{3}(\ctm)) \to \Ga(\ctm)$ and $\sRo^{\ast}:\Ga(\ctm) \to \Ga(S^{3}(\ctm))$ are defined by
\begin{align}
\label{srobe}
\sRo(\be)_{i} &= \be^{abc}R_{iabc} = -\tfrac{2}{n+1}\be_{i}\,^{ab}R_{ab} + \be^{abc}W_{iabc},\\
\sRo^{\ast}(\al)_{ijk} &= -\al^{p}R_{p(ijk)} = -\tfrac{2}{n+1}\al_{(i}R_{jk)} - \al^{p}W_{p(ijk)}.
\end{align}
For $X \in \symplecto(M, \Om)$ and $f \in \cinf(M)$,
\begin{align}\label{tracedhf}
&\sd \lop(X^{\sflat}) = \lie_{X}\ric, & &\sd \hop(f) = \lie_{\hm_{f}}\ric.
\end{align}
\end{lemma}
\begin{proof}
For a one-form $\al$, $\lop(\al)$ can be rewritten as
\begin{align}\label{lopn}
\begin{split}
\lop(\al)_{ijk} & 
= \nabla_{(i}\nabla_{j}\al_{k)} + \al^{p}R_{p(ijk)} 
 = (\sd^{\ast\, 2}\al)_{ijk} -\sRo^{\ast}(\al)_{ijk},
\end{split}
\end{align}
and with \eqref{hfdefined} this shows \eqref{lophop}. The identities \eqref{lopast} follow from \eqref{lophop} by taking formal adjoints.
Tracing \eqref{lienabla1} and using the Ricci identity shows that, for $\nabla \in \affcon(M)$,
\begin{align}\label{lienablatrace}
\nabla_{p}(\lie_{X}\nabla)_{ij}\,^{p} = (\lie_{X}\ric)_{ij} + \nabla_{i}\nabla_{j}\nabla_{p}X^{p} + 2R_{[jp]}\nabla_{i}X^{p} + 2X^{p}\nabla_{i}R_{[jp]}.
\end{align}
Specializing \eqref{lienablatrace} for $\nabla \in \symcon(M, \Om)$, $X \in \symplecto(M, \Om)$, and $f \in \cinf(M)$ yields \eqref{tracedhf}.
\end{proof}

\begin{remark}
On a surface, $W_{ijkl} = 0$ and $2X_{[i}R_{j]k} = -X^{p}R_{pk}\Om_{ij}$, so \eqref{lopn} simplifies to 
\begin{align}\label{lop2d}
\lop(\al)_{ijk}= (\sd^{\ast\, 2}\al)_{ijk} +  \al_{(i}R_{jk)}, 
\end{align}
and \eqref{hfdefined} simplifies to $\hop(f)_{ijk} = (\sd^{\ast\,3}f)_{ijk} - df_{(i}R_{jk)}$.
\end{remark}

\subsection{}
The first variation $\vr_{\Pi}\F(\nabla)$ of a functional $\F$ on $\symcon(M, \Om)$ at $\nabla \in \symcon(M, \Om)$ in the direction of $\Pi \in T_{\nabla} \symcon(M, \Om)$ is defined by $\vr_{\Pi} \F(\nabla) = \tfrac{d}{dt}_{|t = 0}\F(\nabla + t\Pi)$. 	
The formulas for the first and second variations of various curvatures are computed now.

For any $\al_{ijkl} \in \weylmod(\ctm, \Om)$ (recall this means $\al_{ijkl}$ has the algebraic symmetries of a symplectic curvature tensor), define a completely trace-free tensor $\wt(\al)_{ijkl} \in \weylmod(\ctm, \Om)$, by
\begin{align}
\al_{ijkl} = \wt(\al)_{ijkl} + \tfrac{1}{n+1}\left(\Om_{i(k}\al_{l)j} - \Om_{j(k}\al_{l)i} + \Om_{ij}\al_{kl}\right),
\end{align}
where $\al_{ij} = \al_{pij}\,^{p} = \tfrac{1}{2}\al_{p}\,^{p}\,_{ij}$ is the Ricci trace of $\al_{ijkl}$. For example, $W_{ijkl} = \wt(R)_{ijkl}$ when $R_{ijkl}$ is the curvature tensor of a symplectic connection. For $\al_{ijkl}, \be_{ijkl} \in \weylmod(\ctm, \Om)$,
\begin{align}\label{walwbe}
\al_{ijkl}\be^{ijkl} = \wt(\al)_{ijkl}\wt(\be)^{ijkl} + \tfrac{4}{n+1}\al_{ij}\be^{ij}.
\end{align}
If $\al, \be \in T_{\nabla}\symcon(M, \Om)$ are viewed as one-forms taking values in the bundle of symplectic endomorphisms of $TM$, their commutator $[\al, \be] \in \weylmod(M, \Om)$ is defined by $[\al, \be]_{ijk}\,^{l} = (\al \wedge \be + \be \wedge\al)_{ijk}\,^{l} = 2\al_{p[i}\,^{l}\be_{j]k}\,^{p} + 2\be_{p[i}\,^{l}\al_{j]k}\,^{p}$. Moreover, $[\al, \be]_{ijkl} \in \weylmod(\ctm, \Om)$, as a consequence of the complete symmetry of $\al_{ijk}$ and $\be_{ijk}$. In particular, $[\Pi, \Pi]_{ijkl} = 2(\Pi\wedge \Pi)_{ijkl} = 4\Pi_{pl[i}\Pi_{j]k}\,^{p} \in \weylmod(\ctm, \Om)$. Define $B(\Pi)_{ij}$ and $C(\Pi)_{ijkl}$ by
\begin{align}\label{bpicpi}
B(\Pi)_{ij}& = \Pi_{ip}\,^{q}\Pi_{jq}\,^{p} = - (\Pi \wedge \Pi)_{pij}\,^{p} = -\tfrac{1}{2}[\Pi, \Pi]_{pij}\,^{p} = -\tfrac{1}{4}[\Pi, \Pi]_{p}\,^{p}\,_{ij},\\
C(\Pi)_{ijkl} & = \wt([\Pi, \Pi])_{ijkl} = [\Pi, \Pi]_{ijkl} + \tfrac{2}{n+1}\left(\Om_{i(k}B(\Pi)_{l)j} - \Om_{j(k}B(\Pi)_{l)i} +  \Om_{ij}B(\Pi)_{kl}\right).
\end{align}
By definition, $C(\Pi)_{ijkl}$ is completely trace-free. Because
\begin{align}
\Pi_{plj}\Pi_{ik}\,^{p}\Pi_{q}\,^{li}\Pi^{jkq} = -\Pi_{kpi}\Pi_{jq}\,^{k} \Pi_{l}\,^{pj}\Pi^{iql} = -\Pi_{plj}\Pi_{ik}\,^{p}\Pi_{q}\,^{li}\Pi^{jkq},
\end{align}
where the first equality follows from raising and lowering indices, and the second from relabeling indices, there holds $\Pi_{plj}\Pi_{ik}\,^{p}\Pi_{q}\,^{li}\Pi^{jkq} = 0$. Consequently,
\begin{align}\label{pibcontract}
&[\Pi,\Pi]_{ijkl}[\Pi, \Pi]^{ijkl} = 8B(\Pi)_{ij}B(\Pi)^{ij},& &C(\Pi)^{ijkl}C(\Pi)_{ijkl} =\tfrac{8(n-1)}{n+1}B(\Pi)_{ij}B(\Pi)^{ij}.
\end{align}
Define a differential operator on $\Ga(S^{3}(\ctm))$ by
\begin{align}\label{dnwdefined}
\dnw\Pi_{ijkl} = \wt(\dn \Pi)_{ijkl} = \dn\Pi_{ijkl} - \tfrac{1}{n+1}\left(\Om_{i(k}\sd\Pi_{l)j} - \Om_{j(k}\sd\Pi_{l)i} + 2 \Om_{ij}\sd\Pi_{kl}\right).
\end{align}

\begin{lemma}\label{hadjointlemma}
Let $(M, \Om)$ be a $2n$-dimensional symplectic manifold. For $\nabla \in \symcon(M, \Om)$ and $\Pi_{ijk} \in T_{\nabla}\symcon(M, \Om)$,
\begin{align}
\label{rfvary2n}
\begin{split}
\rf(\nabla + t\Pi)_{i} & = \rf(\nabla)_{i} - 2t( \lop^{\ast}(\Pi)_{i} + \sWo(\Pi)_{i}) - 2t^{2}\left(\sd B(\Pi)_{i} + \Pi_{i}\,^{pq}\sd\Pi_{pq} \right) - 2t^{3}T(\Pi)_{i},
\end{split}\\
\label{kvary2n}
\begin{split}
\K(\nabla + t\Pi)  &=  \K(\nabla) + t\hop^{\ast}(\Pi)  + \tfrac{1}{2}t^{2}\sd(\Pi^{\ast}\Pi) \\
& + t^{3}\left(\sd T(\Pi)  +\tfrac{1}{4} \dnw\Pi^{ijkl}C(\Pi)_{ijkl} + \tfrac{n-1}{n+1}\sd\Pi^{ij}B(\Pi)_{ij} \right),
\end{split}
\end{align}
where the linear operator $\sWo:\Ga(S^{3}(\ctm)) \to \Ga(\ctm)$ is defined by
\begin{align}
\sWo(\be)_{i} &= \be^{abc}(W_{iabc} - \tfrac{n-1}{n+1}\Om_{i(a}R_{bc)}) = \be^{abc}(R_{i(abc)} - \Om_{i(a}R_{bc)}) = \be_{i}\,^{pq}R_{pq} + \sRo(\be)_{i},
\end{align}
$B(\Pi)$ and $C(\Pi)$ are defined in \eqref{bpicpi}, $(\Pi^{\ast}\Pi)_{i} = 3\sd B(\Pi)_{i} - \Pi^{abc}\nabla_{i}\Pi_{abc}$, $T(\Pi)_{i} = \Pi_{ia}\,^{b}B(\Pi)_{b}\,^{a}$, and $\lop^{\ast}:T_{\nabla}\symcon(M, \Om) \to \Ga(TM)$ and $\hop^{\ast}:T_{\nabla}\symcon(M, \Om) \to \cinf(M)$ are the adjoints of the operators $\lop$ and $\hop$ with respect to the pairing $\lb \dum, \dum \ra$. In particular,
\begin{align}\label{sympmoment}
&-\tfrac{1}{2}\vr_{\Pi}\rf_{\nabla} = \lop^{\ast}(\Pi)+ \sWo(\Pi),&
&\vr_{\Pi} \K_{\nabla} = \hop^{\ast}(\Pi).&
\end{align}
Moreover, for $f \in \cinf_{c}(M)$,
\begin{align}\label{kvarypaired}
\lb \K(\nabla + t\Pi), f\ra &= \lb \K(\nabla), f\ra + t\lb \hop(f) ,\Pi\ra + \tfrac{1}{2}t^{2}\lb \lie_{\hm_{f}}\Pi, \Pi\ra + O(t^{3}).
\end{align}
\end{lemma}

\begin{proof}
Given $\nabla \in \symcon(M, \Om)$, let $\Pi_{ijk} \in T_{\nabla}\symcon(M, \Om)$ and let $\bnabla = \nabla + \Pi_{ij}\,^{k}$. Let $\bar{R}_{ijk}\,^{l}$ be the curvature of $\bnabla= \nabla + t\Pi_{ij}\,^{k}$ and label with a $\bar{\,}$ the tensors derived from it, e.g. $\bar{R}_{ij}$ is the Ricci curvature of $\bnabla$. When necessary, the dependence of $\sd$ and $\sd^{\ast}$ on $\nabla$ is indicated by writing $\sd_{\nabla}$ and $\sd^{\ast}_{\nabla}$. Then:
\begin{align}
\label{riemvar} \bar{R}_{ijkl} &= R_{ijkl} + 2t\nabla_{[i}\Pi_{j]kl} + 2t^{2}\Pi_{pl[i}\Pi_{j]k}\,^{p}= R_{ijkl} + t\dn\Pi_{ijkl} + \tfrac{1}{2}t^{2}[\Pi, \Pi]_{ijkl},\\
\label{ricvar} \bar{R}_{ij} &= R_{ij} + t\nabla_{p}\Pi_{ij}\,^{p} -t^{2}\Pi_{ip}\,^{q}\Pi_{jq}\,^{p} = R_{ij} + t\sd\Pi_{ij} - t^{2}B(\Pi)_{ij},\\
\label{weylvar} \bar{W}_{ijkl} & = W_{ijkl} + t\dnw\Pi_{ijkl} + \tfrac{1}{2}t^{2}C([\Pi, \Pi])_{ijkl}.
\end{align}
For $\al \in \Ga(S^{k}(\ctm))$,
\begin{align}
\label{sdasttransform}
&\sd^{\ast}_{\bnabla}\al_{i_{1}\dots i_{k+1}} = \sd^{\ast}_{\nabla}\al_{i_{1}\dots i_{k+1}} + \tfrac{1}{3}(\al, \Pi)_{i_{1}\dots i_{k+1}},\\
\label{dnablatransform}
&\dbn\al_{[ij]i_{1}\dots, i_{k-1}} = \dn\al_{[ij]i_{1}\dots, i_{k-1}} -(k-1)\Pi_{i(i_{1}}\,^{p}\al_{i_{2}\dots i_{k-1})jp} + (k-1)\Pi_{j(i_{1}}\,^{p}\al_{i_{2}\dots i_{k-1})ip} ,\\
\label{sdtransform}
&\sd_{\bnabla}\al_{i_{1}\dots i_{k-1}} = \sd_{\nabla}\al_{i_{1}\dots i_{k-1}} + (-1)^{k-1}(k-1)\Pi^{pq}\,_{(i_{1}}\al_{i_{2}\dots i_{k-1})pq}.
\end{align}
Combining \eqref{sdasttransform} and \eqref{dnablatransform} with \eqref{ricvar} yields
\begin{align}
\label{sdastricvary}
\begin{split}
\sd_{\bnabla}^{\ast}\bric &= \sd^{\ast}\ric + t\left(\sd^{\ast}\sd\Pi + \tfrac{1}{3}(\ric, \Pi)\right) + t^{2}\left(-\sd^{\ast}B(\Pi) + \tfrac{1}{3}(\sd\Pi, \Pi)\right) - \tfrac{1}{3}t^{3}(B(\Pi), \Pi),
\end{split}\\
\label{dnablaricvary}
\begin{split}
\dbn\bric_{ijk} & = \dn\ric_{ijk} + t \left(\dn\sd\Pi_{ijk} - 2\Pi_{k[i}\,^{p}R_{j]p} \right) \\&\qquad - t^{2}\left(\dn B(\Pi)_{ijk} + 2\Pi_{k[i}\,^{p}\sd\Pi_{j]p} \right) + 2t^{3}\Pi_{k[i}\,^{p}B(\Pi)_{j]p}.
\end{split}
\end{align}
Combining \eqref{sdastricvary} and \eqref{dnablaricvary} with \eqref{sddecompose2n} yields
\begin{align}
\label{nablaricvary}
\begin{split}
\bnabla_{i}\bar{R}_{jk} &= \nabla_{i}R_{jk} + t\left(\nabla_{i}\sd\Pi_{jk} - 2\Pi_{i(j}\,^{p}R_{k)p}\right) \\
&\qquad -t^{2}\left(\nabla_{i}B(\Pi)_{jk} + 2\Pi_{i(j}\,^{p}\sd\Pi_{k)p}\right) + 2t^{3}\Pi_{i(j}\,^{p}B(\Pi)_{k)p}.
\end{split}
\end{align}
Combining \eqref{sdtransform} with \eqref{ricvar} and $\rf_{i} =2\sd\ric_{i}$, or contracting \eqref{dnablaricvary}, yields
\begin{align}
\label{rfvary2nb}
\begin{split}
\rf(\nabla + t\Pi)_{i} & = \bar{\rf}_{i}  = \rf_{i} + 2t\left(\sd^{2}\Pi_{i} - \Pi_{ipq}R^{pq}\right) - 2t^{2}\left(\sd B(\Pi)_{i} + \Pi_{i}\,^{pq}\sd\Pi_{pq} \right) + 2t^{3}\Pi_{i}\,^{pq}B(\Pi)_{pq}\\
& = \rf(\nabla)_{i} - 2t( \lop^{\ast}(\Pi)_{i} + \sWo(\Pi)_{i}) - 2t^{2}\left(\sd B(\Pi)_{i} + \Pi_{i}\,^{pq}\sd\Pi_{pq} \right) - 2t^{3}T(\Pi)_{i}.
\end{split}
\end{align}
For any $\be \in \Ga(S^{k}(\ctm))$, write $\be^{\ast}$ for the differential operator sending sections of $S^{k}(\ctm)$ to one-forms defined as the adjoint of the Schouten bracket $\{ \be, \dum\}$ on one-forms. Precisely, for $\al \in \Ga(\ctm)$ and $\be, \ga \in \Ga(S^{k}(\ctm))$, $\be^{\ast}$ is defined by the equation
\begin{align}
\lb \{\be, \al\}, \ga\ra = -\lb \al, \be^{\ast}\ga\ra.
\end{align}
A straightforward computation shows that, for $\Pi \in T_{\nabla}\symcon(M, \Om)$, 
\begin{align}\label{dualhf}
(\Pi^{\ast}\Pi)_{i} = 3\sd B(\Pi)_{i} - \Pi^{abc}\nabla_{i}\Pi_{abc}.
\end{align}
From the identities
\begin{align}
\begin{split}
\sd(\Pi_{ipq}\sd\Pi^{pq}) &= \sd\Pi_{pq}\sd\Pi^{pq} - \Pi^{abc}\sd^{\ast}\sd\Pi_{abc},\\
\sd(\dn\Pi_{iabc}\Pi^{abc}) & = -\tfrac{1}{2}\dn\Pi_{ijkl}\dn\Pi^{ijkl} - 2R^{pq}B(\Pi)_{pq} +  \Pi^{abc}\sd^{\ast}\sd\Pi_{abc} - \tfrac{1}{2}R^{ijkl}[\Pi, \Pi]_{ijkl}, \\
\sd B(\Pi) & = \Pi_{ipq}\sd\Pi^{pq} + \Pi^{abc}\nabla_{a}\Pi_{bci},
\end{split}
\end{align}
(where, for example, the abusive notation $\sd(\Pi_{ipq}\sd\Pi^{pq})$ means $\sd$ applied to $\Pi_{ipq}\sd\Pi^{pq}$)
it follows that 
\begin{align}\label{sdsdb}
\begin{split}
& \sd^{2}B(\Pi) + \tfrac{1}{2}\sd\left(\Pi_{ipq}\sd\Pi^{pq} - \dn\Pi_{iabc}\Pi^{abc} \right)   = \tfrac{3}{2}\sd^{2}B(\Pi) - \tfrac{1}{2}\sd\left(\Pi^{abc}\nabla_{i}\Pi_{abc}\right) = \tfrac{1}{2}\sd(\Pi^{\ast}\Pi).
\end{split}
\end{align}
Note that
\begin{align}\label{sdwpi}
\sd \sWo(\Pi) = \nabla^{i}\Pi^{jkl}(\Om_{ij}R_{kl} - R_{ijkl}). 
\end{align}
From \eqref{riemvar}, \eqref{ricvar}, \eqref{sdwpi}, \eqref{sdsdb}, \eqref{walwbe}, and \eqref{pibcontract} there results
\begin{align}\label{rquadvary}
\begin{split}
-\tfrac{1}{2}&\bar{R}_{ij}\bar{R}^{ij}  + \tfrac{1}{4}\bar{R}_{ijkl}\bar{R}^{ijkl} = -\tfrac{1}{2}R_{ij}R^{ij} + \tfrac{1}{4}R_{ijkl}R^{ijkl} - t\sd \sWo(\Pi) \\
& + t^{2}\left(-\tfrac{1}{2}\sd\Pi_{ij}\sd\Pi^{ij} + \tfrac{1}{4}\dn\Pi_{ijkl}\dn\Pi^{ijkl} + R^{ij}B(\Pi)_{ij} + \tfrac{1}{4}[\Pi, \Pi]_{ijkl}R^{ijkl}\right)\\
& + t^{3}\left(\sd\Pi^{ij}B(\Pi)_{ij} + \tfrac{1}{4}\dn\Pi^{ijkl}[\Pi, \Pi]_{ijkl} \right)  + t^{4}\left(-\tfrac{1}{2}B(\Pi)_{ij}B(\Pi)^{ij} + \tfrac{1}{16}[\Pi, \Pi]_{ijkl}[\Pi, \Pi]^{ijkl}\right)\\
& = -\tfrac{1}{2}R_{ij}R^{ij} + \tfrac{1}{4}R_{ijkl}R^{ijkl} - t\sd \sWo(\Pi)   - \tfrac{1}{2}t^{2}\sd\left(\Pi_{ipq}\sd\Pi^{pq} + \dn\Pi_{iabc}\Pi^{abc}\right)\\
& \qquad + t^{3}\left(\tfrac{1}{4} \dnw\Pi^{ijkl}C(\Pi)_{ijkl} + \tfrac{n-1}{n+1}\sd\Pi^{ij}B(\Pi)_{ij}  \right).
\end{split}
\end{align}
By \eqref{sdtransform}, $\sd_{\bnabla}\bar{\rf} = \sd_{\nabla}\bar{\rf}$. With this observation, combining \eqref{rfvary2n} and \eqref{rquadvary} and using \eqref{sdsdb} yields
\begin{align}\label{kvary2}
\begin{split}
\K&(\nabla + t\Pi)  = \K(\nabla) + t\sd\lop^{\ast}(\Pi)  + t^{2}\left(\sd^{2} B(\Pi) + \tfrac{1}{2}\sd(\Pi_{ipq}\sd\Pi^{pq} - \dn\Pi_{iabc}\Pi^{abc})\right)\\
& \qquad + t^{3}\left(\sd T(\Pi)  +\tfrac{1}{4} \dnw\Pi^{ijkl}C(\Pi)_{ijkl} + \tfrac{n-1}{n+1}\sd\Pi^{ij}B(\Pi)_{ij} \right) \\
 & = \K(\nabla) + t\hop^{\ast}(\Pi)  + \tfrac{1}{2}t^{2}\sd(\Pi^{\ast}\Pi)  + t^{3}\left(\sd T(\Pi)  +\tfrac{1}{4} \dnw\Pi^{ijkl}C(\Pi)_{ijkl} + \tfrac{n-1}{n+1}\sd\Pi^{ij}B(\Pi)_{ij} \right).
\end{split}
\end{align}
For $f \in \cinf(M)$, $\lb \sd(\Pi^{\ast}\Pi), f\ra = -\lb \Pi^{\ast}\Pi, \sd^{\ast}f\ra = \lb \lie_{\hm_{f}}\Pi, \Pi\ra$,
and so \eqref{kvarypaired} follows from \eqref{kvary2n}.
\end{proof}

\begin{proof}[Proofs of Theorems \ref{momentmaptheorem} and \ref{rhomomentmaptheorem}]
Reinterpreting Lemma \ref{hadjointlemma} yields Theorems \ref{momentmaptheorem} and \ref{rhomomentmaptheorem}. The equivariance of $\rf$ and $\K$ with respect to the action of $\Symplecto(M, \Om)$ follows from the evident naturalness of their definitions with respect to pullback via elements of $\Symplecto(M, \Om)$.
For $\nabla \in \symcon(M, \Om)$, $\Pi \in T_{\nabla}\symcon(M, \Om)$, $X \in \symplecto(M, \Om)$, and $f \in \cinf_{c}(M)$ it follows from \eqref{rfvary} and \eqref{kvary} that
\begin{align}
\label{rhomomentidentity}
-2\sOmega_{\nabla}( \lop(X^{\sflat}), \Pi) & = -2\lb X^{\sflat}, \lop^{\ast}(\Pi)\ra = \vr_{\Pi}\lb X^{\sflat}, \rf\ra - \lb X^{\sflat}, \sWo(\Pi)\ra,\\
\label{kmomentidentity}
\sOmega_{\nabla}(\hop(f), \Pi) &=   \lb f, \hop^{\ast}(\Pi)\ra = \vr_{\Pi} \lb f, \K(\nabla)\ra.
\end{align}
The identity \eqref{kmomentidentity} shows that $\K$ is a moment map.

If $2n = 2$, then $\sRo(\be)_{i} = -\be_{i}\,^{ab}R_{ab}$, so $\sWo(\be) = 0$. Because $\sWo(\Pi)$, $\dnw\Pi$, and $C(\Pi)$ vanish when $2n = 2$, in this case \eqref{rfvary2n} and \eqref{kvary2n} specialize to
\begin{align}
\label{rfvary}
\begin{split}
\rf(\nabla + t\Pi)_{j} & = \rf(\nabla)_{j} - 2t\L^{\ast}(\Pi)_{j} - 2t^{2}\left(\sd B(\Pi)_{j} + \Pi_{jpq}\sd\Pi^{pq} \right) - 2t^{3}T(\Pi)_{j},
\end{split}\\
\label{kvary}
\begin{split}
\K(\nabla + t\Pi)  &= \K(\nabla) +  t\hop^{\ast}(\Pi) + \tfrac{1}{2}t^{2}\sd(\Pi^{\ast}\Pi)  + t^{3}\sd T(\Pi).
\end{split}
\end{align}
Hence, when $2n = 2$, \eqref{rhomomentidentity} becomes
\begin{align}
\label{2drhomomentidentity}
-2\sOmega_{\nabla}( \lop(X^{\sflat}), \Pi) & = -2\lb X^{\sflat}, \lop^{\ast}(\Pi)\ra = \vr_{\Pi}\lb X^{\sflat}, \rf\ra ,
\end{align}
showing that $-\tfrac{1}{2}[\rf]$ is a moment map.
\end{proof}

\begin{remark}
There follows an alternative derivation of \eqref{rquadvary} that is perhaps clearer conceptually. For $\al_{ijkl} \in \weylmod(\ctm, \Om)$, define $\al_{ij} = \al_{pij}\,^{p}$. For $\al, \be \in \weylmod(\ctm, \Om)$, let $\tau_{ijkl} = \al_{[ij|p|}\,^{q}\be_{kl]q}\,^{p}$. Then
\begin{align}\label{dlefab}
&3 \tau_{p}\,^{p}\,_{ij} = - \al_{ijpq}\be^{pq} - \be_{ijpq}\al^{pq} - 2\al_{[i}\,^{abc}\be_{j]abc},&
&3 \tau_{p}\,^{p}\,_{q}\,^{q}  = 2\al_{abcd}\be^{abcd} - 4\al_{pq}\be^{pq}.
\end{align}
Applying \eqref{dlefab} to the expression \eqref{pon} for the first Pontryagin form $\pon_{1}$ of $\nabla$ yields
\begin{align}\label{dlefpon}
\begin{split}
2\pi^{2}\pon_{1\,p}\,^{p}\,_{ij} &= 
R_{ijpq}R^{pq} - R_{iabc}R_{j}\,^{abc} \\
&= W_{i}\,^{abc}W_{jabc} + \tfrac{n-2}{n+1}\left(R^{pq}W_{ijpq} + \tfrac{1}{n+1}R_{pi}R_{j}\,^{p}\right) + \tfrac{2n-1}{2(n+1)^{2}}R^{pq}R_{pq}\Om_{ij},\\
\pi^{2}\pon_{1\,p}\,^{p}\,_{q}\,^{q}& =  
R_{ab}R^{ab} - \tfrac{1}{2}R_{abcd}R^{abcd} =  \tfrac{n-1}{n+1}R^{ab}R_{ab} - \tfrac{1}{2} W^{abcd}W_{abcd}.
\end{split}
\end{align}
From \eqref{dlefpon} it follows that 
\begin{align}\label{pontryagin}
\begin{split}
8\pi^{2}&\pon_{1} \wedge \Omk{(n-2)} = \pi^{2}\pon_{1\,p}\,^{p}\,_{q}\,^{q} \Omk{n} \\
&= \left(R_{ij}R^{ij} - \tfrac{1}{2}R_{ijkl}R^{ijkl}\right)\Omk{n} = \left( \tfrac{n-1}{n+1}R^{ab}R_{ab} - \tfrac{1}{2} W^{abcd}W_{abcd} \right)\Omk{n}.
\end{split}
\end{align}
(This is stated in both \cite{Bieliavsky-Cahen-Gutt-Rawnsley-Schwachhofer} and \cite{Gutt-remarks}, although the notation is different in those references).

Let $\nabla \in \symcon(M, \Om)$ and $\Pi_{ijk} \in T_{\nabla}\symcon(M, \Om)$, and set $\bnabla = \nabla + \Pi_{ij}\,^{k}$. The Chern-Simons theory shows that first Pontryagin forms $\pon_{1}(\bnabla)$ and $\pon_{1}(\nabla)$ are cohomologous. Precisely, specializing the Chern-Simons theory (for example, rewriting Lemma $3.1$ of the appendix to \cite{Chern-withoutpotential} in the notations used here) shows that $\pon_{1}(\bnabla) -\pon_{1}(\nabla) = -\tfrac{3}{4\pi^{2}}d\si(\nabla, \Pi)$ where $\si_{ijk} = \si(\nabla, \Pi)_{ijk}$ is the three-form
\begin{align}\label{sipon}
\begin{split}
\si_{ijk} & = \Pi_{[i|p|}\,^{q}R_{jk]q}\,^{p} + \tfrac{1}{2}\Pi_{[i|p|}\,^{q}\dn\Pi_{jk]q}\,^{p} + \tfrac{1}{6}\Pi_{p[i}\,^{q}[\Pi,\Pi]_{jkl]q}\,^{p}\\
& = \Pi_{[i|p|}\,^{q}R_{jk]q}\,^{p} + \tfrac{1}{2}\Pi_{[i|p|}\,^{q}\dn\Pi_{jk]q}\,^{p} + \tfrac{2}{3}\Pi_{[i|a|}\,^{b}\Pi_{j|b|}\,^{c}\Pi_{k]c}\,^{a}.
\end{split}
\end{align}
The difference $\pon_{1}(\nabla + t\Pi) - \pon_{1}(\nabla) = d\si(\nabla, t\Pi)$ can be computed either by finding the exterior differential of $\si(\nabla, t\Pi)$ using \eqref{sipon} or by using the definition of $\pon_{1}$ and \eqref{riemvar}. There results
\begin{align}\label{ponvary}
\begin{split}
-\tfrac{4\pi^{2}}{3}&\left( \pon_{1}(\nabla + t\Pi)_{ijkl} -  \pon_{1}(\nabla)_{ijkl} \right) =  2t \left(R_{[ij|p|}\,^{q}\dn\Pi_{kl]q}\,^{p}\right) \\
&\qquad  + t^{2}\left(\dn\Pi_{[ij|p|}\,^{q}\dn\Pi_{kl]q}\,^{p} + R_{[ij|p|}\,^{q}[\Pi, \Pi]_{kl]q}\,^{p}\right) + t^{3}\dn\Pi_{[ij|p|}\,^{q}[\Pi, \Pi]_{kl]q}\,^{p}.
\end{split}
\end{align}
Contracting \eqref{ponvary} and using \eqref{pontryagin} yields an alternative derivation of \eqref{rquadvary}.
\end{remark}

\begin{remark}\label{pontryaginsection}
It is not clear whether the action of $\Symplecto(M, \Om)$ on $\symcon(M, \Om)$ is Hamiltonian when $2n > 2$.
It is straightforward to see that 
\begin{align}\label{wpi}
\begin{split}
-\tfrac{3}{2}&\tfrac{d}{dt}\big|_{t = 0}\si(\nabla, t\Pi)_{ip}\,^{p}  =  -\tfrac{3}{2}\Om^{jk}\Pi_{[i|p|}\,^{q}R_{jk]q}\,^{p} = (R_{i(abc)} - \Om_{i(a}R_{bc)})\Pi^{abc}= \sWo(\Pi)_{i}.
\end{split}
\end{align}
where $\si(\nabla, t\Pi)$ is the primitive of $\pon_{1}(\nabla + t\Pi) - \pon_{1}(\nabla)$ defined in \eqref{sipon}. Nonetheless, it is not clear how to produce a one-form valued function on $\symcon(M, \Om)$ with first variation equal to $\sWo(\Pi)$, although the last equality of \eqref{wpi} is suggestive in this regard. Fix a reference connection $\bnabla \in \symcon(M, \Om)$ and consider the primitive $\si(\nabla, \bnabla - \nabla)$ of $\pon_{1}(\bnabla) - \pon_{1}(\nabla)$ defined in \eqref{sipon}. Then for $\tau(\nabla)_{i} = \si(\nabla, \bnabla - \nabla)_{ip}\,^{p}$, by \eqref{sipon}, the first variation $\vr_{\Pi}\tau_{i} = \tfrac{d}{dt}\big|_{t = 0}\tau(\nabla + t\Pi)_{i} = \tfrac{d}{dt}\big|_{t = 0}\si(\nabla + t\Pi, \bnabla - \nabla - t\Pi)_{ip}\,^{p}$ differs from $\Om^{jk}\Pi_{[i|p|}\,^{q}R_{jk]q}\,^{p} = -\tfrac{2}{3}\sWo(\Pi)_{i}$ by a term involving $\bnabla - \nabla$ and the first variation of the curvature of $\nabla$. For a construction along these lines to work, it is not necessary that the last mentioned term vanish, rather it is sufficient that it be the image of some two-form under $\dad$. Nonetheless, it is not clear if such a condition can be achieved. 

One might consider restricting to the class of symplectic connections for which the equality $\lb \lop(X^{\sflat}), \Pi\ra  = -\tfrac{1}{2}\vr_{\Pi}\lb X^{\sflat}, \rf\ra$ holds. 
When $2n > 2$ that $\sWo(\Pi)$ vanish for all $\Pi$ means that $(n+1)W_{i(jkl)} = (n-1)\Om_{i(j}R_{kl)}$. Tracing this shows that it holds if and only if $R_{ij} = 0$, in whch case $W_{i(jkl)} = 0$ and so also $R_{i(jkl)} = 0$. This implies $R_{ijkl} = 0$. Hence restricting to the class of symplectic connections for which there holds $\lb \lop(X^{\sflat}), \Pi\ra  = -\tfrac{1}{2}\vr_{\Pi}\lb X^{\sflat}, \rf\ra$ amounts to restricting to the flat symplectic connections, in which case $\rf$ is zero anyway.
\end{remark}

\subsection{}
This section concludes recording some miscellaneous facts about the operators $\lop$, $\hop$, and their adjoints, and their relations with the orbits of $\Symplecto(M, \Om)$ and $\Ham(M, \Om)$ on $\symcon(M, \Om)$. These are mostly specializations of facts that are known to hold in general for moment maps on finite-dimensional spaces, but which need to be checked in a infinite-dimensional setting.
\begin{lemma}\label{vkhlemma}
On a symplectic manifold $(M, \Om)$, for $\nabla \in \symcon(M, \Om)$, $X, Y \in \symplecto(M, \Om)$, and $f, g \in \cinf_{c}(M)$,
\begin{align}
\label{isotropyidentity1}
&\lop^{\ast}\lop(X^{\sflat})  + \sWo(\lop(X^{\sflat})) = -\tfrac{1}{2}\lie_{X}\rf(\nabla),&\\
\label{isotropyidentity2}
&\hop^{\ast}\hop(f) = \{f, \K(\nabla)\},&\\
\label{isotropyidentity3}
&\lb \hop(f), \hop(g)\ra = \lb f, \{g, \K(\nabla)\}\ra = \lb \{f, g\}, \K(\nabla)\ra  .&
\end{align}
When $\dim M = 2$, for $X \in \symplecto(M, \Om)$, there holds $\lop^{\ast}\lop(X^{\sflat}) = -\tfrac{1}{2}\lie_{X}\rf(\nabla)$, and, additionally,
\begin{align}
\label{isotropyidentity2d}
&\sOm( \lop(X^{\sflat}), \lop(Y^{\sflat}))  = \lb\Om(X, Y), \K(\nabla)\ra.
\end{align}
\end{lemma}

\begin{proof}
Differentiating the relations $\rf(\phi_{t}^{\ast}\nabla) = \phi_{t}^{\ast}\rf(\nabla)$ and $\K(\phi_{t}^{\ast}\nabla) = \K(\nabla) \circ \phi_{t}$ along the flow $\phi_{t}$ of $X \in \symplecto(M, \Om)$, and using Lemma \ref{hadjointlemma} yields
\begin{align}
\label{vklie}
-2(\lop^{\ast}\lop(X^{\sflat})& + \sWo(\lop(X^{\sflat})) = \tfrac{d}{dt}_{|t = 0}\rf(\phi_{t}^{\ast}\nabla) =  \tfrac{d}{dt}_{|t = 0}\phi_{t}^{\ast}\rf(\nabla) = \lie_{X}\rf(\nabla),\\
\label{hllie}\hop^{\ast}\lop(X^{\sflat})  &= \tfrac{d}{dt}_{|t = 0}\K(\phi_{t}^{\ast}\nabla) = \tfrac{d}{dt}_{|t = 0}\K(\nabla)\circ \phi_{t} = d\K(\nabla)(X) = \Om(X, H_{\K(\nabla)}).
\end{align}
Taking $X = H_{f}$ in \eqref{hllie} gives \eqref{isotropyidentity2}. 
The first equality of \eqref{isotropyidentity3} follows from \eqref{isotropyidentity2}, and the second equality of \eqref{isotropyidentity3} follows from the identity $\lb \{a, b\}, c\ra = - \lb \{b, a\}, c\ra = \lb a, \{b, c\}\ra$, valid for any $a \in \cinf_{c}(M)$ and $b, c \in \cinf(M)$, because $\{b, ac\}\Omk{n}$ is exact. When $\dim M = 2$, \eqref{isotropyidentity2d} follows from \eqref{vklie} and \eqref{cdiv}. 
\end{proof}

The symplectic complement $V^{\perp}$ of a subspace $V$ of a symplectic vector space $(W, \Om)$ is defined by $V^{\perp} = \{w \in W: \Om(v, w) = 0\,\,\text{for all}\,\, v \in v\}$. In the infinite-dimensional case, for a weakly nondegenerate symplectic form, although $V \subset (V^{\perp})^{\perp}$, it need not be the case that $(V^{\perp})^{\perp} = V$.
\begin{lemma}\label{hastkernellemma}
On a $2n$-dimensional symplectic manifold $(M, \Om)$, for $\nabla \in \symcon(M, \Om)$ there holds
\begin{align}
T_{\nabla}(\Ham(M, \Om)\cdot \nabla)^{\perp} = \hop(\ham(M, \Om))^{\perp} = \ker \hop^{\ast}.
\end{align}
Moreover,
\begin{align}
\begin{split}
\ker \hop^{\ast} &= \{\Pi \in T_{\nabla}\symcon(M, \Om):\lop^{\ast}(\Pi)\wedge \Omk{(n-1)} \,\,\text{is closed}\},\\
\lop(\symplecto(M, \Om)^{\sflat})^{\perp} & =  \{\Pi \in T_{\nabla}\symcon(M, \Om):\lop^{\ast}(\Pi)\wedge \Omk{(n-1)} \,\,\text{is exact}\}.
\end{split}
\end{align}
\end{lemma}

\begin{proof}
That $\ker \hop^{\ast} = \hop(\ham(M, \Om))^{\perp}$ follows from the strong nondegeneracy of the $L^{2}$ inner product on functions. By definition, $\sOm_{\nabla}(\hop(f), \Pi) = \lb f, \hop^{\ast}(\Pi)\ra$ for all $f \in \ham(M, \Om)$ and $\Pi \in T_{\nabla}\symcon(M, \Om)$. It is immediate that $\ker \hop^{\ast} \subset\hop(\ham(M, \Om))^{\perp}$. On the other hand, if $\Pi \in \hop(\ham(M, \Om))^{\perp}$, then $0 = \lb f, \hop^{\ast}(\Pi)\ra$ for all $f \in \ham(M, \Om)$. If $M$ is noncompact this means $\hop^{\ast}(\Pi) = 0$, while if $M$ is compact it means $\hop^{\ast}(\Pi)$ is constant; since $\hop^{\ast}(\Pi)$ is a divergence, this constant must be $0$. Hence $\hop(\ham(M, \Om))^{\perp} \subset \ker \hop^{\ast}$. 

For $\Pi \in T_{\nabla}\symcon(M, \Om)$, $\hop^{\ast}(\Pi)\Omk{n} = \sd\lop^{\ast}(\Pi)\Omk{n} = 2d\lop^{\ast}(\Pi)\wedge \Omk{(n-1)}$, from which it is immediate that $\hop^{\ast}(\Pi) =0$ if and only if $\lop^{\ast}(\Pi)\wedge \Omk{(n-1)}$ is closed. From 
\begin{align}
\sOm_{\nabla}(\lop(X^{\sflat}), \Pi) = \lb X^{\sflat}, \lop^{\ast}(\Pi)\ra = \int_{M}X_{p}\lop^{\ast}(\Pi)^{p}\,\Omk{n} = \int_{M} X^{\sflat} \wedge \lop^{\ast}(\Pi)\wedge \Omk{(n-1)}
\end{align}
it follows that $\Pi \in \lop(\symplecto(M, \Om)^{\sflat})^{\perp}$ if and only if $\int_{M} \al \wedge \lop^{\ast}(\Pi)\wedge \Omk{(n-1)} = 0$ for all closed, compactly supported one-forms $\al$. By Poincare duality this holds if and only if $\lop^{\ast}(\Pi)\wedge \Omk{(n-1)}$ is exact. 
\end{proof}

If $\mu$ is the moment map of an action of a connected Lie group $G$ on a finite-dimensional symplectic manifold, a $G$ orbit is isotropic if and only if $\mu$ is constant on the $G$ orbit. Lemma \ref{orbitlemma} extends this observation to the setting considered here. Note that, since $\K$ is $\Symplecto(M, \Om)$-equivariant, that $\K$ be constant on a $\Symplecto(M, \Om)$ or $\Ham(M, \Om)$ orbit implies that $\K(\nabla)$ is a constant function for any $\nabla$ in the orbit.  

\begin{lemma}\label{orbitlemma}
Let $(M, \Om)$ be a symplectic manifold and let $\nabla \in \symcon(M, \Om)$.
\begin{enumerate}
\item The following are equivalent:
\begin{enumerate}
\item $\K(\nabla)$ is constant.
\item The tangent space to $\Symplecto(M, \Om)\cdot \nabla$ is contained in the $\sOm$-complement of the tangent space to $\Ham(M, \Om)\cdot \nabla$.
\item The orbit $\Ham(M, \Om) \cdot \nabla$ is isotropic.
\end{enumerate}
\item If $\dim = 2$ then:
\begin{enumerate}
\item The orbit $\Symplecto(M, \Om) \cdot \nabla$ is isotropic if $\K(\nabla) = 0$.
\item If the orbit $\Symplecto(M, \Om)\cdot \nabla$ is isotropic, then $\K(\nabla)$ is constant.
\end{enumerate}
In particular, if $M$ is compact then $\Symplecto(M, \Om)\cdot \nabla$ is isotropic if and only if $\K(\nabla) = 0$.
\end{enumerate}
\end{lemma}
\begin{proof}
The orbit $\Symplecto(M, \Om)\cdot\nabla$ is the image of the map $\Phi^{\nabla}:\Symplecto(M, \Om) \to \symcon(M, \Om)$ defined by $\Phi^{\nabla}(\phi) = \phi^{\ast}(\nabla)$, and $\Ham(M, \Om)\cdot \nabla = \Phi^{\nabla}(\Ham(M, \Om))$. For $\phi \in \Symplecto(M, \Om)$, an element of $T_{\phi}\Symplecto(M, \Om)$ has the form $T\phi(X)$ for $X \in \symplecto(M, \Om)$. Since $T\Phi^{\nabla}(\phi)(T\phi(X)) = \phi^{\ast} \lop(X^{\sflat})$ for $\phi \in \Symplecto(M, \Om)$ and $X \in \symplecto(M, \Om)$, the pullback $\Phi^{\nabla\, \ast}(\sOmega)$ is given by 
\begin{align}\label{sompb}
\begin{split}
\Phi^{\nabla\,\ast}(\sOmega)_{\phi}(T\Phi(X), T\Phi(Y)) &= \sOmega_{\phi^{\ast}\nabla}(\phi^{\ast}\lop(X^{\sflat}), \phi^{\ast}\lop(Y^{\sflat})) = \sOmega_{\nabla}(\lop(X^{\sflat}), \lop(Y^{\sflat})),
\end{split}
\end{align}
for $X, Y \in \symplecto(M, \Om)$. If $\phi \in \Symplecto(M, \Om)$, $f \in \ham(M, \Om)$, and $X \in \symplecto(M, \Om)$, then $T\phi(\hm_{f}) = \hm_{f\circ \phi}$ and, using \eqref{hllie}, \eqref{sompb} yields
\begin{align}\label{lhampb}
\begin{split}
\Phi^{\nabla\,\ast}(\sOmega)_{\phi}&(T\phi(X), \hm_{f\circ \phi}) = \Phi^{\nabla\,\ast}(\sOmega)_{\phi}(T\Phi(X), T\Phi(\hm_{f}))  = \sOmega_{\nabla}(\lop(X^{\sflat}), \hop(f)) \\ &= -\lb f, \hop^{\ast}\lop(X^{\sflat})\ra = - \lb f, \Om(X, \hm_{\K(\nabla)})\ra 
= -\lb f, d\K(\nabla)(X)\ra.
\end{split}
\end{align}
By \eqref{lhampb}, for $\phi \in \Symplecto(M \Om)$, $T_{\phi}\Symplecto(M, \Om)$ is contained in the $\sOm$-complement of $T_{\phi}\Ham(M, \Om)$ if and only if $d\K(\nabla)(X) = 0$ for all $X \in \symplecto(M, \Om)$, and this last condition holds if and only if $\K(\nabla)$ is constant. 
If $\phi \in \Ham(M, \Om)$, then \eqref{lhampb} yields $\Phi^{\nabla\,\ast}(\sOmega)_{\phi}(\hm_{f\circ \phi}, \hm_{g \circ \phi}) =  \lb f, \{g, \K(\nabla)\}\ra$, for $f, g \in \ham(M, \Om)$. Consequently, $\Ham(M, \Om)\cdot \nabla$ is isotropic if and only if $\{g, \K(\nabla)\} = 0$ for all $g \in \ham(M, \Om)$, and this last condition holds if and only if $\K(\nabla)$ is constant.

If $\dim = 2$, then, by \eqref{sompb}, \eqref{isotropyidentity1}, and \eqref{cdiv}, for $X, Y \in \symplecto(M, \Om)$, 
\begin{align}\label{sompb2}
\begin{split}
\Phi^{\nabla\,\ast}&(\sOmega)_{\phi}(T\Phi(X), T\Phi(Y)) = \sOmega_{\nabla}(\lop(X^{\sflat}), \lop(Y^{\sflat})) = \sOmega(X^{\sflat}, \lop^{\ast}\lop(Y^{\sflat})\ra \\ & =  -\tfrac{1}{2}\lb X^{\sflat}, \lie_{Y}\rf\ra = \lb X^{\sflat}, \K(\nabla)Y^{\sflat}\ra  = \lb \Om(X, Y), \K(\nabla)\ra.
\end{split}
\end{align}
By \eqref{sompb2}, if $\K(\nabla) = 0$ then $\Phi^{\nabla\, \ast}(\sOmega) = 0$, so $\Symplecto(M, \Om)\cdot\nabla$ is isotropic. 
If $\Symplecto(M, \Om)\cdot\nabla$ is isotropic, then $\lb \Om(X, Y), \K(\nabla)\ra = 0$ for all $X, Y \in \symplecto(M, \Om)$. In particular, $0 = \lb \{f, g\}, \K(\nabla)\ra = \lb f, \{g, \K(\nabla)\}\ra$ for every $f, g \in \cinf_{c}(M)$. This means $\{g, \K(\nabla)\} = 0$ for all $g \in \cinf_{c}(M)$ so that $\K(\nabla)$ is constant. 
\end{proof}
	
\begin{lemma}\label{sequencelemma}
\noindent
\begin{enumerate}
\item  Let $(M, \Om)$ be a symplectic $2n$-manifold and let $\nabla \in \symcon(M, \Om)$. The sequence
\begin{align}\label{ksequence}
0 \longrightarrow H^{0}(M; \rea) \stackrel{\inc}{\longrightarrow} \cinf(M) \stackrel{\hop}{\longrightarrow} \Ga(S^{3}(\ctm))\stackrel{\hop^{\ast}}{\longrightarrow}\cinf(M) \longrightarrow 0
\end{align}
is a complex if and only if $d\K(\nabla) = 0$.
\item Let $(M, \Om)$ be a symplectic $2$-manifold and let $\nabla \in \symcon(M, \Om)$. The sequence
\begin{align}\label{rfsequence}
0  \longrightarrow \symplecto(M, \Om)^{\sflat} \stackrel{\lop}{\longrightarrow} \Ga(S^{3}(\ctm))\stackrel{\lop^{\ast}}{\longrightarrow}\Ga(\ctm) \longrightarrow 0
\end{align}
is a complex if and only if $\rf(\nabla) = 0$.
\end{enumerate}
\end{lemma}

For compact $M$, $\hop^{\ast}$ takes values in $\cinf_{0}(M)$, so the second $\cinf(M)$ in \eqref{ksequence} can be replaced by $\cinf_{0}(M)$.

\begin{proof}
By \eqref{isotropyidentity2}, if \eqref{ksequence} is a complex then $\K(\nabla)$ Poisson commutes with every $f \in \cinf(M)$, so $d\K(\nabla) = 0$. On the other hand, by \eqref{isotropyidentity2}, if $d\K(\nabla) = 0$ then $\hop(\ham(M, \Om)) \subset \ker \hop^{\ast}$. By \eqref{isotropyidentity2d}, when $2n = 2$, the sequence \eqref{rfsequence} is a complex if $\rf(\nabla) = 0$. If \eqref{rfsequence} is a complex, then, by \eqref{isotropyidentity1}, $\lie_{X}\rf(\nabla) = 0$ for all $X \in \symplecto(M, \Om)$. For any $p \in M$ there can be chosen $f \in \cinf(M)$ such that $df$ vanishes at $p$ and $\nabla_{i}df^{j}$ is invertible at $p$. Hence, at $p$ there holds $0 = -(\lie_{\hm_{f}}\rf)_{i} = \rf_{p}\nabla_{i}df^{p} + df^{p}\nabla_{p}\rf_{i} = \rf_{p}\nabla_{i}df^{p}$, and, by the invertibility of $\nabla_{i}df^{j}$ at $p$, this means $\rf$ vanishes at $p$. Since this is true for every $p \in M$, $\rf(\nabla) = 0$. 
\end{proof}

The relation of the sequence \eqref{rfsequence} to the deformation sequence for projective structures on surfaces is discussed in Section \ref{cohomsection}. Compare Lemma \ref{curvliezerolemma} and Remark \ref{pdcremark} with Lemma \ref{sequencelemma}. One byproduct of this discussion is the formula \eqref{loplopxgeneral} for $\lop^{\ast}\lop(X^{\sflat})$ for an arbitrary (not necessarily symplectic) vector field $X$.

\section{Critical symplectic connections}\label{extremalsymplecticsection}
Recall that $\nabla \in \symcon(M, \Om)$ is \textit{critical} if it is a critical point with respect to compactly supported variations of the functional $\ex(\nabla)= \int_{M} \K(\nabla)^{2}\,\,\Om$. Recall that Theorem \ref{criticaltheorem} states: \textit{A symplectic connection $\nabla \in \symcon(M, \Om)$ is critical if and only if $\hop(\K(\nabla)) = 0$.} Theorem \ref{criticaltheorem} is the special case of Lemma \ref{criticallemma} where $\phi(t) = t^{2}$.

\begin{lemma}\label{criticallemma}
Let $(M, \Om)$ be a symplectic manifold and let $\phi \in C^{4}(\rea)$. The connection $\nabla \in \symcon(M, \Om)$ is a critical point of $\ex_{\phi}(\nabla) = \int_{M}\phi(\K(\nabla))\,\Omk{n}$ if and only if $\H(\phi^{\prime}(\K(\nabla)) = 0$.
\end{lemma}
\begin{proof}
Calculating the first variation $\vr_{\Pi}\ex_{\phi}(\nabla)$ along $\Pi \in T_{\nabla}\symcon(M, \Om)$ using Lemma \ref{hadjointlemma} yields
\begin{align}\label{varex1}
\begin{split}
\vr_{\Pi}\ex_{\phi} & = \lb \phi^{\prime}(\K(\nabla)), \hop^{\ast}(\Pi)\ra = \sOm_{\nabla}(\hop(\phi^{\prime}(\K(\nabla))), \Pi),
\end{split}
\end{align}
which yields the claim.
\end{proof}
The assumption that $\phi$ be $C^{4}$ is needed only so that $\hop(\phi^{\prime}(\K))$ be defined without further fussing. Note that $\ex_{\phi}$ is constant on an orbit of $\Symplecto(M, \Om)$ in $\symcon(M, \Om)$. 

\begin{remark}
Lemma \ref{criticallemma} shows that the Hamiltonian vector field on $\symcon(M, \Om)$ generated by $\emf_{\phi}$ is $-\hop(\phi^{\prime}(\K(\nabla))$. It follows from \eqref{isotropyidentity2} and Lemma \ref{criticallemma} that the functions $\emf_{\phi}$ and $\emf_{\psi}$ on $\symcon(M, \Om)$ commute with respect to the Poisson structure determined by $\sOm$; that is, $\{\emf_{\phi}, \emf_{\psi}\} = 0$.
\end{remark}

\subsection{}
Finding the critical points of $\emf$ with respect to the smaller class of variations having the form $\sd^{\ast}\al$ for $\al \in \Ga(S^{2}(\ctm))$ leads to the following notion. Define $\nabla \in \symcon(M, \Om)$ to be \textit{gauge critical} if it is a critical point of $\ex$ for all variations of the form $\sd^{\ast} \al$ for compactly supported $\al \in \Ga(S^{2}(\ctm))$. By definition a critical symplectic connection is gauge critical. Lemma \ref{gaugecriticallemma} shows that $\nabla$ is gauge critical if and only if the Hamiltonian vector field $H_{\K(\nabla)}$ generated by its moment map image preserves its Ricci tensor.

\begin{lemma}\label{gaugecriticallemma}
A symplectic connection $\nabla \in \symcon(M, \Om)$ is gauge critical if and only if $\lie_{\hm_{\K(\nabla)}}\ric = 0$.
\end{lemma}
\begin{proof}
By \eqref{varex1} and \eqref{tracedhf}, 
\begin{align}
\vr_{\sd^{\ast}\al}\ex_{\nabla} = 2\lb \hop(\K(\nabla)), \sd^{\ast}\al\ra = 2\lb \sd \hop(\K(\nabla)), \al\ra = 2\lb \lie_{\hm_{\K(\nabla)}}\ric , \al\ra, 
\end{align}
which yields the claim.
\end{proof}

\begin{lemma}
On a compact symplectic $2$-manifold $M$ of genus at least two, if a symplectic connection $\nabla$ is gauge critical and $\K(\nabla)$ is not identically zero, then the Ricci tensor of $\nabla$ is degenerate somewhere on $M$.
\end{lemma}

\begin{proof}
A compact orientable surface supporting a mixed signature metric must be a torus, so if the Ricci tensor is nondegenerate, it must be definite because of the hypothesis on the genus. Then, by Theorem \ref{criticaltheorem}, $\hm_{\K(\nabla)}$ is a Killing field for a Riemannian metric, and so by Bochner's theorem must be identically $0$. Hence $\K(\nabla)$ is constant, so zero, because $M$ is compact.
\end{proof}

\subsection{}
Because $\symcon(M, \Om)$ is affine it carries a canonical flat connection $\vr$. Via the corresponding parallel translation, elements $\al, \be \in \Ga(S^{3}(\ctm))$ can be viewed as constant, parallel vector fields on $\symcon(M, \Om)$. The Lie algebra structure on vector fields on $\symcon(M, \Om)$ is generated by its group of translations, and so, viewed as vector fields on $\symcon(M, \Om)$, $\al$ and $\be$ commute. The differential of $\F$ is $\vr \F$, and $\vr_{\al}\F$ is its evaluation on the constant vector field $\al$. The Hessian of a twice differentiable functional is well defined at a critical point without the imposition of any extra structure. Given the flat connection $\vr$, the Hessian of a functional $\F$ on $\symcon(M, \Om)$  is well defined everywhere as its Hessian with respect to $\vr$. The Hessian $\vr \vr \F$ at $\al$ and $\be$ is $\vr_{\al}(\vr \F)(\be) = \vr_{\al}\vr_{\be}\F - \vr_{\vr_{\al}\be}\F = \vr_{\al}\vr_{\be}\F$, where $\vr_{\vr_{\al}\be}\F$ vanishes because $\vr_{\al}\be = 0$. It follows that the Hessian of $\F$ at $\nabla \in \symcon(M, \Om)$ in the directions $\al$ and $\be$ is given by the first variation in the direction $\al$ of the first variation in the direction $\be$ of $\F$, and so can be written simply $\vr_{\al}\vr_{\be}\F(\nabla)$.

Define the \textit{Jacobi operator} $\jac:\Ga(S^{3}(\ctm)) \to \Ga(S^{3}(\ctm))$ associated with $\nabla \in \symcon(M, \Om)$ by
\begin{align}
\jac(\al) 
= \hop\hop^{\ast}(\al) + \lie_{\hm_{\K}}\al.
\end{align}
The name for $\jac$ is justified by the following observations.  Let $f \in \cinf(M)$. By the definition of $\hop(f)$ and the fact that $H_{f}$ commutes with raising and lowering indices there holds
\begin{align}\label{vrhf}
(\vr_{\al}\hop)(f) = \tfrac{d}{dt}_{|t = 0}\hop_{\nabla + t \al^{\ssharp}}(f) = \tfrac{d}{dt}_{| t= 0} \lie_{\hm_{f}}(\nabla + t\al^{\ssharp})^{\sflat} = \lie_{\hm_{f}}\al,
\end{align}
in which $(\al^{\ssharp})_{ij}\,^{k} = \al_{ij}\,^{k}$ and $\lie_{H_{f}}(\nabla + t\al^{\ssharp})^{\sflat}$ means $\lie_{H_{f}}(\nabla + t\al^{\ssharp})_{ijk}$. 
Let $\nabla(t)$ be a $C^{1}$ curve in the space of critical symplectic connections depending smoothly on $t$ and such that $\tfrac{d}{dt}_{t = 0}\nabla(t) = \al_{ij}\,^{k}$. By \eqref{kvary} and \eqref{vrhf}, differentiating $0 = \hop_{\nabla(t)}(\K(\nabla(t)))$ in $t$ yields $0 = \jac(\al)$, so that the kernel of $\jac$ describes the tangent space to the deformations of a critical symplectic connection through critical symplectic connections.

\begin{lemma}
The Hessian of $\ex$ at $\nabla \in \symcon(M, \Om)$ in the directions $\al, \be \in T_{\nabla}\symcon(M, \Om)$ is
\begin{align}\label{secondvariation}
 \vr_{\al}\vr_{\be}\ex(\nabla) = 2\lb \jac(\al), \be\ra 
= 2\lb \hop^{\ast}(\al), \hop^{\ast}(\be)\ra + 2\lb \lie_{\hm_{\K}}\al, \be\ra.
\end{align}
\end{lemma}
\begin{proof}
By \eqref{varex1} and \eqref{vrhf},
\begin{align*}
\begin{split}
\vr_{\al}\vr_{\be}\ex(\nabla) & = 2\vr_{\al}\lb \hop(\K(\nabla)), \be\ra  = 2\lb \hop\hop^{\ast}(\al), \be \ra + 2\lb (\vr_{\al}\hop)(\K(\nabla)), \be\ra  = 2\lb \jac(\al), \be\ra 
\end{split}
\end{align*}
An alternative proof is to expand $\ex(\nabla + t\Pi)$ using \eqref{kvary} and simplify the result as in \eqref{kvarypaired}.
\end{proof}

\begin{corollary}\label{kconstantcorollary}
Let $\nabla \in \symcon(M, \Om)$. If $\K(\nabla)$ is constant then $\vr_{\al}\vr_{\al}\ex(\nabla) \geq 0$ with equality if and only if $\hop^{\ast}(\al) = 0$. If $M$ is compact then a moment constant $\nabla$ is an absolute minimizer of $\ex$.
\end{corollary}
\begin{proof}
Because $\K(\nabla)$ is constant, the last term in \eqref{secondvariation} vanishes, so $\vr_{\al}\vr_{\al}\ex(\nabla) = \lb\hop^{\ast}(\al), \hop^{\ast}(\al)\ra \geq 0$, with equality if and only if $\hop^{\ast}(\al) = 0$. The final claim holds because, if $M$ is compact, then a moment constant connection is moment flat.
\end{proof}

\begin{lemma}
For $X \in \symplecto(M, \Om)$ and $\al \in T_{\nabla}\symcon(M, \Om)$, $\jac \lop(X^{\sflat}) = \lie_{X}\hop(\K)$, where $\K = \K(\nabla)$. In particular, if $\nabla$ is critical then $\vr_{\al}\vr_{\lop(X^{\sflat})}\ex(\nabla) = 0$ for all $X \in \symplecto(M, \Om)$ and $\al \in T_{\nabla}\symcon(M, \Om)$.
\end{lemma}
\begin{proof}
For $X \in \symplecto(M, \Om)$ and $f \in \cinf(M)$,
\begin{align}\label{hmfx}
\begin{split}
\sd^{\ast}(df(X)) & = -d(df(X)) = -\lie_{X}df = [X, \hm_{f}]^{\sflat},
\end{split}
\end{align}
where the final equality holds because $X$ is symplectic. 
Because $X$ and $\hm_{\K}$ are symplectic, by \eqref{hmfx},
\begin{align}\label{lielop}
\begin{split}
\lie_{\hm_{\K}}\lop(X^{\sflat}) &= \lie_{\hm_{\K}}(\lie_{X}\nabla)^{\sflat} = (\lie_{\hm_{\K}}\lie_{X}\nabla)^{\sflat} = (\lie_{[\hm_{\K}, X]}\nabla)^{\sflat} + (\lie_{X}\lie_{\hm_{\K}})^{\sflat} \\
&= \lop([\hm_{\K(\nabla)}, X]^{\sflat}) + \lie_{X}\hop(\K) = -\lop\sd^{\ast}(d\K(X)) + \lie_{X}\hop(\K) \\
&= -\hop(d\K(X)) + \lie_{X}\hop(\K). 
\end{split}
\end{align}
By \eqref{hllie}, $\hop^{\ast} \lop(X^{\sflat}) = d\K(X)$, and combining this with \eqref{lielop} yields $\jac\lop(X^{\sflat}) = \hop\hop^{\ast}\lop(X^{\sflat}) + \lie_{\hm_{\K(\nabla)}}\lop(X^{\sflat}) = \lie_{X}\hop(\K)$. The final claim follows from \eqref{secondvariation}.
\end{proof}

\section{On a surface, a preferred symplectic connection is critical}\label{preferredsection}

Recall from the introduction that $\nabla \in \symcon(M, \Om)$ is \textit{preferred} if $-\sd^{\ast}\ric_{ijk} = \nabla_{(i}R_{jk)} = 0$. Theorem \ref{preferredtheorem} states that on surface a preferred symplectic connection is critical.
\begin{proof}[Proof of Theorem \ref{preferredtheorem}]
Let $(M, \Om$) be $2$-dimensional. Let $\nabla \in \symcon(M, \Om)$. Since $\rf  = 2\sd \ric$, by \eqref{sdsdast}, 
\begin{align}\label{sdastrf}
\sd^{\ast}\rf = 2\sd^{\ast}\sd\ric = -3\sd\sd^{\ast}\ric. 
\end{align}
Since $-2\K = \sd \rf$, by \eqref{sdsdast}, 
\begin{align}\label{dki}
2d\K = -2\sd^{\ast}\K = \sd^{\ast}\sd\rf = -2\sd\sd^{\ast}\rf + 2R_{ip}\rf^{p}.
\end{align}
By \eqref{sddecompose},
\begin{align}\label{divw1}
&\nabla\rf = -\sd^{\ast}\rf -\K\Om,& &3\nabla_{i}R_{jk} = -3\sd^{\ast}\ric_{ijk} - \Om_{i(j}\rf_{k)}.
\end{align}
Since $\sd d\K = 0$, by \eqref{dki}, $\sd^{2}\sd^{\ast}\rf = \nabla_{p}(R_{q}\,^{p}\rf^{q}) = R^{pq}\sd^{\ast}\rf_{pq}$. 
Differentiating \eqref{dki} and simplifying the result using \eqref{divw1}, the identity $\sd^{2}\sd^{\ast}\rf = R^{pq}\sd^{\ast}\rf_{pq}$, and \eqref{sdsdast} yields
\begin{align}\label{nabladkfull}
\nabla_{i}d\K_{j} = \tfrac{1}{6}\rf_{i}\rf_{j} - \K R_{ij}  - \rf^{p}\sd^{\ast}\ric_{ijp} + 4R_{(i}\,^{p}\sd^{\ast}\rf_{j)p} - \tfrac{3}{2}\sd\sd^{\ast\,2}\rf_{ij}.
\end{align}
Differentiating \eqref{nabladkfull} yields an explicit expression for $\hop(\K)$. Without some further hypothesis, \eqref{nabladkfull} is too complicated to be useful, but it simplifies considerably if $\sd^{\ast}\rf$ or $\sd^{\ast}\ric$ vanishes. 

By \eqref{sdastrf}, if $\nabla \in \symcon(M, \Om)$ is preferred, then $\sd^{\ast}\rf = 0$, so \eqref{nabladkfull} simplifies to 
\begin{align}\label{nabladk}
\nabla_{i}d\K_{j} = \tfrac{1}{6}\rf_{i}\rf_{j} - \K R_{ij}.
\end{align}
Differentiating \eqref{nabladk} shows $\nabla_{i}\nabla_{j}d\K_{k} = -d\K_{i} R_{jk}$, so that $\hop(\K) = 0$.
\end{proof}

\begin{remark}
The essential content of the conclusion of Theorem \ref{preferredtheorem}, that for a preferred symplectic connection the vector field $\hm_{\K(\nabla)}$ is an infinitesimal affine automorphism of $\nabla$, is implicit in the proof of Proposition $6.1$ of \cite{Bourgeois-Cahen}. In particular, in sections $5$ and $6$ of \cite{Bourgeois-Cahen}, in the context of preferred symplectic connections, a key role is played by a function $\be$ and one-form $u$ that are constant multiples of the specializations to this case of $\K(\nabla)$ and $\rf$. For example, the identity  $3\nabla_{i}R_{jk} = -\Om_{i(j}\rf_{k)}$ is in \cite{Bourgeois-Cahen}, and $(5.6)$ of \cite{Bourgeois-Cahen} is equivalent to \eqref{nabladk}. 
\end{remark}

\section{Structural lemmas for symplectic connections on surfaces}\label{localstructuresection}
Lemma \ref{extremalstructurelemma} shows that given a critical symplectic connection that is not moment constant there is a large Darboux coordinate chart on which $\K$ can be considered the action coordinate of action-angle variables. This information is useful both for proving nonexistence of such connections in Theorem \ref{noextremaltheorem} and for constructing examples.

For $f \in \cinf(M)$ let $\crit(f) = \{p \in M: df_{p} = 0\}$ be the set of the critical points of $f$. For $\nabla \in \symcon(M, \Om)$ let $\zero(\rf)$ be the set of points of $M$ where $\rf$ vanishes and let $\degen(\nabla)$ be the set of points in $M$ where $d\K \wedge \rf$ vanishes. Note that $\degen(\nabla) \supset \crit(\K)\cup \zero(\rf)$.

\begin{lemma}\label{extremalstructurelemma}
Let $(M, \Om)$ be a symplectic $2$-manifold and let $\nabla \in \symcon(M,\Om)$ be critical symplectic with moment map $\K = \K(\nabla)$. 
\begin{enumerate}
\item If $\K$ is not constant then each connected component of $\crit(\K)$ is a point or a closed image of a geodesic of $\nabla$, and these connected components are isolated in the sense that around each there is an open neighborhood containing no other connected component of $\crit(\K)$. 
\item There is a constant $\vc$ such that
\begin{align}\label{exvc}
\K^{2} + \rf^{i}d\K_{i} = \vc.
\end{align}
Consequently $\degen(\nabla) = \{p \in M: \K(p)^{2} = \tau\}$ and: 
\begin{enumerate}
\item $\degen(\nabla) = M$ if and only if $\K$ is constant.
\item If $\degen(\nabla)$ is empty then $M$ is noncompact.
\item If $\vc$ is negative then $\degen(\nabla)$ is empty.
\item If $\degen(\nabla)$ is nonempty, $\vc$ is nonnegative and $|\K|$ equals $\sqrt{\vc}$ at every point of $\degen(\nabla)$. 
\item A maximal integral curve of the symplectically dual vector field $\rf^{\ssharp\,i} = \rf^{i}$ intersecting $\degen(\nabla)$ lies entirely in $\degen(\nabla)$ and $\K$ is constant along the curve, equal to one of $\pm \sqrt{\vc}$.
\item If $\vol_{\Om}(M) < \infty$, $\emf(\nabla) < \infty$, and $d(\rf(\hm_{\K}))$ is integrable then 
\begin{align}\label{vcval}
\vc \vol_{\Om}(M) = 3\emf(\nabla) + \int_{M}d(\K\rf).
\end{align}
In particular, if $M$ is compact then $\vc \vol_{\Om}(M) = 3\emf(\nabla)$.
\end{enumerate}
\item\label{siholo} 
If $\K$ is not constant, then on the nonempty open subset $\bar{M} = M \setminus \degen(\nabla)$ there hold: 
\begin{enumerate}
\item The one-form $\si = (\vc - \K^{2})^{-1}\rf$ is closed; 
\item $d(\K\si) = d\K \wedge \si = \Om$;
\item The symplectically dual vector field $\si^{\ssharp\,i} = \si^{i}$ commutes with $\hm_{\K}$ and $\Om(\si^{\ssharp}, \hm_{\K}) = 1$;
\item \label{extremallocallemma}
Around any point $p \in M \setminus \crit(\K)$, 
there are local coordinates $(x,y)$ such that $p$ corresponds to the origin, $\K - \K(p) = x$, $dy = (\tau - \K^{2})^{-1}\rf$, and $\Om = dx \wedge dy$.
\item The complex structure $J$ defined on $\bar{M}$ by $J(\si^{\ssharp}) = \hm_{\K}$ and $J(\hm_{\K}) =  - \si^{\ssharp}$ is $\hm_{\K}$ and $\si^{\ssharp}$ invariant;
\item\label{parabolic}
If $\vc > 0$, the function 
\begin{align}
T = \frac{1}{2\sqrt{\vc}}\log\left|\frac{\sqrt{\vc} + \K}{\sqrt{\vc} - \K}\right| = \begin{cases} \vc^{-1/2}\arctanh(\vc^{-1/2}\K) & \text{if}\,\,\, \vc > \K^{2},\\
\vc^{-1/2}\arccoth(\vc^{-1/2}\K) & \text{if}\,\,\, 0 < \vc < \K^{2},
\end{cases}
\end{align}
is smooth on $\bar{M}$, while if $\vc < 0$ the function
\begin{align}
T = |\vc|^{-1/2}\arccot(|\vc|^{-1/2}\K)
\end{align}
is smooth on all of $M$. With the complex structure $J$ defined by $J(\hm_{T}) = \rf^{\ssharp}$ and $J(\rf^{\ssharp}) = -\hm_{T}$ the flat Riemannian metric
\begin{align}\label{flatmetric}
k_{ij} = dT_{i}dT_{j} + \si_{i}\si_{j}
\end{align}
forms a Kähler structure with volume $(\vc - \K^{2})^{-1}\Om = dT \wedge \si = d(T\si)$. Moreover, $\hm_{\K}$ is a Killing field for $k_{ij}$.
\item\label{completeflat} If $\degen(\nabla)$ is nonempty, then $T$ maps each maximal integral manifold of the vector field $\rf^{\ssharp}$ on $\bar{M}$ diffeomorphically onto its image. Precisely, if $\phi:I \to M$ is a maximal integral curve of $\rf^{\ssharp}$ such that $\phi(0) = p \in \bar{M}$ then $T\circ \phi(t) = t + T(p)$.
\item\label{confflat} If $\degen(\nabla)$ is nonempty, and the restriction to $\bar{M}$ of $\rf^{\ssharp}$ is complete, then the metric $k$ of \eqref{flatmetric} is complete on $\bar{M}$, and each connected component of $\bar{M}$ with the Riemann surface structure $(k, J)$ is conformally equivalent to the complex plane or the punctured disk. 
\end{enumerate}
\end{enumerate}
\end{lemma}

\begin{proof}
Each connected component of the zero set of an infinitesimal automorphism of a torsion-free affine connection is a closed totally geodesic submanifold (see p.~61 of \cite{Kobayashi-Nomizu-2}). In particular this applies to the zeros of $\hm_{\K}$, that is to $\crit(\K)$. If a connected component of $\crit(\K)$ has nonempty interior then it must be all of $M$, in which case $\K$ is constant. Hence if $\K$ is nonconstant each connected component of $\crit(\K)$ is a closed totally geodesic submanifold of $M$ of codimension at least one, so is a closed image of a geodesic or a point. Suppose $\K$ is nonconstant. Suppose $x \in \crit(\K)$ and every sufficiently small neighborhood of $x$ contains some point of $\crit(\K)$ distinct from $x$. Fix a geodesically convex open neighborhood $U$ of $x$ and choose $y$ in $\crit(\K)\cap U$ distinct from $x$. Let $L$ be the unique geodesic segment connecting $x$ to $y$ and contained in $U$. Since the flow of $\hm_{\K}$ fixes $x$ and $y$ it fixes $L$ too, and so $L \subset \crit(\K)$. Hence $x$ and $y$ belong to the same connected component of $\crit(\K)$. It follows that the connected components of $\crit(\K)$ are isolated.

Since the flow of $\hm_{\K}$ preserves $\nabla$ and $\Om$ it preserves $\rf$, so, using \eqref{cdiv},
\begin{align}
\begin{split}
0 &= \lie_{\hm_{\K}}\rf = \imt(\hm_{\K})d\rf + d(\rf(\hm_{\K})) = 2\K d\K + d(\rf(\hm_{\K})) = d(\K^{2} + \rf(\hm_{\K})),
\end{split}
\end{align}
from which it follows that there is a constant $\vc$ satisfying \eqref{exvc}. From \eqref{exvc} it is apparent that $\degen(\nabla)$ is the zero locus of $\vc - \K^{2}$. In particular, $\K$ is constant if and only if $\degen(M) = M$.  If $\degen(\nabla)$ is empty, then $d\K$ does not vanish, so $M$ must be noncompact.

By \eqref{exvc}, if $q \in \degen(\nabla)$ then $\K(q)^{2} = \vc$, so if $\degen(\nabla)$ is nonempty, then $\vc$ is nonnegative and, at every point of $\degen(\nabla)$, $|\K|$ equals $\sqrt{\vc}$. Let $q \in \degen(\K)$ and let $I\subset \rea$ be a maximal open interval around $0$ such that $\phi:I\to M$ is a smooth integral curve of $\rf^{\ssharp}$ satisfying $\phi(0) = q$. By \eqref{exvc}, $u(t) = \K \circ \phi(t)$ solves $\dot{u} = \vc - u^{2}$. If $\phi(0) = q \in \degen(\nabla)$ then $u(0) = \K(q) = \pm \sqrt{\vc}$ and so, by the uniqueness of the solution to the initial value problem for $\dot{u} = \vc - u^{2}$, $\K\circ \phi(t) = u(t) = \pm \sqrt{\vc}$ for all $t \in I$, and hence $\phi(I) \subset \degen(\nabla)$.

If $\vol_{\Om}(M)$ is finite and $\emf(\nabla)$ is finite, then integrating \eqref{exvc} and using \eqref{cdiv} yields 
\begin{align}
\begin{split}
\vc \vol_{\Om}(M) & = \int_{M}\K^{2}\Om + \int_{M}\rf^{p}d\K_{p}\Om = \emf(\nabla) + \int_{M}d\K \wedge \rf \\
& = \emf(\nabla) - \int_{M}\K\wedge d\rf + \int_{M}d(\K\rf) = 3\emf(\nabla) + \int_{M}d(\K\rf),
\end{split}
\end{align}
where the last expression makes sense provided $d(\K\rf)$ is integrable. This shows \eqref{vcval}.

That $\si = (\vc - \K^{2})^{-1}\rf$ is closed on $\bar{M}$ follows from differentiating \eqref{exvc} and using \eqref{cdiv} and 
\begin{align}\label{exvc2}
d\K \wedge \rf = (\vc - \K^{2})\Om. 
\end{align}
That $d\K \wedge \si = \Om$ is immediate from \eqref{exvc2}. Since $\lie_{\hm_{\K}}\rf = 0$, the vector fields $\rf^{i}$ and $\hm_{\K}$ commute. Since any function of $\K$ is constant along the flow of $\hm_{\K}$, this implies $\si^{i}$ commutes with $\hm_{\K}$. Because $\si^{\ssharp}$, $\hm_{\K}$, and $\Om$ are all invariant under $\hm_{\K}$ and $\si^{\ssharp}$ so too is the complex structure $J$.

If $\nabla$ is critical but not moment constant, then on the universal cover of each connected component of $\bar{M}$ there are coordinates $x$ and $y$ such that $\Om = dx \wedge dy$, $\K = x$, and $\rf = (\vc - x^{2})^{-1}dy$. Simply define $x = \K$ and let $y$ be a global primitive of $\si$. This proves \eqref{extremallocallemma}.

The claims of \eqref{parabolic} are all verified by straightforward computations. Now suppose $\degen(\nabla)$ is a nonempty proper subset of $M$, so that $\vc > 0$. Given $p \in \bar{M}$ let $\phi:I = (-a, b) \to M$ be a maximal integral curve of $\rf^{\ssharp}$ such that $\phi(0) = p$. Since $dT_{p}\rf^{p} = (\vc - \K^{2})^{-1}d\K_{p}\rf^{p} = 1$, the function $u(t) = T\circ \phi$ satisfies $u(0) = T(p)$ and $\dot{u} = 1$ so $u(t) = t + T(p)$. Consequently $T$ maps $\phi(I)$ diffeomorphically onto its image $(T(p) - a, T(p) + b)$. 

If $\rf^{\ssharp}$ is complete then its flow $\phi_{t}$ is globally defined.
Let $C$ be a connected component of $\bar{M}$. Then $C$ contains a connected component $S$ of $K^{-1}(0)$ and, because $T \circ \phi_{t}(p) = t$ for $p \in \K^{-1}(0)$, the map $\Phi:\rea \times S \to C$ defined by $\Phi(t, p) = \phi_{t}(p)$ is a diffeomorphism. Since $S$ is a connected one-manifold it is diffeomorphic to a circle $\onesphere$ or the line $\rea$. Since $\phi_{t}^{\ast}(\si) = \si$, the pullback $\Phi^{\ast}(\si)$ is a closed one-form on $S$ extended trivially to $\rea \times S$, so is exact if $S$ is a line, or is a multiple of the generator $d\theta$ of the cohomology of $\onesphere$ if $S$ is a circle. It follows that the pullback via $\Phi$ of the metric $k$ of \eqref{flatmetric} is the flat metric on the plane $\rea \times \rea$ or on the infinite Euclidean cylinder $\rea \times \onesphere$. The completeness of $k$ follows. In this case, each connected component of $M \setminus\degen(\K)$ carries a complete flat Kähler structure and a nontrivial holomorphic vector field preserving this Kähler structure. A Riemann surface with biholomorphism group that is not discrete is biholomorphic to one of the following: the Riemann sphere, the plane, the punctured plane, a torus, the unit disc, the punctured unit disc, or an annulus of radii $r < 1$ and $1$. Of these the only surfaces covered holomorphically by the plane are the plane, torus, and punctured disk. Since $\K$ is not constant and $\degen(\nabla)$ nonempty, the torus cannot occur as the complement of $\degen(\nabla)$.
\end{proof}

\begin{remark}
Let $M$ and $\nabla$ be as in Lemma \ref{extremalstructurelemma}. Then the metric
\begin{align}\label{vcmetric}
h_{ij} = |\vc - \K^{2}|^{-1}d\K_{i}d\K_{j} + |\vc - \K^{2}|\si_{i}\si_{j} = |\vc - \K^{2}|k_{ij}
\end{align} 
on $\bar{M} = M \setminus \degen(\nabla)$ has constant scalar curvature $\sR_{h} = 2$ where $\vc > \K^{2}$ and constant scalar curvature $\sR_{H} = -2$ where $\vc < \K^{2}$, its volume form $\vol_{h}$ consistent with the orientation given by $\Om$ equals $\Om$, and $\hm_{\K}$ is a Killing field for $h$. However the restriction of $h$ to a connected component of $\bar{M}$ is generally not complete, so it is not clear how useful $h$ is.

These claims can be proved as follows. Suppose $\degen(\nabla)$ is nonempty and fix $p \in \bar{M} = M \setminus \degen(\nabla)$. Since $\si$ is closed, there is a neighborhood of $p$ on which there is a smooth function $\phi$ such that $d\phi = \sqrt{\vc}\si$. If $p$ is contained in a connected component of $\bar{M}$ on which $\vc - \K^{2}$ is positive, then, since $\K \in [-\sqrt{\vc}, \sqrt{\vc}]$, there can be defined on this neighborhood a smooth function $\theta$ such that $K = -\sqrt{\vc}\cos\theta$. In the coordinates $(\theta, \phi)$ around $p$ the metric $h$ takes the form $d\theta^{2} + \sin^{2}\theta d\phi^{2}$, which is one of the well known standard forms of the spherical metric. This shows $h$ has constant scalar curvature $2$. Its volume form is $\sin \theta d\theta \wedge d\phi = d\K \wedge \si = \Om$. If $p$ is contained in a connected component of $\bar{M}$ on which $\vc - \K^{2}$ is negative, then setting $\K = \pm\sqrt{\vc}\cosh \theta$ and $d\phi = \pm \sqrt{\vc}\si$ as $\K$ is positive or negative, there result $h = d\theta^{2} + \sinh^{2}\theta d\phi^{2}$, a standard form of the hyperbolic metric, and $\Om = d\K \wedge \si = \sinh \theta d\theta \wedge d\phi$. Since $\hm_{\K}$ preserves $K$, $d\K$, $\Om$, and $\si$, it preserves $h$. 
\end{remark}

\begin{remark}
Since, by Theorem \ref{preferredtheorem}, preferred symplectic connections are critical, it makes sense to speak of the special cases of Lemmas \ref{extremalstructurelemma} and \ref{compactextremalstructurelemma} for preferred symplectic connections. 
For preferred symplectic connections, Lemma \ref{extremalstructurelemma} specializes to the results of sections $5$ and $6$ of \cite{Bourgeois-Cahen}, and \eqref{extremallocallemma} of Lemma \ref{extremalstructurelemma} specializes to Proposition $11.4$ of \cite{Bourgeois-Cahen}.
\end{remark}

Applying Lemma \ref{compactextremalstructurelemma} yields Lemma \ref{extremalstructurelemma}, showing that on a compact surface carrying a critical symplectic connection the complement of the set of critical points of $\K$ is a union of parabolic Riemann surfaces.

\begin{lemma}\label{compactextremalstructurelemma}
Let $(M, \Om)$ be a compact symplectic $2$-manifold and let $\nabla \in \symcon(M,\Om)$ be critical symplectic with moment map $\K = \K(\nabla)$. 
\begin{enumerate}
\item There holds $\zero(\rf) \subset \crit(\K) = \degen(\nabla)$ and every critical point of $\K$ is a global extremum at which $|\K|$ equals $\sqrt{\vc}$. In particular, $\vc = 0$ if and only if $\K$ is identically zero.
\item If $\K$ is not constant, then: 
\begin{enumerate}
\item Each $r \in (-\sqrt{\vc}, \sqrt{\vc})$ is a regular value of $\K$ and the level set $\lc_{r}(\K) = \{q \in M: \K(q) = r\}$ is a disjoint union of smoothly embedded circles.
\item For $r, s \in (-\sqrt{\vc}, \sqrt{\vc})$, the level sets $\lc_{r}(\K)$ and $\lc_{s}(\K)$ are diffeomorphic.
\item\label{concom} The number of connected components of $\crit(\K)$ where $\K$ assumes its minimum equals the number of connected components of $\crit(\K)$ where $\K$ assumes its maximum.
\item\label{cylinder} Each connected component of $\bar{M} = M \setminus \crit(\K)$ is diffeomorphic to the infinite Euclidean cylinder $\rea \times \onesphere$, and the metric $k_{ij}$ of \eqref{flatmetric} is complete, isometric to the standard flat metric on $\rea \times \onesphere$.
\end{enumerate}
\end{enumerate}
\end{lemma}

\begin{proof}
Because $M$ is compact, $\K$ has a maximum and a minimum, at which it takes the values $\pm \sqrt{\vc}$. Since it takes these same values at any point in $\degen(\nabla)$, any point of $\degen(\nabla)$ is a global extremum. In particular $\degen(\nabla) = \crit(\K)$.
If $\K$ is not constant then for $r \in (-\sqrt{\vc}, \sqrt{\vc})$, by \eqref{exvc}, along $\lc_{r}(\K)$ there holds $\rf^{p}d\K_{p} = \vc - r^{2}> 0$, so that $d\K$ is nonvanishing along $\lc_{r}(\K)$. This suffices to show that $\lc_{r}(\K)$ is a smoothly embedded one-dimensional submanifold. Since it is also closed, it must be a union of circles. If $-\sqrt{\vc} < r < s < \sqrt{\vc}$ then $\K^{-1}([r, s])$ is compact and contains no critical point of $\K$, so $\lc_{r}(\K)$ is diffeomorphic to $\lc_{s}(\K)$ (see Theorem $3.1$ of \cite{Milnor-morse}); moreover, the flow of $\rf^{\ssharp}$ maps $\lc_{r}(\K)$ diffeomorphically onto $\lc_{s}(\K)$. The claim about the equality of the numbers of components of $\crit(\K)$ on which $\K$ assumes its minimum and maximum follows.
Since $M$ is compact, $\rf^{\ssharp}$ is complete, and so \eqref{cylinder} follows from \eqref{confflat} of Lemma \ref{extremalstructurelemma}.
\end{proof}

Note that in Lemma \ref{compactextremalstructurelemma} the compactness is used in part to guarantee the existence of a global flow for $\rf^{\ssharp}$ and several of the conclusions are valid in greater generality provided such a global flow exists. The compactness is used in a more essential way to invoke the regular interval theorem to conclude that the level sets of $\K$ are unions of circles.

With Lemma \ref{compactextremalstructurelemma} in hand the proof of Theorem \ref{noextremaltheorem} is straightforward.

\begin{proof}[Proof of Theorem \ref{noextremaltheorem}]
Suppose $\nabla$ is critical and not moment flat. It will be shown that $M$ must be a sphere or a torus. As in the proof of Lemma \ref{compactextremalstructurelemma},  $\K^{-1}(I_{\ep})$ is a disjoint union of cylinders, where $I_{\ep} = [-\sqrt{\vc} + \ep, \sqrt{\vc}- \ep])$. Each of the cylinders constituting $\K^{-1}(I_{\ep})$ has a positive end and a negative end, as $\K$ tends to a maximum or a minimum at the end. For a suitably small $\ep$, the complement $M \setminus \K^{-1}(I_{\ep})$ is a disjoint union of tubular neighborhoods of the connected components of $\crit(\K)$; precisely, its connected components are cylindrical bands around the closed geodesics contained in $\crit(\K)$ and disks around the points contained in $\crit(\K)$. Each of these cylinders and disks is positive or negative as $\K$ has on the connected component of $\K$ that it contains a maximum or minimum. It follows that $M$ is the union of cylinders attached end to end and disks attached to cylinders in a manner compatible with the assignment of signs to the cylinders and disks. 
A closed compact surface obtained by attaching cylinders end to end and cylinders to disks has nonnegative Euler characteristic; since $M$ is oriented, it must be a sphere or a torus.

An alternative way of reaching the conclusion about the structure of the complement of $\crit(\K)$ goes as follows. By \eqref{parabolic} of Lemma \ref{compactextremalstructurelemma}, each connected component of $M \setminus\crit(\K)$ carries a parabolic complex structure $J$ and a nontrivial holomorphic vector field preserving $J$. A Riemann surface with biholomorphism group that is not discrete is biholomorphic to one of the following: the Riemann sphere, the plane, the punctured plane, a torus, the unit disc, the punctured unit disc, or an annulus of radii $r < 1$ and $1$. Of these the only parabolic surfaces are the plane, torus, and punctured disk. If $\crit(\K)$ is nonempty, closed compact surfaces cannot be obtained as connected components of its complement, and the plane is not the complement of a set containing at least two connected components. The remaining option is the punctured disk which is conformally equivalent to a cylinder, and the rest of the argument is as before.
\end{proof}

There remains the question of whether there exist on the torus or sphere critical symplectic connections that are not moment flat.
It was shown in \cite{Bourgeois-Cahen} that a preferred symplectic connection on a compact surface must be locally symmetric. A key point in the argument is to show that $\K(\nabla)$ must have a nondegenerate critical point; this uses in an essential way the simplification of \eqref{nabladkfull} to \eqref{nabladk} available in this case, and it is not clear if this argument can be adapted to the more general setting of critical symplectic connections considered here. 

\section{Symplectic connections in Darboux coordinates}\label{darbouxsection}
Claim \eqref{extremallocallemma} of Lemma \ref{extremalstructurelemma} motivates calculating $\hop(\K)$ explicitly for a symplectic connection in Darboux coordinates. 

\begin{lemma}\label{examplelemma}
Let $x$ and $y$ be global affine coordinates on $\rea^{2}$ with respect to the standard flat affine connection $\pr$ and let $\Om = dx\wedge dy$ be the standard symplectic form. Write $X = \pr_{x}$ and $Y = \pr_{y}$ for the coordinate vector fields (partial derivatives with respect to $x$ and $y$ are indicated by subscripts). The most general $\nabla \in \symcon(\rea^{2},\Om)$ has the form $\nabla = \pr + \Pi$ where 
\begin{align}\label{piabcd}
&\Pi(X, X) = AX + BY,& &\Pi(X, Y) = \Pi(Y, X) = -DX - AY,& &\Pi(Y, Y) = CX + DY,
\end{align}
for some $A, B, C, D \in \cinf(\rea^{2})$. For $\nabla \in \symcon(\rea^{2}, \Om)$ of the form $\nabla = \pr + \Pi$ with $\Pi$ as in \eqref{piabcd}, 
\begin{align}\label{rfreal2}
\begin{split}
-\tfrac{1}{2}\rf & = 
\left(-2A_{xy} - B_{yy} - D_{xx} + (AD - BC)_{x} + 3(A^{2} - BD)_{y} +  2AD_{x} - DA_{x} - BC_{x}\right)dx\\
&\,\, + \left(2D_{xy} + A_{yy} + C_{xx}  + 3(AC - D^{2})_{x} - (AD-BC)_{y} - 2DA_{y} + A D_{y} + C B_{y}\right)dy,
\end{split}
\end{align}
\begin{align}\label{krea2}
\begin{split}
\K(\nabla) &  =  3A_{xyy} + 3D_{xxy}  + B_{yyy} + C_{xxx} - B_{x}C_{y} + B_{y}C_{x} + 3(A_{x}D_{y} - A_{y}D_{x}) \\
&\qquad + 3(AC- D^{2})_{xx} + 3(BD - A^{2})_{yy} - 3(AD - BC)_{xy},
\end{split}
\end{align}
and, writing $\K = \K(\nabla)$,
\begin{align}\label{hk}
\begin{split}
-\hop(\K)& =  \left(\K_{xxx} - 3A\K_{xx} - 3B\K_{xy} + \K_{x}B_{y} - B_{x}\K_{y}\right) dx^{\tensor\,3}\\
 &+\left(\K_{yyy} - 3D\K_{yy} - 3C\K_{xy} + \K_{y}C_{x} - C_{y}\K_{x}\right)dy^{\tensor\,3}\\
& +3\left(\K_{xxy} + 2D\K_{xx} - B\K_{yy} + A\K_{xy} + A_{x}\K_{y} - A_{y}\K_{x}\right)dx\odot dx \odot dy\\
 &+3\left(\K_{xyy} + 2A\K_{yy} - C\K_{xx} + D\K_{xy} + D_{y}\K_{x} - D_{x}\K_{y}\right)dx\odot dy \odot dy,
\end{split}
\end{align}
where $\odot$ denotes the symmetrized tensor product, so, for example, $2dx\odot dy = dx\tensor dy + dy\tensor dx$.
\end{lemma}

\begin{proof}
By \eqref{cdiv}, differentiating \eqref{rfreal2} yields \eqref{krea2} (several terms cancel), so it suffices to check \eqref{rfreal2}. Routine computation shows that the Ricci tensor of $\nabla$ is
\begin{align}\label{exric}
\begin{split}
\ric = & \left(A_{x} + B_{y} + 2(BD - A^{2})\right)dx\tensor dx + \left(C_{x} + D_{y} + 2(AC - D^{2})\right)dy \tensor dy\\
& + 2\left(-A_{y}-D_{x} + AD - BC\right)dx \odot dy. 
\end{split}
\end{align}
Since $X$ and $Y$ constitute a symplectic frame,
\begin{align}\label{exrff}
\begin{split}
-\tfrac{1}{2}\rf(Z) &=  (\nabla_{X}\ric)(Z, Y) - (\nabla_{Y}\ric)(Z, X),
\end{split}
\end{align}
for all $Z \in \Ga(TM)$. Routine computations show that
\begin{align}\label{exnablaric}
\begin{split}
(\nabla_{X}\ric)(X, X) =& A_{xx} + B_{xy} + 2(BD - A^{2})_{x} - 2A(A_{x} + B_{y})  + 2B(A_{y} + D_{x}) \\
  &\qquad  - 6ABD + 4A^{3} + 2B^{2}C,\\
(\nabla_{Y}\ric)(Y, Y) =& C_{xy} + D_{yy} + 2(AC - D^{2})_{y} + 2C(A_{y} + D_{x}) - 2D(C_{x} + D_{y})\\
  & \qquad -6ACD + 4D^{3} + 2BC^{2},\\
(\nabla_{X}\ric)(X, Y) =&  -A_{xy} - D_{xx} + (AD - BC)_{x} - B(C_{x} + D_{y}) + D(A_{x} + B_{y})\\
  & \qquad + 4BD^{2} - 2A^{2}D - 2ABC,\\
(\nabla_{Y}\ric)(X, X) =& A_{xy} + B_{yy} + 2(BD - A^{2})_{y} + 2D(A_{x} + B_{y}) - 2A(A_{y} + D_{x})\\
& \qquad + 4BD^{2} - 2A^{2}D - 2ABC,\\
(\nabla_{X}\ric)(Y, Y) =&  C_{xx} + D_{xy} + 2(AC - D^{2})_{x} - 2D(A_{y} + D_{x}) + 2A(C_{x} + D_{y})\\
  & \qquad + 4A^{2}C - 2AD^{2} - 2BCD,\\
(\nabla_{Y}\ric)(X, Y) =& -A_{yy} - D_{xy} + (AD - BC)_{y} + A(C_{x} + D_{y}) - C(A_{x} + B_{y}) \\
  & \qquad + 4A^{2}C - 2AD^{2} - 2BCD.
\end{split}
\end{align}
Substituting \eqref{exnablaric} in \eqref{exrff}, and observing that the terms involving no derivatives cancel yields
\begin{align}\label{rfreal2b}
\begin{split}
-\tfrac{1}{2}\rf = &\left(-2A_{xy} - B_{yy} - D_{xx} + (AD - BC)_{x} + 2(A^{2} - BD)_{y} \right. \\ &\qquad \left. - B(C_{x} + D_{y}) - D(A_{x} + B_{y}) + 2A(A_{y} + D_{x})\right)dx \\
& +\left( 2D_{xy} + C_{xx} + A_{yy} - (AD - BC)_{y} + 2(AC - D^{2})_{x}\right. \\ &\qquad \left. - 2D(A_{y} + D_{x}) + A(C_{x} + D_{y}) + C(A_{x} + B_{y})\right)dy.
\end{split}
\end{align}
Simplifying \eqref{rfreal2b} yields \eqref{rfreal2}.
From
\begin{align}\label{nabladx}
&\nabla dx = -Adx^{\tensor\,2} - Cdy^{\tensor\,2} + 2Ddx\odot dy,& 
&\nabla dy =  -Bdx^{\tensor\,2} - Ddy^{\tensor\,2} + 2Adx\odot dy, 
\end{align}
it follows straightforwardly that
\begin{align}\label{nabladk2}
\begin{split}
\nabla d\K = \,&(\K_{xx} - A\K_{x} - B\K_{y})dx^{\tensor\,2} + (\K_{yy} - C\K_{x} - D\K_{y})dy^{\tensor\,2}\\
& + (\K_{xy} + D\K_{x} + A\K_{y})(dx \tensor dy + dy \tensor dx).
\end{split}
\end{align}
Using \eqref{nabladk}, \eqref{exric}, \eqref{nabladx}, and \eqref{nabladk2} to compute $\nabla^{2}d\K + d\K \tensor \ric$, and symmetrizing the result yields \eqref{hk}.
\end{proof}

\begin{remark}
Different simplifying assumptions on the functions $A$, $B$, $C$, and $D$ lead to different simplifications of the equations \eqref{krea2} that are amenable to study:
\begin{enumerate}
\item If $A$, $B$, $C$, and $D$ depend only on $x$, then $\K$ is an affine function of $x$. As is shown in the remainder of the present section, this leads to complete examples on $\rea^{2}$.
\item That $AD - BC = 0$, $A^{2} - BD = 0$, and $AC - D^{2} = 0$ are the conditions that the difference tensor $\Pi$ be the cube of a one-form $\si$. 
\begin{enumerate}
\item As is shown in section \ref{twospheresection}, the case where $\si$ is the metric dual of a metric Killing field leads to examples of critical symplectic connections that are not preferred.
\item As is shown in section \ref{momentflatexamplesection}, the case where $\si$ is closed gives rise to moment flat symplectic connections that are not projectively flat.
\end{enumerate} 
\end{enumerate}
These remarks suggest considering symplectic connections whose difference tensor with some nice metric connection has some other special form. In section \ref{momentflatexamplesection} and \eqref{representabilitysection} it is shown that when $\nabla = D + \Pi$ with $D$ the Levi-Civita connection of a constant curvature Riemannian metric $g$ and $\Pi_{ijk}$ the symmetric product of $g$ and a harmonic one-form leads to examples of moment flat symplectic connections for which $\rf(\nabla)$ represents an arbitrary de Rham cohomology class.
\end{remark}

Combining \eqref{extremallocallemma} of Lemma \ref{extremalstructurelemma} and Lemma \ref{examplelemma} yields equations for critical symplectic connections that can be solved. Let $\nabla \in \symcon(M, \Om)$ be critical with $\K$ nonconstant and suppose $p \notin \crit(\K)$. Let $\vc = \rf(\hm_{\K}) + \K^{2}$. By \eqref{extremallocallemma} of Lemma \ref{extremalstructurelemma}, in an open neighborhood of $p$ there may be taken as Darboux coordinates $x = \K - a$ where $a = \K(p)$, and $y$ where $dy = (\vc - \K^{2})^{-1}\rf = (\vc - (x + a)^{2})^{-1}\rf$. Hence $\K = x + a$ and $\rf = (\vc - (x + a)^{2})dy$. Let $\nabla$ have the form $\pr + \Pi$ where $\Pi$ is as in \eqref{piabcd} and $\pr$ is the standard flat affine connection in the coordinates $(x, y)$. By \eqref{hk},
\begin{align}
0 = -\hop(\K) = B_{y}dx^{\tensor 3} -3A_{y}dx\odot dx \odot dy + 3D_{y}dx\odot dy \odot y - C_{y}dy^{\tensor 3},
\end{align}
so that $A$, $B$, $C$, and $D$ depend only on $x$. This can be seen in another way as follows. The vector field $\hm_{\K} = \pr_{y}$ is Killing for both the flat affine connection $\pr$ and $\nabla$, so the Lie derivative along $\hm_{\K}$ of the difference tensor $\Pi = \nabla - \pr$ vanishes.

Comparing $\rf = (\vc - (x+a)^{2})dy$ and \eqref{rfreal2} shows that $A$, $B$, $C$, and $D$ satisfy the equations
\begin{align}\label{leneq}
&(C_{x} + 3(AC - D^{2}))_{xx} = x+ a,&&
D_{xx} = (AD - BC)_{x} + 2AD_{x} - DA_{x} -BC_{x}.
\end{align}
The preceding is summarized in Lemma \ref{extlemma}.
\begin{lemma}\label{extlemma}
Let $\pr$ be the standard flat affine connection in the coordinates $(x, y)$ and let $\Om = dx \wedge dy$. The connection $\nabla = \pr + \Pi$, where $\Pi$ is defined as in \eqref{piabcd}, is critical symplectic if $A$, $B$, $C$, and $D$ are functions of $x$ alone and satisfy the equations \eqref{leneq}. In this case $\K(\nabla) = x + a$ for some constant $a$.
\end{lemma}

Computing $\sd^{\ast}\ric$ and $\sd^{\ast}\rf$ using \eqref{exnablaric} and \eqref{nabladx} yields
\begin{align}\label{rfzerosdric}
\begin{split}
-3\sd^{\ast}\ric & = \left(A_{xx} + 2B_{x}D + 4BD_{x} - 3(A^{2})_{x} + 4A^{3} - 6ABD + 2B^{2}C \right)dx^{\tensor 3}\\
& + \left(2CD_{x} - 2C_{x}D - 6ACD + 2BC^{2} + 4D^{3}\right)dy^{\tensor 3}\\
& + \left(-2D_{xx} + 6A_{x}D - 4BC_{x} - 2B_{x}C - 6ABC + 12BD^{2} - 6A^{2}D\right)dx\odot dx \odot dy\\
& + \left(C_{xx} + 6AC_{x} - 3(D^{2})_{x} + 12A^{2}C - 6AD^{2} - 6BCD\right)dx\odot dy \odot dy,
\end{split}
\end{align}
\begin{align}
\sd^{\ast}\rf & = (\vc - (x + a)^{2})(Bdx^{\tensor 2} + D dy^{\tensor 2} - 2A dx \odot dy) + 2(x+a)dx\odot dy.
\end{align}
This suggests choosing $A$, $B$, and $D$ so that $\sd^{\ast}\rf = 0$. Then $B = D = 0$, $A =(x+a) (\vc - (x+a)^{2})^{-1}$, and, by \eqref{leneq}, $(C_{x} +3AC)_{xx} = x+a$. For any constants $p$ and $q$ the function $C = -6^{-1}(\vc - (x+a)^{2})((x + a)^{2} + p(x+a) + q)$ yields a solution defined at least on the region where $\vc \neq (x+a)^{2}$. Substituting into \eqref{rfzerosdric} yields $9\sd^{\ast}\ric = (\vc - q)dx\odot dy\odot dy$. If $q = \vc$ then the resulting connection is preferred. The family of connections corresponding to $q = \vc$ (with $p$ arbitrary) was found in Proposition $11.4$ of \cite{Bourgeois-Cahen}, where it is shown that every preferred symplectic connection on $\rea^{2}$ that is not symmetric is equivalent to one of these connections with $q = \vc < 0$ and $p$ arbitrary. On the other hand, if $q \neq \vc$, then the resulting connections are critical but not preferred.

\begin{theorem}\label{bcexampletheorem}
Let $\pr$ be the standard flat affine connection in the coordinates $(x, y)$ and let $\Om = dx \wedge dy$. For any choices of constants $a$, $p$, $q$, and $\vc$, the connection $\nabla = \pr + \Pi$ where $\Pi$ is defined as in \eqref{piabcd} with $B = 0 = D$, $A =(x+a) (\vc - (x+a)^{2})^{-1}$ and $C = -6^{-1}(\vc - (x+a)^{2})((x + a)^{2} + p(x+a) + q)$ satisfies:
\begin{enumerate}
\item On each connected component of $\{(x, y) \in \rea^{2}: \vc \neq (x + a)^{2}\}$, $\nabla$ is a critical symplectic connection satisfying $\K(\nabla) = x + a$ and $\sd^{\ast}\rf = 0$. 
\item $\nabla$ is preferred if and only if $q = \vc$. 
\item If $\vc < 0$ then $\nabla$ is defined on all of $\rea^{2}$ and is geodesically complete.
\end{enumerate}
\end{theorem}

\begin{proof}
There remains only to prove that $\nabla$ is complete when $\vc < 0$. In the case $q = \vc$ this is claimed in section $11$ of \cite{Bourgeois-Cahen}. The equations of the geodesics of $\nabla$ as in Lemma \ref{examplelemma} are
\begin{align}\label{geodesics}
&\ddot{x} + A\dot{x}^{2} - 2D\dot{x}\dot{y} + C\dot{y}^{2} = 0,& &\ddot{y} + B\dot{x}^{2} - 2A\dot{x}\dot{y} + D\dot{y}^{2} = 0.
\end{align}
No generality is lost by supposing $a = 0$. Then \eqref{geodesics} becomes
\begin{align}\label{geodesics2}
&\ddot{x} + \tfrac{x}{\vc - x^{2}}\dot{x}^{2} -\tfrac{1}{6}(\vc - x^{2})(x^{2} + px + q) \dot{y}^{2} = 0,& &\ddot{y}  - \tfrac{2x}{\vc - x^{2}}\dot{x}\dot{y} = 0.
\end{align}
From \eqref{geodesics2} it follows that along a solution $(\vc - x^{2})\dot{y}$ equals some constant $r$. (This follows from the constancy of $\rf(\dot{\ga})$ for any $\nabla$-geodesic $\ga$, which is an immediate consequence of $\sd^{\ast}\rf = 0$.) Hence 
\begin{align}
&\ddot{x} + \tfrac{x}{\vc - x^{2}}\dot{x}^{2} -\tfrac{r^{2}(x^{2} + px + q)}{6(\vc - x^{2})} = 0,& &\dot{y} = \tfrac{r}{\vc - x^{2}}.
\end{align}
Suppose $\vc < 0$ and let $x = \sqrt{-\vc}\sinh u$. Then $u$ solves the conservative equation
\begin{align}\label{ddotu}
\ddot{u} = f(u) = -\tfrac{r^{2}}{6}\tfrac{-\vc \sinh^{2} u + p\sqrt{-\vc}\sinh u + q}{(\sqrt{-\tau}\cosh u)^{3}}.
\end{align}
Along each solution there is a constant $E$ such that $\dot{u}^{2} + g(u) = E$ where $g(u) - g(u_{0}) = -2\int_{u_{0}}^{u}f(v)\,dv$. 
It is straightforward to see that there is a constant $C > 0$ such that $|f(u)| \leq C/(4\cosh(u)) = C\tfrac{d}{du}\arctan(e^{u})$ for $u \in \rea$. This bound implies $|g(u) - g(u_{0})| \leq C|\arctan(e^{u}) - \arctan(e^{u_{0}})| \leq C\pi$. Consequently, if $u(t)$ solves \eqref{ddotu} with initial conditions $u(0) = u_{0}$ and $\dot{u}(0) = v_{0}$, then 
\begin{align}
\dot{u}(t)^{2} = E - g(u(t)) \leq E - g(u(0)) + |g(u(t)) - g(u(0)| \leq E - g(u_{0}) + C\pi = v_{0}^{2} + C\pi.
\end{align}
Hence there is a constant $K > 0$ depending on $v_{0}$ such that $|\dot{u}(t)| \leq K$ for all $t$. This implies $|u(t)| \leq |u_{0}| + K|t|$, so that $u(t)$ does not blow up in finite time. Thus for any $u_{0}$ and $v_{0}$ there is a unique solution $u$ of \eqref{ddotu} defined for all time, so $x$ and $y$ are as well. Hence $\nabla$ is complete.
\end{proof}

\begin{remark}
In the examples of Theorem \ref{bcexampletheorem} the condition $B = 0$ can be dropped and solutions will still be obtained. With $D = 0$, the solutions are as above, but with $B_{x}C = -2BC_{x}$.
\end{remark}

\section{Critical symplectic connections on the standard two sphere}\label{twospheresection}
Other examples can be constructed using Lemma \ref{extlemma}, but in general it is difficult to find solutions that yield complete connections. For example, let $P(x) = x^{4}/24 + ax^{3}/6 + bx^{2}/2 + cx + d$. By \eqref{krea2}, \eqref{hk}, and Lemma \ref{extlemma}, the $\nabla$ defined by taking $A = B = D = 0$ and $C = P(x)$ in \eqref{piabcd} has $\K(\nabla) = x + a$ and $\hop(\K(\nabla)) = 0$ so is critical symplectic but not moment constant. 
These connections are not generally complete. By \eqref{geodesics}, a geodesic satisfies $y(t) = pt + q$ and $x(t)$ solves $\ddot{x} = -p^2P(x)$. For a polynomial $Q(x)$ with derivative equal to $P(x)$, $\dot{x}^{2} + 2p^{2}Q(x)$ is constant along a solution. Since the constant term of $Q$ is arbitrary, the constant of integration can be absorbed into $Q$, and there can be written $\dot{x}^{2} + 2p^{2}Q(x) = 0$, where the choice of primitive $Q$ depends on the particular solution curve. In general solutions of such an equation blow up in finite time.  

The following addresses the question of whether, for an appropriate choice of $P$, the connection so obtained can be extended to a symplectic connection on the two sphere with its usual volume form. The strange result is the construction of a family $\nabla(t)$ of critical symplectic connections defined on the complement in $\sphere$ of two antipodal points, such that $\nabla(0)$ is the Levi-Civita connection of the round metric, and, for $t \neq 0$, the difference tensor $\nabla(t) - \nabla(0)$ extends continuously but not differentiably at the two excluded antipodal points. Moreover, $\emf(\nabla(t))$ is a multiple of $t^{2}$, so while these connections are critical they are not minimizers of $\emf$ except when $t = 0$.

The round two-dimensional sphere $\sphere$ of volume $4\pi$ is the subset $\{(\sin \theta \cos \phi, \sin \theta \sin \phi, \cos \theta)\in \rea^{3}: \theta \in [0, \pi], \phi \in [0, 2\pi)\}$ of $\rea^{3}$ with the induced metric. 
In the coordinates $(x, y)$ defined on the complement of the poles by $x = -\cos\theta \in (-1, 1)$ and $y = \phi$ the standard round metric $g = d\theta^{2} + \sin^{2}\theta d\phi^{2}$ takes the form $g = (1 - x^{2})^{-1}dx^{2} + (1-x^{2})dy^{2}$, and the volume form $\Om = \sin \theta d\theta \wedge d\phi$ of $g$ equals the Darboux form $dx \wedge dy$. 
Let $\pr$ be the flat connection in the coordinates $(x, y)$. The Levi-Civita connection $D$ of $g$ has the form $D = \pr + \Lambda$ where $\Lambda(X, X) = x(1-x^{2})^{-1}X$, $\Lambda(X,Y) = \Lambda(Y,X) = -x(1 - x^{2})^{-1}Y$, and $\Lambda(Y,Y) = x(1 - x^{2})X$. Define $\nabla = D + \Ga = \pr + \Pi$ where the components $A$, $B$, $C$, and $D$ of $\Pi = \Ga + \Lambda$ are as in \eqref{piabcd}. Taking $A = x(1-x^{2})^{-1}$, $B = D= 0$, and $C = x(1-x^{2}) + P(x)$, where $P$ is a quartic polynomial in $x$, yields a connection $\nabla$ with $\K(\nabla)$ linear in $x$. Precisely, $C_{x} + 3AC = 1 + P^{\prime} + 3x(1-x^{2})^{-1}P$. Taking $P = (1-x^2)(ax^{2} + bx + c)$ yields $\K(\nabla) = C_{xxx} + 3(AC)_{xx} = -6ax$, so if $a = -1/6$, the resulting connection satisfies $\K(\nabla) = x$. So far this $\nabla$ has been defined only away from the poles in $\sphere$. There remains to determine whether $b$ and $c$ can be chosen so that it extends smoothly at the poles. 

As is explained following its proof, Theorem \ref{sphereextremaltheorem} gives an intrinsic construction of a special case of the connections just described. Although its statement appears more general, Theorem \ref{sphereextremaltheorem} is interesting mainly in the special case $D$ is the Levi-Civita connection of a round metric on $M = \sphere$.

\begin{theorem}\label{sphereextremaltheorem}
Let $(M, \Om)$ be a symplectic $2$-manifold and let $D \in \symcon(M, \Om)$ have parallel Ricci tensor $\sR_{ij}$. Suppose $Z_{i}$ is a nontrivial one-form satisfying $D_{(i}Z_{j)} = 0$. Let $\ga = D_{p}Z^{p}$. Then
\begin{align}\label{nuga}
\nu = \ga^{2} + 4Z^{a}Z^{b}\sR_{ab} 
\end{align}
is constant. On $\hat{M} = \{p\in M: \ga(p)^{2} \neq \nu\}$ define $\Ga_{ijk} = (\nu - \ga^{2})^{-1}Z_{i}Z_{j}Z_{k}$. The symplectic connection $\nabla(t) = D + t\Ga_{ij}\,^{k}$ satisfies
\begin{align}
\label{riczzz}R(t)_{ij}& = \sR_{ij} + t\ga(\nu - \ga^{2})^{-1}Z_{i}Z_{j},\\
\label{sriczzz}-\sd^{\ast}_{\nabla(t)}\ric(\nabla(t))_{ijk} &= \nabla(t)_{(i}R(t)_{jk)} = -4t\nu(\nu - \ga^{2})^{-2}Z_{(i}Z_{j}\sR_{k)p}Z^{p},\\
\label{serfk1} \rf(\nabla(t))_{i} &= \tfrac{1}{2}t(\nu - 3\ga^{2})(\nu - \ga^{2})^{-1}Z_{i}, \\
\label{serfk2}\K(\nabla(t)) &= -\tfrac{3}{4}t\ga,\\
\label{vcnt} \vc(\nabla(t)) &= (3t/16)\nu,
\end{align}
where $R(t)_{ij} = \ric(\nabla(t))_{ij}$ and $\vc$ is the constant of \eqref{exvc} of Lemma \ref{extremalstructurelemma}. 
Moreover, $\hop_{\nabla(t)}(\ga) = 0$, so $\nabla(t)$ is critical. If $4\nu(\nu - \ga^{2})^{-2}Z_{(i}Z_{j}\sR_{k)p}Z^{p}$ is nonzero somewhere on $\hat{M}$, then $\nabla(t)$ is not preferred for $t \neq 0$. Finally, the Jacobi operator satisfies $\jac_{\nabla(t)}(\Ga) = 0$ for all $t \in \rea$.
\end{theorem}

\begin{proof}
Let $\nabla = \nabla(1)$. For $t \neq 0$ it suffices to check all the claims for $\nabla$ because $t\Ga_{ijk}$ is obtained from $tZ_{i}$ in the same way as $\Ga_{ijk}$ is obtained from $Z_{i}$; that is the connection $\nabla$ corresponding to $tZ$ in place of $Z$ is simply $\nabla(t)$.

By assumption $2D_{i}Z_{j} = 2D_{[i}Z_{j]} = \ga\Om_{ij}$. By the Ricci identity, 
\begin{align}\label{dgai0}
2d\ga_{i} = 2D_{i}D_{p}Z^{p} = 2D_{p}D_{i}Z^{p} - 2\sR_{ip}Z^{p} = D_{p}(\ga\delta_{i}\,^{p}) - 2\sR_{ip}Z^{p} = d\ga_{i} - 2\sR_{ip}Z^{p}, 
\end{align}
so 
\begin{align}\label{dgai}
&d\ga_{i} = -2\sR_{ip}Z^{p}.&
\end{align}
Since $\sR_{ij}$ is parallel, $D_{i}(Z^{a}Z^{b}\sR_{ab}) = 2Z^{a}\sR_{ab}D_{i}Z^{b} = \ga \sR_{ip}Z^{p}$. Combined with \eqref{dgai} this shows that $d(\ga^{2} + 4Z^{a}Z^{b}\sR_{ab}) = 0$, so there is a constant $\nu$ satisfying \eqref{nuga}. From \eqref{dgai} and \eqref{nuga} there follows
\begin{align}\label{zdga}
Z^{p}d\ga_{p} = -2Z^{a}Z^{b}\sR_{ab} = -(\nu - \ga^{2})/2.
\end{align}
Computing using $2D_{i}Z_{j} = \ga\Om_{ij}$ and \eqref{zdga} shows
\begin{align}\label{sdpi}
\sd \Ga_{ij} = D_{p}\Ga_{ij}\,^{p} = \ga(\nu - \ga^{2})^{-1} Z_{i}Z_{j}.
\end{align}
Since $\Ga_{ip}\,^{l}\Ga_{jk}\,^{p} = 0$ and $\Ga_{ip}\,^{p} = 0$, \eqref{riczzz} follows from \eqref{sdpi}. A bit of computation using \eqref{riczzz} shows
\begin{align}\label{sriczzz0}
\begin{split}
\nabla_{i}R_{jk} & = D_{i}R_{jk} - 2\Ga_{i(j}\,^{p}R_{k)p} = D_{i}(\ga(\nu - \ga^{2})^{-1} Z_{j}Z_{k}) - 2(\nu -\ga^{2})^{-1}Z_{i}Z_{(j}\sR_{k)p}Z^{p}\\
& = -2((\nu - \ga^{2})^{-1} + 2\ga^{2}(\nu - \ga^{2})^{-2})Z^{p}\sR_{ip}Z_{j}Z_{k} \\
&\quad- 2(\nu - \ga^{2})^{-1}Z_{i}Z_{(j}\sR_{k)p}Z^{p} + \ga^{2}(\nu - \ga^{2})^{-1}\Om_{i(j}Z_{k)},
\end{split}
\end{align}
from which \eqref{sriczzz} follows. Contracting \eqref{sriczzz0} and simplifying using \eqref{zdga} yields \eqref{serfk1}. Alternatively, \eqref{serfk1} follows from \eqref{rfvary} coupled with the observations that, by \eqref{sdpi}, $4\lop_{D}^{\ast}(\Ga)_{i} = (\nu - \ga^{2})^{-1}(3\ga^{2} - \nu)Z_{i}$ and that in \eqref{rfvary} the terms of order at least two in $t$ all vanish. Since $Z^{p}\rf(\nabla)_{p} = 0$, $\nabla_{i}\rf(\nabla)_{j} = D_{i}\rf(\nabla)_{j}$, so to compute $\K(\nabla)$ it suffices to compute $D_{i}\rf(\nabla)_{j}$ and contract the result. There results \eqref{serfk2}. As for \eqref{serfk1}, \eqref{serfk2} can also be computed using \eqref{kvary} and 
\begin{align}\label{hdpi}
\hop_{D}^{\ast}(\Ga) = \sd_{D}\lop_{D}^{\ast}(\Ga) = -(3/4)\ga.
\end{align}
Differentiating \eqref{dgai} yields $D_{i}d\ga_{j} = -\ga\sR_{ij}$, and, with \eqref{zdga}, there follow
\begin{align}\label{nndg}
\begin{split}
\nabla_{i}d\ga_{j} &= D_{i}d\ga_{j} - Z^{p}d\ga_{p}Z_{i}Z_{j} = -\ga \sR_{ij} + 2(Z^{a}Z^{b}\sR_{ab})Z_{i}Z_{j} = -\ga\sR_{ij}  + \tfrac{1}{2}Z_{i}Z_{j},\\
\nabla_{i}\nabla_{j}d\ga_{k} & = D_{i}\nabla_{j}d\ga_{k} - 2(\nu - \ga^{-2})^{-1}Z_{i}Z_{(j}Z^{p}\nabla_{k)}d\ga_{p}\\
& = -d\ga_{i}\sR_{jk} + \tfrac{1}{2}\ga\Om_{i(j}Z_{k)} + 2\ga(\nu - \ga^{-2})^{-1}Z_{i}Z_{(j}Z^{p}\sR_{k)p}.
\end{split}
\end{align}
Using \eqref{riczzz}, \eqref{dgai}, and \eqref{nndg} yields
\begin{align}
\begin{split}
-\hop_{\nabla}(\ga) & = \nabla_{i}\nabla_{j}d\ga_{k} + d\ga_{i}R_{jk}  = \nabla_{i}\nabla_{j}d\ga_{k} + d\ga_{i}\sR_{jk} + \ga(\nu - \ga^{-2})^{-1}d\ga_{i}Z_{j}Z_{k}  \\
&=   \tfrac{1}{2}\ga\Om_{i(j}Z_{k)} + \ga(\nu - \ga^{-2})^{-1}\left( d\ga_{i}Z_{j}Z_{k} -Z_{i}Z_{(j}d\ga_{k)}\right)  \\
&=   \tfrac{1}{2}\ga\Om_{i(j}Z_{k)} + \ga(\nu - \ga^{-2})^{-1}\left( d\ga_{[i}Z_{j]}Z_{k} + d\ga_{[i}Z_{k]}Z_{j}\right) = 0,
\end{split}
\end{align}
in which the final equality follows from $4d\ga_{[i}Z_{j]} = 2Z^{p}d\ga_{p}\Om_{ij} = -(\nu - \ga^{2})\Om_{ij}$, which uses \eqref{zdga}.

It follows from \eqref{riczzz} that $\lop^{\ast}_{\nabla}(\Ga)_{ij} = \lop^{\ast}_{D}(\Ga)_{ijk}$, and from \eqref{sdpi} and \eqref{hdpi} that $\hop^{\ast}_{\nabla}(\Ga) = \sd_{\nabla}\lop^{\ast}_{\nabla}(\Ga) = \sd_{\nabla}\lop^{\ast}_{D}(\Ga) = \sd_{D}\lop^{\ast}_{D}(\Ga) = \hop^{\ast}_{D}(\Ga) = -(3/4)\ga$. Then
\begin{align}\label{jacn}
\begin{split}
\jac_{\nabla}(\Ga) &= \hop_{\nabla}\hop_{\nabla}^{\ast}(\Ga) + \lie_{\hm_{\K(\nabla)}}\Ga = \hop_{\nabla}(-(3/4)\ga) - (3/4)\lie_{X}\Ga = - (3/4)\lie_{X}\Ga,
\end{split}
\end{align}
where $X^{i} = - d\ga^{i}$. Since
\begin{align}
D_{p}\Ga_{ijk} & = \tfrac{3}{2}\ga(\nu - \ga^{2})^{-1}\Om_{p(i}Z_{j}Z_{k)} + 2\ga(\nu - \ga^{2})^{-2}d\ga_{p}Z_{i}Z_{j}Z_{k},
\end{align}
and $D_{i}d\ga_{j} = - \ga \sR_{ij}$,
\begin{align}
\begin{split}
(\lie_{X}\Ga)_{ijk}& = -d\ga^{p}D_{p}\Ga_{ijk} - 3D_{(i}d\ga^{p}\Ga_{jk)p} \\
&= \tfrac{3}{2}\ga(\nu - \ga^{2})^{-1}d\ga_{(i}Z_{j}Z_{k)} + 3\ga(\nu - \ga^{2})^{-1}Z^{p}\sR_{p(i}Z_{j}Z_{k)} = 0,
\end{split}
\end{align}
the last equality by \eqref{dgai}, it follows from \eqref{jacn} that $\jac_{\nabla}(\Ga) = 0$. While the preceding argument does not show that $\jac_{D}(\Ga) = 0$, this is easily checked directly.
\end{proof}

If $D$ is the Levi-Civita connection of a Riemannian metric then the hypothesis of Theorem \ref{sphereextremaltheorem} means that $Z_{i}$ is metrically dual to a Killing field, so if $M$ is compact and $Z$ is nontrivial then $M$ must be a sphere or a torus. Since by hypothesis the Ricci tensor is parallel, in the case $M$ is a torus, the metric must be flat and any Killing field is parallel; in this case $\nu = 0$ and $\ga = 0$ and the set $\hat{M}$ of Theorem \ref{sphereextremaltheorem} is empty, so the construction is vacuous. 

In the case $M = \sphere$, the vector field $Y = \pr_{y}$, which is rotation around the axis through the deleted antipodal poles, is a nontrivial Killing field for the round metric $g$ for which the metrically dual one-form $Z = \imt(Y)g = (1-x^{2})dy$ has the properties required in Theorem \ref{sphereextremaltheorem}. (This example motivated the theorem.) Here the notations are as in the paragraphs preceding the statement of Theorem \ref{sphereextremaltheorem}. Precisely, $Z$ satisfies $DZ = -x\Om$, and for an appropriate choice of $P$, the difference tensor $\Ga$ is $(1/4)(1 - x^{2})^{-1}Z_{i}Z_{j}Z_{k} = (1/4)(1 - x^{2})^{2}dy \tensor dy \tensor dy$. Since $\ga = -2x$ and $Z^{a}Z^{b}\sR_{ab} = |Z|_{g}^{2} = (1-x^{2})$, the constant $\nu$ equals $4$. Theorem \ref{sphereextremaltheorem} shows that in this setting the resulting $\nabla$ is critical but not preferred. However, while the tensor $\Ga_{ijk}$ extends continuously at the poles, vanishing there, its extension is not differentiable at the poles. The behavior at the poles is seen most easily in different coordinates. A convenient choice is $u = \tan{\tfrac{\theta}{2}}\cos y = \sqrt{\tfrac{1 + x}{1-x}}\cos y$ and $v = \tan{\tfrac{\theta}{2}}\sin y = \sqrt{\tfrac{1 + x}{1-x}}\sin y$.
In these coordinates $g$ has the standard form $4(1 + u^{2} + v^{2})^{-2}(du^{2} + dv^{2})$ and the origin corresponds to $x = -1$ (the pole at $x = 1$ can be handled similarly). Since $u^{2} + v^{2} = (1+x)/(1-x)$, $x = (u^{2} + v^{2} - 1)/(u^{2} + v^{2} + 1)$, and $dy = (u^{2} + v^{2})^{-1}(udv - vdu)$,
\begin{align}
\Ga = 4(1 + u^{2} + v^{2})^{-4}(u^{2} + v^{2})^{-1}(udv - vdu)^{\tensor\,3}.
\end{align}
The components of $\Ga$ behave like a smooth multiples of $u^{a}v^{3-a}(u^{2} + v^{2})^{-1}$ for $a \in \{0, 1, 2, 3\}$, so extend continuously at the origin of $(u, v)$ coordinates, but do not extend differentiably there.

For the $\nabla(t)$ of Theorem \ref{sphereextremaltheorem}, by \eqref{serfk2}, since $\K(\nabla(t))$ extends smoothly to all of $\sphere$, 
\begin{align}\label{emfnablat}
\emf(\nabla(t)) = \int_{\sphere}\K(\nabla)^{2}\,\Om = (9/4)t^{2}\int_{0}^{2\pi}\int_{-1}^{1}x^{2}\,dx\,dy = 3\pi t^{2}.
\end{align}
Since the Levi-Civita connection $D$ of the round metric on $\sphere$ has $\emf(D) = 0$, this shows that $\nabla(t)$ is not an absolute minimizer of $\emf$ when $t \neq 0$, although it is critical. 
In particular, there is a one-parameter family of critical symplectic connections on the Darboux cylinder along which $\emf$ takes on all nonnegative real values. 
Comparing \eqref{vcnt} and \eqref{emfnablat} shows the necessity of the boundary term in \eqref{vcval}. For $\nabla = \nabla(1)$, by \eqref{vcnt}, $\tau\vol_{\Om}(M) = 3\pi$, while, by \eqref{emfnablat}, $\emf(\nabla) = 3\pi$, rather than $9\pi$ as \eqref{vcval} would imply were the boundary term null. 
However, by \eqref{serfk1} there holds $\rf = 2^{-1}(1 - 3x^{2})dy$, and so $\K\rf = (3/4)x(1 - 3x^{2})dy$, which does not extend differentiably at the poles $x = \pm 1$ since the one-form $dy$ does not extend differentiably. It follows that the boundary term $\int_{M}d(\K\rf)$ in \eqref{vcval} does not vanish. Indeed, 
\begin{align}
\begin{split}
\int_{0}^{2\pi}\int_{-1 + \ep_{1}}^{1 - \ep_{2}}d(\K\rf) &= -\int_{0}^{2\pi}(\K\rf)_{x = 1- \ep_{2}} + \int_{0}^{2\pi}(\K\rf)_{x =-1 + \ep_{1}}\\
& = -2\pi(3/4)(1-\ep_{2})(1 - 3(1-\ep_{2})^{2}) + 2\pi(3/4)(\ep_{1} - 1)(1 - 3(\ep_{1} - 1)^{2})
\end{split}
\end{align}
which tends to $6\pi$ when $\ep_{1} \to 0$ and $\ep_{2} \to 0$. This is what is needed to yield equality in \eqref{vcval}.

\section{Moment flat connections that are not projectively flat}\label{momentflatexamplesection}

For $\Pi$ as in \eqref{piabcd}, the components of $\Pi^{\sflat}(U, V, W) = \Om(\Pi(U, V), W)$ are 
\begin{align}\label{pisflat}
&\Pi^{\sflat}(X, X, X) = -B, & &\Pi^{\sflat}(X, X, Y) = A,& & \Pi^{\sflat}(X, Y, Y) = -D,& \Pi^{\sflat}(Y, Y, Y) = C.
\end{align}
The expressions \eqref{rfreal2} and \eqref{krea2} in Lemma \ref{examplelemma} simplify considerably if $AD - BC = 0$, $A^{2} - BD = 0$, and $AC - D^{2} = 0$. These are the conditions for the cubic form \eqref{pisflat} to take values in the image of a rational normal curve, that is for $\Pi^{\sflat}$ to be decomposable in the sense that there is a one-form $\si$ such that $\Pi^{\sflat} = \si \tensor \si \tensor \si$. Theorem \ref{sphereextremaltheorem} shows what happens when $\si$ is metrically dual to a metric Killing field. A different simplification is obtained by supposing that $\si$ is closed. Lemma \ref{xxxlemma} shows how Lemma \ref{examplelemma} simplifies when $\Pi^{\sflat}$ is supposed to be the cube of a closed one-form, and Theorem \ref{cubeexampletheorem} shows how it simplifies further when $\Pi^{\sflat}$ is the cube of an exact one-form. 

\begin{lemma}\label{xxxlemma}
Let $(M, \Om)$ be a surface with a volume form, and let $D \in \symcon(M, \Om)$. Let $\nabla = D + \Pi_{ij}\,^{k} \in \symcon(M, \Om)$, where $\Pi_{ijk} = X_{i}X_{j}X_{k}$ and the one-form $X_{i}$ is closed. Let $R_{ij}$ and $\sR_{ij}$ be the Ricci curvatures of $\nabla$ and $D$ and let $\mx = \det D_{i}X^{j}$. Then
\begin{align}
\label{dxric}
R_{ij} &= \sR_{ij} + \sd\Pi_{ij} = \sR_{ij} + 2X^{p}X_{(i}D_{j)}X_{p},\\
\label{rfxxx}
\begin{split}
\rf(\nabla)_{i} &= \rf(D)_{i} + 12 \mx X_{i} - 6X_{i}X^{p}X^{q}\sR_{pq}  - 2D_{i}(X^{p}X^{q}D_{p}X_{q})\\
& =  \rf(D)_{i} + 8 \mx X_{i} - 6X_{i}X^{p}X^{q}\sR_{pq}  - 2X^{p}X^{q}D_{i}D_{p}X_{q},
\end{split}\\
\label{rfxxx2}\K(\nabla) & = \K(D) - 6X^{p}d\mx_{p} + 3X^{p}D_{p}(X^{a}X^{b}\sR_{ab}).
\end{align}
\end{lemma}

\begin{proof}
Because $2D_{[i}X_{j]} = D_{p}X^{p}\Om_{ij}$, that $X$ is closed means that $D_{p}X^{p} = 0$. For any $\binom{1}{1}$-tensor $A_{i}\,^{j}$ such that $A_{p}\,^{p} = 0$ there holds 
$A_{i}\,^{p}A_{p}\,^{j} = -\det(A)\delta_{i}\,^{j}$, and so there hold $A_{p}\,^{q}A_{q}\,^{p} = -2\det(A)$, $A^{pq}A_{qp} = 2\det(A)$, and $A_{a}\,^{b}A_{b}\,^{c}A_{c}\,^{a} = 0$. Applying these observations to $A_{i}\,^{j} = D_{i}X^{j}$ yields
\begin{align}\label{dxmx}
&D_{i}X^{p}D_{p}X_{j} = -\mx \Om_{ij},& &D^{p}X^{q}D_{p}X_{q} = 2\mx.
\end{align}
By \eqref{dxmx} and the Ricci identity,
\begin{align}\label{dxmx2}
\begin{split}
X^{p}X^{q}D_{i}D_{p}X_{q} & = D_{i}(X^{p}X^{q}D_{p}X_{q}) -  2X^{p}D_{i}X^{q}D_{p}X_{q} = D_{i}(X^{p}X^{q}D_{p}X_{q}) -  2\mx X_{i}.
\end{split}
\end{align}
Since $\sR_{p}\,^{p}\,_{ij} = 2\sR_{ij}$, $D^{p}D_{p}X_{i} = \sR_{ip}X^{p}$, and, by \eqref{symplecticweyl}, $\sR_{ijkl}X^{j}X^{k}X^{l} = -X_{i}X^{p}X^{q}\sR_{pq}$. Because $X_{i}$ is closed, $\sd \Pi_{ij} = D_{p}\Pi_{ij}\,^{p} = 2X^{p}X_{(i}D_{j)}X_{p}$, and \eqref{dxric} follows. Using \eqref{dxmx}, $D^{p}D_{p}X_{i} = \sR_{ip}X^{p}$, the Ricci identity, and finally \eqref{dxmx2} there results
\begin{align}
\label{dxmx3}
\begin{split}
\lop^{\ast}(\Pi)_{i} &= - \sd^{2}\Pi_{i} - \Pi_{ipq}\sR^{pq}  = - 4\mx X_{i} + 2X^{p}X^{q}\sR_{pq}X_{i} + X^{p}X^{q}D_{p}D_{q}X_{i}\\
& =  - 4\mx X_{i} + 2X^{p}X^{q}\sR_{pq}X_{i} + X^{p}X^{q}(D_{i}D_{p}X_{q} - \Om_{pi}\sR_{q}\,^{a}X_{a})\\
& =  - 4\mx X_{i} + 3X^{p}X^{q}\sR_{pq}X_{i} + X^{p}X^{q}D_{i}D_{p}X_{q}\\
& = - 6\mx X_{i} + 3X^{p}X^{q}\sR_{pq}X_{i} + D_{i}(X^{p}X^{q}D_{p}X_{q}).
\end{split}
\end{align}
Since $\Pi_{ip}\,^{q}\Pi_{jq}\,^{p} =0$ and $\Pi_{ipq}\sd \Pi^{pq} = 0$, by \eqref{rfvary}, $\rf(\nabla) = \rf(D) - 2\lop^{\ast}(\Pi)$, and with \eqref{dxmx3} this yields \eqref{rfxxx}. Since $2\K(\nabla) = \nabla^{p}\rf(\nabla)_{p} = D^{p}\rf(\nabla)_{p}$, applying $D^{j}$ to \eqref{rfxxx} yields \eqref{rfxxx2}.
\end{proof}

On a surface $M$, let $(g_{ij}, J_{i}\,^{j})$ be a constant curvature Kähler structure having Levi-Civita connection $D$ and volume form $\Om_{ij} = J_{i}\,^{p}g_{pj}$. For $f \in \cinf(M)$, the function $\Hs(f) = \det (D_{i}df^{j})$ is the usual Hessian determinant. Precisely, the determinant $\det \Om$ of the covariant two-tensor $\Om$ is identified in a canonical way with the section $\Om^{\tensor 2}$ of the tensor square of the line bundle of two-forms.
Consequently, as $D_{i}df^{p}\Om_{pj} = D_{i}df_{j}$, $\Hs(f)\Om^{\tensor\,2} = \det Ddf$. In the case $D$ is a flat affine connection and $\Om = dx \wedge dy$, $\Hs(f) = \det Df \tensor \Om^{\tensor\,-2}$ is just the usual Hessian determinant. Since $D^{i}df^{p}D_{p}df_{j} = \Hs(f)\delta_{j}\,^{i}$, the tensor $D^{i}df^{j}$ is the adjugate tensor of $D_{i}df_{j}$. Define $\U(f) = df_{i}df_{j}D^{i}df^{j}$, the contraction of the adjugate tensor of the Hessian $\pr df$ of $f$ with $df \tensor df$.

\begin{corollary}\label{xxxcorollary}
On a surface $M$, let $(g_{ij}, J_{i}\,^{j})$ be a Kähler structure having Levi-Civita connection $D$, volume form $\Om_{ij} = J_{i}\,^{p}g_{pj}$, and constant scalar curvature $\sR$. 
For $f \in \cinf(M)$, let $\Hs(f) = \det (D_{i}df^{j})$ and $\U(f) = df_{i}df_{j}D^{i}df^{j}$. 
Then $\nabla = D + df_{i}df_{j}df^{k} \in \symcon(M, \Om)$ satisfies
\begin{align}
&\rf(\nabla)_{i} = - 3\sR|df|^{2}_{g}df_{i} + 12 \Hs(f)df_{i} - 2d\U(f)_{i},&
&\K(\nabla)  = 6\{f, \Hs(f)\} - \tfrac{3}{2}\sR\{f, |df|^{2}_{g}\}.
\end{align}
\end{corollary}
\begin{proof}
Take $X_{i} = df_{i}$ in Lemma \ref{xxxlemma}.
\end{proof}

Corollary \ref{xxxcorollary} yields moment flat symplectic connections that are not projectively flat.

\begin{theorem}\label{cubeexampletheorem}
Let $\pr$ be the standard flat connection on $\rea^{2}$ equipped with the symplectic form $\Om = dx \wedge dy$. Define $\nabla \in \symcon(\rea^{2}, \Om)$ by $\nabla = \pr + \Pi$ where $\Pi_{ijk} = df_{i}df_{j}df_{k}$ for $f \in \cinf(M)$. Then
\begin{align}\label{rff}
&\rf  = 12\Hs(f)df - 2d\U(f),& & & \K(\nabla) = 6\{f, \Hs(f)\}.
\end{align}
If $\Hs(f)$ Poisson commutes with $f$ then $\K(\nabla) = 0$, but $\nabla$ is not projectively flat whenever $d\U(f) - 6\Hs(f)df$ is somewhere nonzero. 
In particular, if the graph of $f$ is an improper affine sphere then $\K(\nabla) = 0$. There exist improper affine spheres for which the resulting $\nabla$ is not projectively flat.
\end{theorem}

\begin{proof}
Except for the final claims, this is a special case of Corollary \ref{xxxcorollary}. The graph of $f$ is an improper affine sphere if and only if $\Hs(f)$ equals a nonzero constant. In this case $\Hs(f)$ Poisson commutes with $f$ so $\K(\nabla) = 0$.

If $f$ has homogeneity $2$, meaning $xf_{x} + yf_{x} = 2f$, then $\U(f) = 2f\Hs(f)$, so that in this case, by \eqref{rff}, $\rf = 12\Hs(f)df - 4d(f\Hs(f)) = 8\Hs(f)df - 4fd\Hs(f)$. If, moreover, $\Hs(f)$ is constant but $f$ is not, then $\rf = 8\Hs(f)df$ is not zero, although $\K(\nabla)$ vanishes. This occurs for any homogeneous quadratic polynomial, e.g. $x^{2} \pm y^{2}$ or $xy$. 
More interesting examples are obtained as follows. Let $u(x)$ be any smooth function on the line. Then $f = xy + u(x)$ satisfies $\Hs(f) = -1$ and $\U(F) = x^{2}u^{\prime\prime} - 2x(y + u^{\prime})$, so that $\rf = -2d\U(f) = (4y + 4u^{\prime} - 2x^{2}u^{\prime\prime\prime})dx + 4xdy$. Since $\Hs(f) = -1$ the graph of $f$ is an improper (ruled) affine sphere, and since $\rf$ never vanishes identically the resulting connection is not projectively flat.
\end{proof}

\begin{remark}
Specializing \eqref{dxric} shows that the Ricci curvature of $\nabla$ as in Theorem \ref{cubeexampletheorem} has the form $\ric = -\nabla_{\hm_{f}}(df\tensor df) = -\pr_{\hm_{f}}(df\tensor df)$. Hence $\ric(\hm_{f}, \hm_{f}) = 0$, so $\ric$ degenerates along $\hm_{f}$.
\end{remark}

Horospheres in hyperbolic space are convex submanifolds, flat in the induced metric, and having constant Gaussian curvature, so are analogous to affine spheres. This motivates the next example.

\begin{theorem}\label{hyperbolicplaneexampletheorem}
Let $D$ be the Levi-Civita connection of the metric $g$ on the hyperbolic plane and let $\Om$ be the corresponding volume form. Let $\be$ be a Busemann function normalized to take the value $-\infty$ at a fixed point in the ideal boundary. Define $\nabla \in \symcon(\rea^{2}, \Om)$ by $\nabla = D + \Pi$ where $\Pi_{ijk} = df_{i}df_{j}df_{k}$ and $f = e^{\be}$. Then $\rf = 6f^{2}df = 2d(f^{3})$, so $\K(\nabla) = 0$ but $\nabla$ is not projectively flat.
\end{theorem}

\begin{proof}
In the upper half-space model of hyperbolic space, the Busemann function at the point at infinity is the negative of the logarithm of the vertical coordinate. Using this it is straightforward to check that $Dd\be + d\be \tensor d\be = g$, so that $Ddf = Dde^{\be} = e^{\be}g = fg$. Hence $D_{i}df^{j} = -fJ_{i}\,^{j}$, so $\Hs(f) = \det(D_{i}df^{j}) = f^{2}$. As $|d\be|_{g}^{2} = 1$, there holds $|df|^{2}_{g}=  e^{2\be} = f^{2}$. Hence $\U(f) = df^{i}df^{j}D_{i}df_{j} = f|df|_{g}^{2} = f^{3}$. Consequently $\rf = -2d\U(f) + 12\Hs(f)df = 6f^{2}df$.
\end{proof}

Theorem \ref{hyperbolicexampletheorem} provides examples of moment flat symplectic connections on compact surfaces that are not projectively flat. These are used later in the proof of Theorem \ref{cohomtheorem}, that shows that any cohomology class in $H^{1}(M;\rea)$ is represented by the curvature one-form of some moment flat symplectic connection.

\begin{theorem}\label{hyperbolicexampletheorem}
On a surface $M$, let $(g_{ij}, J_{i}\,^{j})$ be a Kähler structure having Levi-Civita connection $D$, volume form $\Om_{ij} = J_{i}\,^{p}g_{pj}$, and scalar curvature $\sR_{g}$. 
Let $X_{i}$ be a nontrivial harmonic one-form. Then, for  $\Pi_{ijk} = 3X_{(i}g_{jk)}$, the symplectic connection $\nabla = D + \Pi_{ij}\,^{k}$ satisfies $\rf(\nabla)_{i} = \rf(D)_{i} - 6\sR_{g} X_{i}$. In particular, if $g$ has constant curvature, then $\rf(\nabla) = -6\sR_{g}X_{i}$ and $\K(\nabla) = 0$. 
\end{theorem}

\begin{proof}
The difference $\rf(\nabla) - \rf(D)$ is computed using \eqref{rfvary}. The operators $\sd$, $\lop$, etc. are those associated with $D$. Since one customarily raises and lowers indices with the metric rather than the volume form, the notation used here can be confusing. For instance, since $J_{i}\,^{p}\Om_{pj} = -g_{ij}$, $J_{ij} = -g_{ij}$ and $g_{i}\,^{j} = -J_{i}\,^{j}$. That $X$ is harmonic implies that $D_{i}X_{j} = D_{j}X_{i}$, $J_{i}\,^{p}D_{j}X_{p} = J_{j}\,^{p}D_{i}X_{p}$, $g^{pq}D_{p}X_{q} = 0$, and $D^{p}X_{p} = 0$. Using $D^{p}X_{p} = 0$, the symmetry of $J_{j}\,^{p}D_{i}X_{p}$ and $D_{i}X_{j}$, and $g^{pq}\Pi_{ipq} = 4X_{i}$, there result
\begin{align}\label{dpiijp}
\sd\Pi_{ij} &= - D^{p}\Pi_{ijp} = D^{p}(X_{i}J_{jp} + X_{j}J_{ip} + X_{p}J_{ij}) = -2J_{(i}\,^{p}D_{|p|}X_{j)} = -2J_{i}\,^{p}D_{j}X_{p},\\
\sd^{2}\Pi_{i} & = D^{p}\sd\Pi_{ip} = 2J_{i}\,^{p}D_{q}D^{q}X_{p} = -\sR_{g}J_{i}\,^{p}g_{p}\,^{q}X_{q} = -\sR_{g}X_{i},\\
\label{bp1}\lop^{\ast}(\Pi)_{i} & = -\sd^{2}\Pi_{i} - \tfrac{1}{2}\sR_{g}\Pi_{i}\,^{ab}g_{ab} = 3\sR_{g}X_{i},\\
\label{pisquared}
B(\Pi)_{ij} &= (2X_{(i}J_{p)}\,^{q}  + J_{ip}X^{q})(2X_{(j}J_{q)}\,^{p} + J_{jq}X^{p}) = -6X_{i}X_{j} + 2J_{i}\,^{p}J_{j}\,^{q}X_{p}X_{q},\\
\label{bp3}
\sd B(\Pi)_{i}& + \Pi_{ipq}\sd\Pi^{pq} = D^{p}B(\Pi)_{ip} +  4J_{i}\,^{a}J_{p}\,^{b}D^{p}X_{a} = 6X^{p}D_{p}X_{i} + 6J_{i}\,^{a}J_{p}\,^{b}D^{p}X_{a} = 0.
\end{align}
Substituting \eqref{bp1}, \eqref{bp3}, and $T(\Pi)_{i} = -\Pi_{i}\,^{pq}B(\Pi)_{pq} = 0$ in \eqref{rfvary} yields $\rf(\nabla)_{i} = \rf(D)_{i} - 6\sR_{g} X_{i}$. If $g$ is assumed to have constant curvature, $\rf(D) = 0$, so $\rf(\nabla) = -6\sR_{g}X_{i}$ and $\K(\nabla) = 0$.
\end{proof}

\section{Representation of \texorpdfstring{$H^{1}(M;\rea)$}{dr} by curvature one-forms of moment flat connections }\label{representabilitysection}
Let $(M, \Om)$ be a symplectic $2$-manifold. If $\nabla \in \K^{-1}(0)$, then, by \eqref{cdiv}, $\rf(\nabla)$ is closed so the de Rham cohomology class $[\rf]$ is defined.
If $\phi_{t}$ is the flow of $X \in \symplecto(M, \Om)$ and $\nabla \in \K^{-1}(0)$, then $\rf(\phi_{t}^{\ast}(\nabla)) - \rf(\nabla) = \phi_{t}^{\ast}(\rf(\nabla)) - \rf(\nabla)$ is homotopic to the zero form, so exact, and hence the cohomology class $[\rf(\nabla)]$ is preserved by the action on $\K^{-1}(0)$ of the path connected component of the identity $\Symplecto(M, \Om)_{0} \subset \Symplecto(M, \Om)$. 
Hence there is a map
\begin{align}\label{rfcdefined0}
\rfc: \K^{-1}(0)/\Symplecto(M, \Om)_{0} \to H^{1}(M; \rea),
\end{align}
defined by 
\begin{align}\label{rfcdefined}
\rfc(\nabla \cdot \Symplecto(M, \Om)_{0}) = [\rf(\nabla)]. 
\end{align}
It is natural to ask if $\rfc$ is surjective, that is whether for a given symplectic form a given de Rham cohomology class $[\al] \in H^{1}(M; \rea)$ can be represented by $\rf(\nabla)$ for some $\nabla \in \K^{-1}(0)$. For compact surfaces of negative Euler characteristic, Theorem \ref{cohomtheorem} shows the answer is affirmative.  That is, the map $\rfc$ of \eqref{rfcdefined0} and \eqref{rfcdefined} is surjective.

\begin{proof}[Proof of Theorem \ref{cohomtheorem}]
Pick a constant curvature metric $g$ such that the $g$-volume of $M$ equals $\int_{M}\Om$. By a theorem of Moser there exists a diffeomorphism $\phi$ of $M$ isotopic to the identity such that $\vol_{\phi^{\ast}(g)} = \phi^{\ast}(\vol_{g}) = \Om$. Since $\phi^{\ast}(g)$ has constant curvature, it can be assumed from the beginning that $g$ is a constant curvature metric with volume form equal to $\Om$. Since the Euler characteristic is nonzero, the curvature $\sR_{g}$ is nonzero. Let $D$ be the Levi-Civita connection of $g$ and let $\al_{i}$ be the unique $g$-harmonic representative of $[\al]$. By Theorem \ref{hyperbolicexampletheorem} the connection $\nabla = D + \Pi_{ij}\,^{k}$, where $\Pi_{ijk} = -(6\sR_{g})^{-1}\al_{(i}g_{jk)}$, is symplectic and satisfies $\K(\nabla) = 0$ and $\rf(\nabla) = \al$.
\end{proof}

For $X \in \symplecto(M, \Om)$ and $f \in \ham(M, \Om)$, the tensors $\lop(X^{\sflat})$ and $\hop(f)$ can be viewed as the vector fields generated on $\symcon(M, \Om)$ by the actions of the flows of $X$ and $\hm_{f}$. The images $\lop(\symplecto(M, \Om)^{\sflat})$ and $\hop(\ham(M, \Om))$ are subbundles of $T\symcon(M, \Om)$, where $\symplecto(M, \Om)^{\sflat} = \{\al \in\Ga(\ctm): d\al = 0\}$ is the space of closed one-forms on $M$, regarded as the space symplectically dual to $\symplecto(M, \Om)$. 
Since the Lie derivative commutes with the decomposition of tensors by symmetries, 
\begin{align}\label{loplie}
\lop([X, Y]^{\sflat}) = \lie_{X}\lop(Y^{\sflat}) - \lie_{Y}\lop(X^{\sflat})
\end{align} 
for $X, Y \in \symplecto(M, \Om)$. It is shown in \cite{Fox-sympauto}, that on a compact symplectic manifold of vanishing Euler characteristic and dimension at least four an infinitesimal automorphism of a symplectic connection need not be symplectic, although this is the case in dimension $2$ and in certain other situations. By \eqref{loplie}, the intersection $\ker \lop \cap \symplecto(M, \Om)^{\sflat}$ and the quotient $\ker \hop/\rea$ (modulo constant functions) are identified with the Lie subalgebras of $\symplecto(M, \Om)$ and $\ham(M, \Om)$, respectively, comprising infinitesimal automorphisms of $\nabla$ that preserve $\Om$. Since the group of automorphisms of an affine connection is a finite-dimensional Lie group (see \cite{Kobayashi-transformationgroups}), this implies $\ker \lop \cap \symplecto(M, \Om)^{\sflat}$ and $\ker \hop$ are finite dimensional. 

\begin{lemma}\label{injectivitylemma}
Let $(M, \Om)$ be a symplectic $2$-manifold and let $\nabla \in \symcon(M, \Om)$. Then $\ker \hop^{\ast}$ comprises $\Pi \in T_{\nabla}\symcon(M, \Om)$ such that $\lop^{\ast}(\Pi)$ is a closed one-form, and $\lop(\symplecto(M, \Om)^{\sflat})^{\perp}$ comprises $\Pi \in T_{\nabla}\symcon(M, \Om)$ such that $\lop^{\ast}(\Pi)$ is an exact one-form. If, moreover, $\nabla \in \K^{-1}(0)$, then
\begin{align}
\hop(\ham(M, \Om)) \subset \lop(\symplecto(M, \Om)^{\sflat}) \subset \lop(\symplecto(M, \Om)^{\sflat})^{\perp} \subset \ker \hop^{\ast} = \hop(\ham(M, \Om))^{\perp}.
\end{align}
As a consequence the linear map map $\ker\hop^{\ast}/\lop(\symplecto(M, \Om)^{\sflat})^{\perp} \to H^{1}(M; \rea)$ defined by $\Pi + \lop(\symplecto(M, \Om)^{\sflat}) \to [-2\lop^{\ast}(\Pi)]$ is injective and this map can be interpreted as the derivative at $\nabla$ of the map $\rfc$ defined in \eqref{rfcdefined}.
\end{lemma}
\begin{proof}

By Lemma \ref{hastkernellemma}, $\ker \hop^{\ast}$ comprises $\Pi \in T_{\nabla}\symcon(M, \Om)$ such that $\lop^{\ast}(\Pi)$ is a closed one-form, and $\lop(\symplecto(M, \Om)^{\sflat})^{\perp}$ comprises $\Pi \in T_{\nabla}\symcon(M, \Om)$ such that $\lop^{\ast}(\Pi)$ is an exact one-form.

When $\nabla \in \K^{-1}(0)$, that $\lop(\symplecto(M, \Om)^{\sflat}) \subset \lop(\symplecto(M, \Om)^{\sflat})^{\perp}$ is immediate from \eqref{isotropyidentity2d}. Consider a path $\nabla(t)$ in $\K^{-1}(0)/\Symplecto(M, \Om)_{0}$ such that $\rfc(\nabla(t))$ is constant. Let $\tnabla(t)$ be a path in $\K^{-1}(0)$ projecting to $\nabla(t)$. Then $\rf(\tnabla(t)) - \rf(\tnabla(0))$ is exact for all $t$ and so, by \eqref{rfvary} of Lemma \ref{hadjointlemma}, $-2\lop^{\ast}(\Pi) = \tfrac{d}{dt}_{t = 0}\rf(\tnabla(t))$ is exact where $\Pi = \tfrac{d}{dt}_{t = 0}\tnabla(t)$. The preceding computations transform naturally under the action of $\Symplecto(M, \Om)_{0}$ on all the objects involved, and so this justifies regarding the map $\Pi + \lop(\symplecto(M, \Om)^{\sflat}) \to [-2\lop^{\ast}(\nabla)]$ as the derivative of $\rfc$ at $\nabla$.
\end{proof}

\section{Relation with the Goldman moment map for projective structures}\label{cohomsection}
This section describes, for a surface $M$, the relation between the moment map $\K$ on $\symcon(M, \Om)$ and the Goldman moment map on the space of of flat real projective structures on $M$. 

There is included a discussion of the projective deformation complex and its use in the construction of a fine resolution of the sheaf of projective Killing fields. The description given here of the cohomology of the sheaf of projective Killing fields on a surface, and its relation to the deformation space of flat projective structures, is modeled on Calabi's treatment in \cite{Calabi-constantcurvature} of the cohomology of the sheaf of Killing fields on a constant curvature Riemannian manifold; see also \cite{Berard-Bergery-Bourguignon-Lafontaine} and \cite{Hangan-resolution}. As is discussed below, it would be interesting to describe similarly deformations of the space of moment flat symplectic connections, but, because of the absence of a local geometric interpretation of the vanishing of $\K$, and despite the close relation with deformations of flat projective structures, it is not clear what form such a description would take.

The space of projective structures on $M$ is written $\projcon(M)$. The space $\prin(\dens)$ of principal connections on  the $\reat$ principal bundle $\dens$ obtained by deleting the zero section from $|\det\ctm|$ is an affine space modeled on $\Ga(\ctm)$. A projective structure $\en \in \projcon(M)$ is in bijection with $\prin(\dens)$ via the map that assigns to $\be \in \prin(\dens)$ the unique $\nabla\in \en$ that induces $\be$. The difference tensor of two projective structures is by definition the difference tensor of their unique representatives corresponding to $\be \in \prin(\dens)$; this does not depend on $\be$ and is trace free. 
Hence $\projcon(M)$ is an affine space modeled on $\{\Pi_{ij}\,^{k} \in \Ga(S^{2}(\ctm)\tensor TM): \Pi_{ip}\,^{p} = 0\}$, which can be taken as the tangent space $T_{\en}\projcon(M)$. 
The map sending $\nabla\in \affcon(M)$ to $(\en, \be) \in \projcon(M) \times \prin(\dens)$, where $\be \in \prin(\dens)$ is induced by $\nabla$, is an affine bijection, equivariant for the action of $\Ga(\ctm)$ on $\affcon(M)$ generating projective equivalence, and the action of $\Ga(\ctm)$ on $\projcon(M) \times \prin(\dens)$ by (appropriately scaled) translations in the second factor; the quotient of $\affcon(M)$ by this action is $\projcon(M)$.
The action of $\diff(M)$ on $\affcon(M)$ by pullback commutes with the action of $\Ga(\ctm)$, so induces an action,  $\phi^{\ast}(\en) = [\phi^{\ast}(\nabla)]$, on $\projcon(M)$, also by pullback. The bijection $\affcon(M) \to \projcon(M) \times \prin(\dens)$ is also $\diff(M)$-equivariant, where the action of $\diff(M)$ on $\projcon(M) \times \prin(\dens)$ is the product action, and the action of $\diff(M)$ on $\prin(\dens)$ is that induced from pullback of densities on $M$. The Lie derivative $\lie_{X}\en$ is the derivative of the difference tensor with $\en$ of the pullback of $\en$ by the flow of $X$. For any $\nabla \in \en$, this the completely trace-free part of $\lie_{X}\nabla$. 
Regarded as a functional on $\projcon(M)$, the projective Cotton tensor $C$ is evidently $\diff(M)$ equivariant in the sense that $C(\phi^{\ast}(\en)) = \phi^{\ast}(C(\en))$ for $\phi \in \diff(M)$. 

Let $M$ be $2$-dimensional. Define sheaves $\cm^{i}$ by $\cm^{0}(U) = \Ga(TU)$, $\cm^{i} = \{0\}$ for $i > 2$, and 
\begin{align}
\begin{split}
\cm^{1}(U)& = \{\Pi_{ij}\,^{k} \in \Ga(S^{2}(\ct U) \tensor TU): \Pi_{ip}\,^{p} = 0\},\\
\cm^{2}(U)& = \{\si_{ijk} \in \Ga(\tensor^{3}\ct U): \si_{ijk} = \si_{[ij]k}\,\,\text{and}\,\, \si_{[ijk]} = 0\}.
\end{split}
\end{align}
where $U \subset M$ is an open subset. The restriction homomorphisms are given by restriction in the ordinary sense. Define maps $\C^{i}:\cm^{i}\to \cm^{i+1}$ by
\begin{align}
&(\C^{0}X)_{ij}\,^{k} = \lie_{X}\en_{ij}\,^{k},&
  &(\C^{1}\Pi)_{ijk}  = \vr_{\Pi} C(\en)_{ijk}.
\end{align}
A bit of computation shows $2\nabla_{[i}\nabla_{|p|}\Pi_{j]k}\,^{p}  = 2\nabla_{p}\nabla_{[i}\Pi_{j]k}\,^{p}$, and a bit more yields
\begin{align}\label{vrbc}
\begin{split}
(\C^{1}\Pi)_{ijk}  =\vr_{\Pi}C(\en) &= -2\nabla_{[i}\nabla_{|p|}\Pi_{j]k}\,^{p} + 2\Pi_{k[i}\,^{p}R_{j]p}  + \tfrac{2}{3}\Pi_{ki}\,^{p}R_{[pj]} - \tfrac{2}{3}\Pi_{kj}\,^{p}R_{[pi]}\\
&=   -2\nabla_{p}\nabla_{[i}\Pi_{j]k}\,^{p} + 2\Pi_{k[i}\,^{p}R_{j]p}  + \tfrac{2}{3}\Pi_{ki}\,^{p}R_{[pj]} - \tfrac{2}{3}\Pi_{kj}\,^{p}R_{[pi]},
\end{split}
\end{align}
in which $\nabla \in \en$ is arbitrary. Note that this shows that the right side of \eqref{vrbc} is a projectively invariant differential operator, something tedious to check directly. 

\begin{lemma}\label{curvliezerolemma}
Let $M$ be a $2$-dimensional manifold. For $\en \in \projcon(M)$ and $X \in \Ga(TM)$, there holds $\C^{1}\C^{0}(X) = \lie_{X}C(\en)$.
Moreover, $\en$ is flat if and only if $\C^{1}\C^{0} = 0$. That is, the sequence
\begin{align}\label{projdefcomplex}
0 \longrightarrow \cm^{0}  \stackrel{\C^{0}}{\longrightarrow} \cm^{1} \stackrel{\C^{1}}{\longrightarrow} \cm^{2} \longrightarrow 0
\end{align}
is a complex if and only if $\en$ is flat. 
\end{lemma}
\begin{proof}
Let $\phi_{t}$ be the flow of $X$. As $(\en + t\lie_{X}\en) - \phi_{t}^{\ast}\en$ has order at least two in $t$, 
\begin{align}
\begin{split}
\vr_{\lie_{X}\en}C(\en) &= \tfrac{d}{dt}_{|t = 0}C(\en + \lie_{tX}\en) =\tfrac{d}{dt}_{|t = 0}C(\phi_{t}^{\ast}\en)  = \tfrac{d}{dt}_{|t = 0}\phi_{t}^{\ast}C(\en) = \lie_{X}C(\en),
\end{split}
\end{align}
the penultimate equality by the diffeomorphism equivariance of the curvature, and the last equality by definition of the Lie derivative. This shows $\C^{1}\C^{0}(X) = \lie_{X}C(\en)$. If $\en$ is flat then $C(\en) = 0$, so $\C^{1}\C^{0}(X) = \vr_{\lie_{X}\en} C(\en) = 0$. At any $p \in M$, there can be chosen $X$ which vanishes at $p$ and such that $\nabla X$ is the identity endomorphism on $T_{p}M$. For such an $X$ and any tensor $A_{i_{1}\dots i_{r}}^{j_{1}\dots j_{s}}$ there holds $\lie_{X}A = (r-s)A$ at $p$, and so, if $r \neq s$ and $\lie_{X}A = 0$ for all $X \in \Ga(TM)$, then $A$ is identically $0$. Applying this with $A$ taken to be $C_{ijk}$, it follows from $\C^{1}\C^{0}(X) = \lie_{X}C(\en)$ that if $\C^{1}\C^{0}(X) = 0$ for all $X$, then, at every $p$, $C_{ijk} = 0$, so $\en$ is flat. 
\end{proof}
When $\en \in \projcon(M)$ is flat, the complex \eqref{projdefcomplex} is called the \textit{projective deformation complex}.

If $(M, \Om)$ is a $2$-dimensional symplectic manifold the map $\sflat: \C^{1}(U) \to \Ga(U;S^{3}(\ctm))$ defined by $\Pi_{ij}\,^{k} \to (\Pi^{\sflat})_{ijk} = \Pi_{ijk}$ is a linear isomorphism.

\begin{lemma}
Let $(M, \Om)$ be a $2$-dimensional symplectic manifold. For $\nabla \in \symcon(M, \Om)$ and the projective structure $\en$ generated by $\nabla$, there hold
\begin{align}\label{clop}
&\C_{\en}^{0}(X)^{\sflat} = \lop(X^{\sflat}),& &\C_{\en}^{1}(\Pi)_{ijk} = -\lop^{\ast}(\Pi^{\sflat})_{k}\Om_{ij}.
\end{align}
for all $X \in \Ga(TM)$ and $\Pi \in \C^{1}(TM)$. 
Consequently, for any open $U \subset M$ the sequence \eqref{rfsequence} of Lemma \ref{sequencelemma} and the sequence \eqref{projdefcomplex} are isomorphic via symplectic duality, as indicated in the diagram \eqref{pdcsym}:
\begin{align}\label{pdcsym}
\begin{split}
\xymatrix{
0 \ar@{->}[r]& \cm^{0}(U)   \ar@{->}[r]^{\C^{0}} \ar@{<-}[d]_{X^{i} \leftrightarrow X_{i}} & \cm^{1}(U)  \ar@{->}[rr]^{\C^{1}}  \ar@{<->}[d]_{\si_{ij}\,^{k} \leftrightarrow \si_{ijk}} &   &\cm^{2}(U)  \ar@/^/[d]^{\si_{ijk} \to \tfrac{1}{2}\si_{p}\,^{p}\,_{k}}  \ar@{->}[r] &0\\
0 \ar@{->}[r]& \Gamma(T^{\ast}U)   \ar@{<->}[r]^{\lop}  & \Ga(S^{3}(T^{\ast}U))  \ar@{->}[rr]^{-\lop^{\ast}} &&\Ga(T^{\ast}U)  \ar@/^/[u]^{\si_{k} \to \Om_{ij}\si_{k}} \ar@{->}[r] &0
}
\end{split}
\end{align}
Here the vertical arrows are the indicated linear isomorphisms.
In particular,
\begin{align}\label{loplopxgeneral}
-2\lop_{\nabla}^{\ast}\lop_{\nabla}(X^{\sflat}) = \lie_{X}\rf(\nabla) + (\sd_{\nabla} X)\rf(\nabla)
\end{align}
for all $X \in \Ga(TM)$. Consequently $\lop^{\ast}\lop = 0$ if and only if $\rf(\nabla) = 0$.
\end{lemma}
\begin{proof}
In this proof the subscripts indicating dependence on $\nabla$ and $\en$ are omitted for readability.
By definition $\lie_{X}\en$ is the trace-free part of $\lie_{X}\nabla$. Since $(\lie_{X}\nabla)_{ip}\,^{p} = \nabla_{i}\nabla_{p}X^{p} +X^{q}R_{qip}\,^{p} = d_{i}\sd X$,
$(\lie_{X}\en)_{ij}\,^{k} = (\lie_{X}\nabla)_{ij}\,^{k} - \tfrac{2}{3}\delta_{(i}\,^{k}d_{j)}\sd X$, so that 
\begin{align}\label{c0lopb}
(\lie_{X}\en)_{ijk} = (\lie_{X}\nabla)_{ijk} + \tfrac{2}{3}\Om_{k(i}d_{j)}\sd X.
\end{align}
On the other hand, since $dX^{\sflat} = (\sd X) \Om_{ij}$, by \eqref{lienablasym},
\begin{align}
\lop(X^{\sflat})_{ijk} = (\lie_{X}\nabla)_{(ijk)} -\tfrac{2}{3}\nabla_{(i}dX^{\sflat}_{j)k} = (\lie_{X}\nabla)_{(ijk)} +\tfrac{2}{3}\Om_{k(i}d_{j)}\sd X= (\lie_{X}\en)_{ijk},
\end{align}
the last equality by \eqref{c0lopb}. This shows the first equality of \eqref{clop}. By \eqref{vrbc},
\begin{align}\label{c1lopb}
\begin{split}
\C^{1}(\Pi)_{ijk} &= -2\nabla_{p}\nabla_{[i}\Pi_{j]k}\,^{p} + 2\Pi_{k[i}\,^{p}R_{j]p} \\
&= -(\nabla_{p}\nabla_{q}\Pi_{k}\,^{pq} + \Pi_{k}\,^{pq}R_{pq})\Om_{ij} = -\lop^{\ast}(\Pi^{\sflat})_{k}\Om_{ij}.
\end{split}
\end{align}
This shows the second equality of \eqref{clop}. This shows the commutativity of \eqref{pdcsym}.
Finally, since 
\begin{align}\label{lieomrf}
\lie_{X}(\Om \tensor \rf) = \Om \tensor \lie_{X}\rf + dX^{\sflat} \tensor \rf = \Om \tensor (\lie_{X}\rf + (\sd X) \rf),
\end{align}
combining \eqref{c0lopb}, \eqref{c1lopb}, $C_{p}\,^{p}\,_{i} =  \rf_{i}$, and \eqref{lieomrf} yields
\begin{align}\label{loplopcc}
\begin{split}
-\lop^{\ast}\lop(X^{\sflat})_{k}\Om_{ij}& = -\lop^{\ast}(\C^{0}(X)^{\sflat})_{k}\Om_{ij} = \C^{1}\C^{0}(X)_{ijk} \\
&= (\lie_{X}C(\en))_{ijk} = \tfrac{1}{2}\lie_{X}(\Om \tensor \rf)_{ijk} = \tfrac{1}{2}\left((\lie_{X}\rf)_{k}+ (\sd X)\rf_{k}\right)\Om_{ij},
\end{split}
\end{align}
showing \eqref{loplopxgeneral}. If $\rf(\nabla) = 0$, then \eqref{loplopxgeneral} shows $\lop^{\ast}\lop = 0$. If $\lop^{\ast}\lop(X^{\sflat}) = 0$ for alL $X \in \Ga(\ctm)$, then, by \eqref{loplopcc}, $\lie_{X}C(\en) = 0$ for all $X \in \Ga(TM)$, and by Lemma \ref{curvliezerolemma} this implies $C(\en) = 0$, so $\rf_{i} = C_{p}\,^{p}\,_{i} = 0$.
\end{proof}

\begin{remark}\label{pdcremark}
The conclusion of Lemma \ref{curvliezerolemma} transported to the sequence \eqref{rfsequence} is stated in Lemma \ref{sequencelemma}. 
\end{remark}

If $M$ is an oriented surface, then for $\al, \be \in T_{\en}\projcon(M)$ at least one of which is compactly supported it makes sense to integrate the two-form $-2\al_{p[i}\,^{q}\be_{j]q}\,^{p}$ over $M$. There results the $\diff(M)$-invariant symplectic form $\Omega_{\en}(\al, \be) = -2\int_{M}\al_{p[i}\,^{q}\be_{j]q}\,^{p}$ on $\projcon(M)$. 

Regard the space $\vect_{c}(M)$ of compactly supported vector fields on $M$ as the Lie algebra of the identity component of the group $\diff_{c}(M)$ of compactly supported diffeomorphisms. Regard $\lie_{X}\en$ as the vector field on $\projcon(M)$ generated by $X \in \vect(M)$. For $X \in \vect_{c}(M)$ and $A \in \cm^{2}(M)$ integration determines a pairing $\lb X, A \ra = \int_{M}X^{p}A_{ijp}$, in which $X^{p}A_{ijp}$ is regarded as a two-form, that identifies $\cm^{2}(M)$ with a subspace of the dual vector space $\vect_{c}(M)^{\ast}$. In particular, via this identification $C(\en)$ is regarded as taking values in this subspace.

\begin{theorem}[W. Goldman;  \cite{Goldman-convex}, \cite{Goldman}]\label{goldmandeformationtheorem}
Let $M$ be a smooth surface. 
\begin{enumerate}
\item  The projective Cotton tensor is a moment map for the action of $\diff_{c}(M)$ on the space $\projcon(M)$ of projective structures on the oriented surface $M$. Precisely, for $X \in \vect_{c}(M)$, $\en \in \projcon(M)$, and $\Pi \in T_{\en}\prj(M)$, $\vr_{\Pi}\lb C(\en), X\ra = \Omega_{\en}(\lie_{X}\en, \Pi)$, where the first variation $\vr C$ at $\en$ in the direction of $\Pi \in T_{\en}\projcon(M)$ is defined by $\vr_{\Pi}C(\en) = \tfrac{d}{dt}_{|t = 0}C(\en + t\Pi)$ (see \eqref{vrbc} for an explicit formula for $\vr_{\Pi}C$).
\item For a compact surface $M$, the symplectic quotient of $\projcon_{0}(M) = C^{-1}(0)$ by the connected component $\diff(M)_{0}$ of the identity of the group of diffeomorphisms of $M$ is the deformation space of isotopy classes of flat real projective structures on $M$. 
\item If $M$ is compact and $\chi(M) < 0$ then the deformation space $\pmod(M) = \projcon_{0}(M)/\diff(M)_{0}$ is a real analytic manifold of dimension $-8\chi(M)$.
\end{enumerate}
\end{theorem}

Lemma \ref{bothmomentmapslemma} shows the relation between the moment maps $\K$ and $C$ on $\symcon(M, \Om)$ and $\projcon(M)$.
\begin{lemma}\label{bothmomentmapslemma}
For a symplectic $2$-manifold $(M, \Om)$, the maps $\inc:\symcon(M, \Om) \to \projcon(M)$ and $\sinc_{\Om}:\projcon(M) \to \symcon(M, \Om)$ defined by $\inc(\nabla) = \en$ and by setting $\sinc_{\Om}(\en)$ equal to the unique representative of $\en$ preserving $\Om$ are inverse symplectic affine bijections that intertwine the moment maps given by $C$ and $\rf$ in the sense that, for $\en \in \projcon(M)$, $\nabla = \sinc_{\Om}(\en) \in \en$, and $f \in \cinf(M)$ there holds 
\begin{align}\label{bothmomentmaps}
\begin{split}
\lb C(\en), \hm_{f}\ra & = \tfrac{1}{2}\lb \rf, \sd^{\ast}f\ra = -\tfrac{1}{2}\lb \sd \rf , f\ra = \lb \K(\sinc_{\Om}(\en)), f\ra.
\end{split}
\end{align}
\end{lemma}
\begin{proof}
It is straightforward to check that $\inc$ and $\sinc$ are inverse symplectic affine bijections. Once a symplectic form $\Om$ is fixed, the projective Cotton tensor $C$ of the projective structure $\en$ generated by $\nabla \in \symcon(M, \Om)$ can be identified with $\rf$ via contraction with $\Om^{ij}$ as in \eqref{gi5}. The identity \eqref{bothmomentmaps} follows from \eqref{gi5} and \eqref{cdiv}. 
\end{proof}
Lemma \ref{bothmomentmapslemma} can be summarized as saying that the diagram 
\begin{align*}
\begin{split}
\xymatrix{
&& \cinf(M)\\
\symcon(M, \Om) \ar@{->}[urr]^{\K}  \ar@{->}[rr]_{-\tfrac{1}{2}\rf}  \ar@/^/[d]^{\inc}& & \ext^{1}(M)/\dad\ext^{2}(M)  \ar@{->}[u]_{\sd = \dad}\\
 \projcon(M) \ar@{->}[rr]_{C}  \ar@/^/[u]^{\sinc_{\Om}}&& \cm^{2}(M) \ar@{->}[u]_{\Om^{ij}}\\
}
\end{split}
\end{align*}
commutes. 
Theorem \ref{moserlemma} shows that $\inc$ and $\sinc_{\Om}$ descend to the quotients modulo the actions of the relevant groups.

\begin{theorem}\label{moserlemma}
For a finite volume symplectic $2$-manifold $(M, \Om)$, the maps $\inc:\symcon(M, \Om) \to \projcon(M)$ and $\sinc_{\Om}:\projcon(M) \to \symcon(M, \Om)$ induce inverse bijections between $\symcon(M, \Om)/\Symplecto(M, \Om)_{0}$ and $\projcon(M)/\diff_{c}(M)_{0}$, where $\diff_{c}(M)_{0}$ is the path connected component of the identity in $\diff_{c}(M)$.
\end{theorem}
\begin{proof}
It is straightforward to check that $\inc$ and $\sinc$ are inverse symplectic affine bijections. Once a symplectic form $\Om$ is fixed, the projective Cotton tensor $C$ of the projective structure $\en$ generated by $\nabla \in \symcon(M, \Om)$ can be identified with $\rf$ via contraction with $\Om^{ij}$ as in \eqref{gi5}. The identity \eqref{bothmomentmaps} follows from \eqref{gi5} and \eqref{cdiv}. 
If $\nabla \in \symcon(M, \Om)$ and $\phi \in \Symplecto(M, \Om)_{0}$ then $\inc(\phi^{\ast}(\nabla)) = \phi^{\ast}(\inc(\nabla))$, so that the image under $\inc$ of the $\Symplecto(M, \Om)_{0}$ orbit of $\nabla$ is contained in a $\diff_{c}(M)_{0}$ orbit of $\inc(\nabla)$ and $\inc$ descends to a well-defined and evidently surjective map $\inc: \symcon(M, \Om)/\Symplecto(M, \Om)_{0} \to \projcon(M)/\diff_{c}(M)_{0}$. If $\phi \in \diff_{c}(M)_{0}$ then $\phi^{\ast}(\Om)$ and $\Om$ are equal outside of some compact set $K$. Since $\int_{M}\phi^{\ast}(\Om) = \int_{M}\Om$, by a theorem of J. Moser (\cite{Moser}; see also section $1.5$ of \cite{Banyaga}) there is a diffeomorphism $\psi$, supported in $K$ and smoothly isotopic to the identity such that $\psi^{\ast}\circ \phi^{\ast}(\Om) = \Om$. Then $\tau = \phi \circ \psi \in \Symplecto(M, \Om)_{0}$ is smoothly isotopic to $\phi$ and equal to $\phi$ outside a compact set. Therefore, given $\en \in \projcon(M)$, $\sinc_{\Om}(\tau^{\ast}\en)$ preserves $\tau^{\ast}(\Om) = \Om$, so $\sinc_{\Om}(\tau^{\ast}\en) = \tau^{\ast}(\sinc_{\Om}(\en))$. Hence, if the projective structures $\ben$ and $\en$ generated by $\bnabla, \nabla \in \symcon(M, \Om)$ lie in the same $\diff_{c}(M)_{0}$ orbit, then there is a $\tau \in \Symplecto(M, \Om)_{0}$ such that $\ben = \tau^{\ast}\en$ and $\bnabla = \sinc_{\Om}(\ben) = \tau^{\ast}\sinc_{\Om}(\en) = \tau^{\ast}\nabla$, so $\bnabla$ and $\nabla$ lie in the same $\Symplecto(M, \Om)_{0}$ orbit. This shows that $\inc$ is a bijection with inverse induced by $\sinc_{\Om}$. 
\end{proof}

\begin{remark}
Since, by Theorem \ref{moserlemma}, the fiber $\rfc^{-1}([0])  \subset \K^{-1}(0)/\Symplecto(M, \Om)_{0}$ over the trivial cohomology class contains a subset identified with $\projcon_{0}(M)/\diff(M)_{0}$ and this last space has dimension $8\dim H^{1}(M; \rea)$, the fibers of $\rfc$ can be quite large. It would be interesting to know if the fibers of $\rfc$ are finite-dimensional.  Since, by Lemma \ref{injectivitylemma}, the derivative of $\rfc$ is $-2\lop^{\ast}$, this would follow if it could be shown that, when $\K(\nabla) = 0$, then $\lop(\symplecto(M, \Om)^{\sflat})^{\perp}/\lop(\symplecto(M, \Om)^{\sflat})$ is finite-dimensional. 

A reformulation of this question yields the following generalization. For a fixed cohomology class $[\al] \in \H^{1}(M; \rea)$ let $\symcon_{[\al]} = \{\nabla \in \K^{-1}(0): \rf(\nabla) \in [\al]\}$. What can be said about the structure of the quotient space $\symcon_{[\al]}/\Symplecto(M, \Om)_{0} = \rfc^{-1}([\al])$?
\end{remark}

By Theorem \ref{moserlemma} the space of equivalence classes of symplectomorphic projectively flat symplectic connections is identified with the deformation space $\pmod(M) = \projcon_{0}(M)/\diff(M)_{0}$ of isotopy classes of flat real projective structures on $M$. The tangent space of $\pmod(M)$ is naturally identified with the first cohomology of the projective deformation complex. 

Given a flat $\en \in \projcon(M)$ define the presheaf of \textit{projective Killing fields} on $M$ by 
\begin{align}
\projkil(U) = \{X \in \Ga(U, TM): \lie_{X}\en = 0\}
\end{align}
for an open set $U \subset M$. The restriction homomorphisms are given by ordinary restriction of vector fields. It is clear that $\projkil$ is a sheaf of Lie algebras. Let $i:\projkil \to \cm^{0}$ be the inclusion homomorphism. 
By Theorem \ref{projelliptictheorem}, in the projectively flat case the complex of Lemma \ref{projdefcomplex} gives rise to a fine resolution of the sheaf $\projkil$ of projective Killing fields, and it follows that the tangent space to $\pmod(M)$ is identified with the first Cech cohomology $H^{1}(M; \projkil)$.

\begin{theorem}\label{projelliptictheorem}
Let $M$ be a $2$-dimensional manifold. If $\en \in \projcon(M)$ is flat then 
\begin{align}\label{projres}
0 \longrightarrow \projkil \stackrel{i}{\longrightarrow} \cm^{0}  \stackrel{\C^{0}}{\longrightarrow} \cm^{1} \stackrel{\C^{1}}{\longrightarrow} \cm^{2} \longrightarrow 0
\end{align}
is a fine resolution of the sheaf $\projkil$ of projective Killing fields and an elliptic complex. The cohomology of the complex $\C^{\bullet}(M)$ of global sections is isomorphic to the Cech cohomology of the sheaf $\projkil$ of projective Killing fields. If $M$ is compact these cohomologies are finite-dimensional.
\end{theorem}

\begin{remark}Theorem \ref{projelliptictheorem} is motivated by the analogous statement for constant curvature metrics due to Calabi in \cite{Calabi-constantcurvature} (see also \cite{Berard-Bergery-Bourguignon-Lafontaine}). 
That \eqref{projres} is a fine resolution is stated without proof as Theorem $1$ of T. Hangan's \cite{Hangan}, and also as Theorem $2.1$ of \cite{Hangan-resolution} where it is described in more detail, although also without proof (it is stated that the proof will appear in future work). Presumably Hangan's proof was similar to that here; it seems that it was never published.
\end{remark}

\begin{remark}
The sequence $(\cm^{\bullet}, \C^{\bullet})$ of \eqref{projdefcomplex} is a concrete realization of the generalized BGG sequence associated with the adjoint representation of $\sll(3, \rea)$. In particular, Theorem \ref{projelliptictheorem} can be obtained by specializing the main theorem about BGG sequences proved in either \cite{Calderbank-Diemer} or \cite{Cap-Slovak-Soucek}, although the demonstration of this claim requires too much space to be included here (consult also \cite{Eastwood-projective} and \cite{Eastwood-Gover} for discussion of BGG sequences in the context of projective structures).  Although it is mostly formal, some work is required, because to connect Theorem \ref{projelliptictheorem} with the parabolic geometry formalism of parabolic geometries there must be used a lifting construction based on the Thomas or tractor connection. Although the resolution \ref{projres} can be deduced from the general BGG machinery for parabolic geometries, it seems useful to record the simple direct argument given here, and its presentation makes it possible to discuss possible parallels with the more general setting of moment flat symplectic connections. 
\end{remark}

\begin{proof}[Proof of Theorem \ref{projelliptictheorem}]
Lemma \ref{curvliezerolemma} implies that the sequence $(\cm^{\bullet}, \C^{\bullet})$ is a complex. 
That the complex be elliptic means that the associated principal symbol complex is exact over the complement of the zero section. It suffices to check that if $\si_{ij}\,^{k} \in \cm^{1}_{x}$ satisfies $Z_{p}Z_{[i}\si_{j]k}\,^{p} = 0$ for $Z \in T_{x}^{\ast}M \setminus \{0\}$, then there is $A^{i} \in T_{x}M$ such that $\si_{ij}\,^{k} = Z_{i}Z_{j}A^{k} - \tfrac{2}{3}Z_{p}A^{p}Z_{(i}\delta_{j)}\,^{k}$. Because $Z_{p}Z_{[i}\si_{j]k}\,^{p} = 0$ there is $\tau_{i} \in T_{x}M$ such that $Z_{p}\si_{ij}\,^{p} = Z_{i}\tau_{j}$. Then $Z_{j}\tau_{i} = \si_{ji}\,^{k}Z_{k} = \si_{ij}\,^{k}Z_{k} = Z_{i}\tau_{j}$, so there is $c \in \rea$ such that $\tau_{i} = cZ_{i}$. Choose linearly independent $X^{i}, Y^{i} \in T_{x}M$ such that $X^{p}Z_{p} = 1$ and $Y^{p}Z_{p} = 0$ and let $U_{i} \in T_{x}^{\ast}M$ be such that $X^{p}U_{p} = 0$ and $Y^{p}U_{p} = 1$. Since $Z_{k}(\si_{ij}\,^{k} - 3c(Z_{i}Z_{j}X^{k} - \tfrac{2}{3}Z_{(i}\delta_{j)}\,^{k})) = 0$ there are constants $p, q, r \in \rea$ such that 
\begin{align}
\si_{ij}\,^{k} = 3c(Z_{i}Z_{j}X^{k} - \tfrac{2}{3}Z_{(i}\delta_{j)}\,^{k})) + (pZ_{i}Z_{j} + 2qZ_{(i}U_{j)} + r U_{i}U_{j})Y^{k}.
\end{align}
Since $0 = \si_{ip}\,^{p} = qZ_{i} + rU_{i}$ and $Z$ and $U$ are linearly independent, it must be $q = 0 = r$. Setting $A^{i} = 3cX^{i} + pY^{i}$ there results $\si_{ij}\,^{k} =  Z_{i}Z_{j}A^{k} - \tfrac{2}{3}Z_{p}A^{p}Z_{(i}\delta_{j)}\,^{k}$, as claimed.

The sheaves $\cm^{\bullet}$ are fine because they are given by sections of smooth tensor bundles. That the sequence of sheaves \eqref{projres} is exact at the $0$ level is immediate from the definition of $\projkil$. To prove exactness at the level $1$ it suffices to prove that if a tensor $\Pi_{ij}\,^{k} \in \cm^{1}(U)$ satisfies $\vr_{\Pi}C(\en) = 0$ then given any $p \in U$ there is an open neighborhood $V$ containing $p$ and contained in $U$ such that the restriction of $\Pi$ to $V$ is equal to $\lie_{X}\en$ for some $X \in \Ga(TV)$. Because $\en$ is projectively flat, there can be chosen a neighborhood $W$ of $p$ and a representative $\pr$ of the restriction to $W$ of $\en$ such that $\pr$ is a flat affine connection. In the remainder of this proof the word \textit{locally} means \textit{restricting to a smaller open neighborhood of $p$ (if necessary)}. Locally there is a $\pr$-parallel symplectic (volume) form $\Om_{ij}$, which will be used to raise and lower indices. Since any two-form $\al_{ij}$ satisfies $2\al_{ij} = \al_{p}\,^{p}\Om_{ij}$, that $\pr^{p}\A_{pi_{1}\dots i_{k}} = 0$ is equivalent to $\pr_{[i}A_{j]i_{1}\dots i_{k}} = 0$ and so, by the usual Poincaré lemma, implies that locally there is $B_{i_{1}\dots i_{k}}$ such that $A_{ii_{1}\dots i_{k}} = \pr_{i}B_{i_{1}\dots i_{k}}$. This observation will be used several times. By the flatness of $\Pi$ and the hypothesis $\C^{1}(\Pi) = 0$, $\pr_{[i}\pr_{|p|}\Pi_{j]k}\,^{p} = \pr_{p}\nabla_{[i}\Pi_{j]k}\,^{p} = 0$. Hence locally there is a one-form $A_{i}$ such that $\pr_{p}\Pi_{ij}\,^{p} = \pr_{i}A_{j}$. Since $0 = \pr_{p}\Pi_{[ij]}\,^{p} = \pr_{[i}A_{j]}$, again by the Poincaré lemma, locally there is a function $f$ such that $\pr_{p}\Pi_{ij}\,^{p} = \pr_{i}\pr_{j}f$. Locally there is a vector field $X^{i}$ such that $\pr_{p}X^{p} = f$. Then $\pr_{p}(\Pi_{ij}\,^{p} - \pr_{i}\pr_{j}X^{p}) = 0$ and so locally there is a tensor $A_{ij} = A_{(ij)}$ such that $\Pi_{ijk} = \pr_{i}\pr_{j}X_{k} + \pr_{k}A_{ij}$. Then $0 = \Pi_{ip}\,^{p} = \pr_{i}\pr_{p}X^{p} - \pr_{p}A_{i}\,^{p}$ so that $\pr_{p}(A_{i}\,^{p} - \pr_{i}X^{p}) = 0$. Hence there is a vector field $Y_{i}$ such that $A_{ij} = \pr_{i}X_{j} + \pr_{j}Y_{i}$, and so $\Pi_{ijk} = \pr_{i}\pr_{j}X_{k} + \pr_{k}\pr_{i}X_{j} + \pr_{k}\pr_{j}Y_{i}$. Since $0 = A_{p}\,^{p} = \pr_{p}(X^{p} - Y^{p})$, locally there is a function $g$ such that $Y_{i} = X_{i} + g$. Hence $\Pi_{ijk} = 3\pi_{(i}\pr_{j}X_{k)} + \pr_{i}\pr_{j}\pr_{k}g$. Using $\pr_{i}\pr_{j}X_{k} = \pr_{i}\pr_{k}X_{j} + 2\pr_{i}\pr_{[j}X_{k]} =   \pr_{i}\pr_{k}X_{j} + \pr_{i}f \Om_{jk}$ there results $\Pi_{ijk} = 3\pr_{i}\pr_{j}X_{k} - 2\pr_{(i}f\Om_{j)k} + \pr_{i}\pr_{j}\pr_{k}g$. Since $\pr_{p}X^{p} = f$, setting $Z^{i} = 3X^{i} + \pr^{i}g$ there results 
\begin{align}
\begin{split}
\Pi_{ij}\,^{k} = \pr_{i}\pr_{j}Z^{k} - \tfrac{2}{3}\delta_{(i}\pr_{j)}\pr_{p}Z^{p} = \lie_{Z}[\pr] = \lie_{Z}\en.
\end{split}
\end{align}
This completes the proof of the local exactness of \eqref{projres}. By the abstract de Rham theorem the cohomology of the complex $\C^{\bullet}(M)$ of global sections is isomorphic to the Cech cohomology of the sheaf $\projkil$ of projective Killing fields. If $M$ is compact, becase $\C^{\bullet}(M)$ is elliptic, its cohomology is finite-dimensional, by Proposition $6.5$ of \cite{Atiyah-Bott-ellipticI}.
\end{proof}

It is not clear if there is a parallel construction in the more general context of moment flat symplectic connections. For moment flat connections a claim analogous to Theorem \ref{projelliptictheorem}, that would be based on Lemma \ref{vkhlemma}, has not been proved, and it is not clear whether it could be correct. The key ingredient in the description of the space of flat projective structure via symplectic reduction is the fine resolution of the sheaf of projective Killing fields. This construction uses in a fundamental way that a flat projective structure can be locally represented by a flat affine connection. While there is a corresponding complex for moment flat symplectic connections, it is not clear that it yields a resolution. The problem is precisely that a local geometric interpretation of the vanishing of $\K$ is lacking.

\section{Critical symplectic Kähler connections on surfaces}\label{criticalkahlersection}
The Levi-Civita connections of Kähler metrics are among the most studied and accessible examples of symplectic connections, and it is natural to ask when they are moment flat or critical.
Recall from the introduction that a Kähler structure $(g, J, \Om)$ is \textit{critical symplectic} if its Levi-Civita connection $D$ is critical symplectic. 

The main result of this section is Theorem \ref{2dkahlertheorem}, that shows that on a compact surface the Levi-Civita connection of a Kähler structure is critical symplectic if and only if the Kähler metric has constant curvature. 

The well known Lemma \ref{killingdlemma} can be proved by computing the squared $L^{2}$-norm of $(\lie_{X}g)_{ij} = 2D_{(i}X^{\flat}_{j)}$, where $X^{\flat}_{i} = X^{p}g_{ip}$, via an integration by parts using the identity $2(\lie_{X}D)_{i(j}\,^{p}g_{k)p} = D_{i}(\lie_{X}g)_{jk} = 2D_{i}D_{(j}X^{\flat}_{k)}$.
\begin{lemma}[K. Yano \cite{Yano}]\label{killingdlemma}
On an orientable manifold $M$, a compactly supported vector field $X$ is an infinitesimal automorphism of the Levi-Civita connection $D$ of a Riemannian metric $g$ if and only if it is $g$-Killing. That is $(\lie_{X}D)_{ij}\,^{k} = 0$ if and only if $(\lie_{X}g)_{ij} = 0$.
\end{lemma}

\begin{lemma}\label{kahlerlemma}
On a compact $2n$-dimensional manifold $M$, a Kähler structure $(g, J, \Om, D)$ is critical symplectic if and only if the Hamiltonian vector field $X^{i} = -\Om^{ip}D_{p}\K$ generated by $\K(D)$ is real holomorphic. In this case the metric gradient $g^{ip}D_{p}\K(D)$ is also real holomorphic and $\K(D)$ Poisson commutes with the scalar curvature $\sR_{g}$. 
\end{lemma}

\begin{proof}
By definition, $D$ is critical symplectic if and only if $\lie_{X}D = 0$. By Lemma \ref{killingdlemma}, since $M$ is compact, this is the case if and only if $\lie_{X}g = 0$. On a compact Kähler manifold a (real) vector field is metric Killing if and only if it is real holomorphic and preserves the volume form (see Corollary $2.125$ of \cite{Besse}). 
Since $X$ is Hamiltonian it preserves the volume form, and so $X$ is real holomorphic if and only if $D$ is critical symplectic. In this case, as the metric gradient $g^{ip}D_{p}\K$ equals $J_{p}\,^{i}X^{p}$, it is also real holomorphic. Also, since $X$ is $g$-Killing its flow preserves $\sR_{g}$, so $0 = d\sR_{g}(X) = \{\K(D), \sR_{g}\}$.
\end{proof}

By Lemma \ref{kahlerlemma}, for a critical symplectic Kähler structure on a compact manifold the $(1,0)$ part of $\hm_{\K(D)}$ is holomorphic. 
This implies that a Kähler structure on a compact manifold admitting no nontrivial holomorphic vector field is critical symplectic if and only if $\K(D)$ is constant, and so necessarily $0$. More generally:
\begin{corollary}\label{nonposcorollary}
If, on a compact manifold, a Kähler structure with Levi-Civita connection $D$ has nonpositive Ricci curvature, it is critical symplectic if and only if $\K(D)$ is constant.
\end{corollary}

\begin{proof}
The classical Bochner argument shows that the metric gradient $Y$ of $\K(D)$ is parallel. Since $Y$ must vanish where $\K(D)$ assumes its maximum, $Y$ must be identically $0$, so $\K(D)$ is constant.
\end{proof}

On a surface, \eqref{kkahlern} specializes to
\begin{align}
\label{kkahler}
\begin{split}
2\K(D) &= \lap_{g}\sR_{g}.
\end{split}
\end{align}
For a moment flat
 Kähler structure on a compact surface, since $0 = 2\K(D) = \lap_{g}\sR_{g}$, $\sR_{g}$ is constant by the maximum principle. In particular, since a compact surface of genus at least two has no nontrivial holomorphic vector field, on a compact orientable surface of genus at least two, the Levi-Civita connection $D$ of a Kähler structure is critical symplectic if and only if $\sR_{g}$ is constant. With a different argument, the restriction on the genus can be removed, yielding Theorem \ref{2dkahlertheorem}.
\begin{proof}[Proof of Theorem \ref{2dkahlertheorem}]
There is a unique complex structure $J_{i}\,^{j}$ such that $J_{i}\,^{p}J_{j}\,^{q}g_{pq} = g_{ij}$ and the symplectic form $\Om_{ij} = J_{i}\,^{p}g_{pj}$ determines the given orientation. By \eqref{kkahler} the Levi-Civita connection $D$ of $g$ satisfies $2\K(D) = \lap_{g}\sR_{g}$. That constant curvature implies critical symplectic is immediate. If $D$ is critical symplectic, then by Lemma \ref{kahlerlemma}, the metric gradient $X^{i} = g^{ij}D_{j}\K(D)$ is the real part of a holomorphic vector field. On a Riemann surface a vector field is real holomorphic if and only if it is conformal Killing. By Theorem II.$9$ of \cite{Bourguignon-Ezin}, on a compact Riemannian manifold $(M, g)$ with scalar curvature $\sR_{g}$ any conformal Killing vector field $Y^{i}$ satisfies $\int_{M}Y^{p}D_{p}\sR_{g}\,d\vol_{g} = 0$, and so, by integration by parts, 
\begin{align}
\begin{split}
0 &= \int_{M}X^{j} D_{j}\sR_{g}\,d\vol_{g} = \int_{M}g^{ij}D_{i}\K(D)D_{j}\sR_{g} \,d\vol_{g} = -\int_{M}\K(D)\lap_{g}\sR_{g} \,d\vol_{g}= -\emf(D).
\end{split}
\end{align}
Hence $0 = 2\K(D) = \lap_{g}\sR_{g}$, and so $\sR_{g}$ is constant by the maximum principle. 
\end{proof}
Some brief remarks about the characterization of critical symplectic Kähler metrics in higher dimensions are made in Section \ref{higherdimensionalkahlersection}.

\section{Critical symplectic connections of metric origin are projectively flat}\label{metricsection}
Let $M$ be a compact orientable surface. As explained in the proof of Theorem \ref{cohomtheorem}, as a consequence of Moser's theorem, any volume form $\Om$ on $M$ can be realized as the volume form of some Riemannian metric $g$, so there is a distinguished subset of $\symcon(M, \Om)$ comprising the Levi-Civita connections of Riemannian metrics with volume $\Om$. Theorem \ref{2dkahlertheorem} can also be seen as saying that in the distinguished subset of $\symcon(M, \Om)$ constituted by the Levi-Civita connections of Riemannian metrics with volume $\Om$ the only critical symplectic connections are the constant curvature metrics. Fixing $g$ determines the unique complex structure $J$ such that $g_{ij} = J_{j}\,^{p}\Om_{ip}$, so once $g$ has been chosen it makes sense to speak of holomorphic cubic differentials. If the genus of $M$ is at least one, then there are nontrivial holomorphic cubic differentials. A larger distinguished subset of $\symcon(M, \Om)$ comprises those $\nabla \in \symcon(M, \Om)$ that differ from the Levi-Civita connection of a Riemannian metric with volume $\Om$ by the real part of a cubic differential holomorphic with respect to the complex structure determined by the metric. Let $D$ be the Levi-Civita connection of $g$ and consider $\nabla = D + \Pi_{ijp}\Om^{kp}$ where $\Pi$ is the real part of a holomorphic cubic differential. Theorem~\ref{momentflattheorem} shows that, on a compact surface, $\nabla$ is moment flat if and only if it is projectively flat.

\begin{theorem}\label{momentflattheorem}
Let $M$ be an orientable compact surface of genus at least one. Let $\Om$ be the volume form of a Riemannian metric $g$ with Levi-Civita connection $D$ and compatible complex structure $J$. Let $\Pi_{ijk}$ be the real part of a holomorphic cubic differential. The following are equivalent.
\begin{enumerate}
\item\label{mft1} The symplectic connection $\nabla = D + \Pi_{ijp}\Om^{kp}$ is moment flat.
\item\label{mft2} The symplectic connection $\nabla = D + \Pi_{ijp}\Om^{kp}$ is projectively flat.
\item\label{mft3} $\sR_{g} - |\Pi|^{2}_{g}$ is constant, where $\sR_{g}$ is the scalar curvature of $g$ and $|\Pi|_{g}^{2} = g^{ia}g^{jb}g^{jc}\Pi_{ijk}\Pi_{abc}$.
\end{enumerate}
If the genus of $M$ is at least two then the conditions \eqref{mft1}-\eqref{mft3} are equivalent to the condition:
\begin{enumerate}
\setcounter{enumi}{3}
\item\label{mft4}  The symplectic connection $\nabla = D + \Pi_{ijp}\Om^{kp}$ is critical.
\end{enumerate}
\end{theorem}
\begin{proof}
In this proof the operators $\sd$, $\lop$, etc. are those associated with $D$. That $\Pi$ be the real part of a holomorphic cubic differential is equivalent to the conditions that $\Pi$ be $g$-trace tree, $g^{pq}\Pi_{ipq} = 0$, and that $\Pi$ be $D$-divergence free, $g^{pq}D_{p}\Pi_{ijq} = 0$; see Lemmas $3.3$ and $3.5$ of \cite{Fox-2dahs}. In this case $2\Pi^{(3, 0)} = \Pi - i \jfib(\Pi)$ where $\Pi^{(3, 0)}$ is the $(3, 0)$ part of $\Pi$ and $\jfib(\Pi)_{ijk} = J_{k}\,^{p}\Pi_{ijp}$; see Lemma $3.4$ of \cite{Fox-2dahs}. Since $\jfib(\Pi)$ is the real part of the holomorphic cubic differential $i\Pi^{(3, 0)}$, it is also completely symmetric, $g$-trace free, and $D$-divergence free. Since $\sd \Pi_{ij} = \Om^{pq}D_{p}\Pi_{ijq}$ is the $D$-divergence of $\jfib(\Pi)_{ijk} = \Pi_{ijp}J_{k}\,^{p}$, it vanishes. That is $\sd \Pi = 0$. Since the Ricci curvature of $g$ equals $(\sR_{g}/2)g_{ij}$ and $\Pi$ is $g$-trace free, it follows from \eqref{lopast} that $\lop^{\ast}(\Pi) = 0$. Recall the notation used in Lemma \ref{hadjointlemma}. From the fact that $\jfib(\Pi)$ is completely $g$-trace free it follows that $2B(\Pi) = 2J(\Pi)_{ipa}g^{qa}J(\Pi)_{jqb}g^{pb} = |J(\Pi)|^{2}_{g}g_{ij} = |\Pi|_{g}^{2}g_{ij}$; see Lemma $3.2$ of \cite{Fox-2dahs}. Here the tensor norm is that given by complete contraction with the metric. Because $\Pi$ is $g$-trace free, $2T(\Pi)_{i} = 2\Pi_{iap}\Om^{bp}B(\Pi)_{bq}\Om^{qa} =  |\Pi|^{2}_{g}\Pi_{iap}\Om^{bp}g_{bq}\Om^{qa} = 0$. In \eqref{rfvary} the preceding shows that $\rf(\nabla) = \rf(D) - 2\sd B(\Pi)_{j} = \rf(D) - J_{i}\,^{p}D_{p}|\Pi|_{g}^{2}$. This can be written more compactly as $\rf(\nabla) = \rf(D) - 2\sd B(\Pi) = \rf(D) + \star d |\Pi|^{2}_{g}$, where $\star$ is the Hodge star operator. By \eqref{rfkahler}, $\rf(D) = -\star d \sR_{g}$, so by \eqref{rfvary},  
\begin{align}\label{rfholo}
\rf(\nabla) = \rf(D) + \star d|\Pi|_{g}^{2} = \star d(|\Pi|_{g}^{2} - \sR_{g}).
\end{align} 
From \eqref{rfholo} it is immediate that $\nabla$ is projectively flat if and only if $\sR_{g} - |\Pi|_{g}^{2}$ is constant.
By \eqref{rfholo}, $\star \rf(\nabla)$ is exact. If $\K(\nabla) = 0$ then $\rf(\nabla)$ is also closed, so $\star \rf(\nabla)$ is metrically coclosed and hence harmonic. Hence $\star\rf(\nabla)$ is an exact harmonic one-form. On a compact surface, an exact harmonic one-form is identically zero. Hence $\rf(\nabla) = 0$ and $\nabla$ is projectively flat. Finally, by Theorem \ref{noextremaltheorem}, \eqref{mft1} and \eqref{mft4} are equivalent when the genus of $M$ is at least two.
\end{proof}

Theorems \ref{2dkahlertheorem} and \ref{momentflattheorem} have similar characters. 
They can be summarized as saying that critical symplectic connections with a metric character must be projectively flat. It might be interesting to turn this remark on its head and to interpret projectively flat connections as arising from critical symplectic connections subject to some metric compatibility.

\begin{remark}
Together Theorems \ref{noextremaltheorem} and \ref{momentflattheorem} give an alternative proof of Theorem \ref{2dkahlertheorem} for compact surfaces of genus at least two. If a Kähler structure on such a surface is critical symplectic, then by Theorem \ref{noextremaltheorem} it is moment flat, while by Theorem \ref{momentflattheorem} with $\Pi = 0$ it is projectively flat.
\end{remark}

\section{Critical symplectic Kähler connections in higher dimensions}\label{higherdimensionalkahlersection}
The characterizations of Kähler metrics for which the Levi-Civita connection is moment constant or critical are interesting questions also in dimension $2n > 2$. Here, to justify this expectation, some simple results are exhibited for Kähler surfaces. In this case the decomposition of the conformal Weyl tensor into its self-dual and anti-self-dual parts simplifies the expression for $\K(D)$. The most interesting conclusion obtained here is Theorem \ref{k3theorem}, showing that the Levi-Civita connection of the Ricci-flat Yau metric on a K3 surface is not moment constant. 

Recall, from \eqref{conformalweyl}, the definition of the conformal Weyl tensor of a Riemannian metric $g$ on a $2n$-dimensional manifold $M$. When $2n = 4$, the conformal Weyl tensor decomposes orthogonally as the sum $A = A^{+} + A^{-}$ of its self-dual and anti-self-dual parts. If $M$ is compact, the generalized Gauss-Bonnet theorem shows that the Euler characteristic $\chi(M)$ is given by
\begin{align}\label{gaussbonnet4d}
32\pi^{2}\chi(M) = \int_{M}\left(|A^{+}|^{2} + |A^{-}|^{2} + \tfrac{1}{6}\sR_{g}^{2} - 2|E|^{2}\right)\vol_{g},
\end{align}
where $E$ is the trace-free Ricci tensor, and, by the Hirzebruch signature theorem, the signature $\sig(M) = \tfrac{1}{3}\pon_{1}(M)$ satisfies
\begin{align}\label{hirzebruch}
48\pi^{2}\sig(M) = \int_{M}\left(|A^{+}|^{2} - |A^{-}|^{2}\right)\vol_{g}.
\end{align}

\begin{lemma}\label{kahlerscalarlemma}
On a Kähler surface the self-dual part $A^{+}$ of the conformal Weyl tensor 
is given by
\begin{align}\label{sda}
\begin{split}
A^{+}_{ijkl} &  = \tfrac{1}{12}\sR_{g}\left(J_{k[i}J_{j]l} + \Om_{k[i}\Om_{j]l}  - \Om_{ij}\Om_{kl}\right).
\end{split}
\end{align}
\end{lemma}

\begin{proof}
Let $\ext^{2}\ctm = \ext^{+} \oplus \ext^{-}$ be the decomposition into self-dual and anti-self-dual two forms under the action of the Hodge star operator $\star$. With the conventions used here $\star \al_{ij} = - J_{i}\,^{p}J_{j}\,^{q}\al_{pq} + \tfrac{1}{2}\al_{p}\,^{p}\Om_{ij}$. Equation \eqref{sda} results upon rewriting, in the notations in use here, Proposition $2$ of \cite{Derdzinski}. Proposition $2$ of \cite{Derdzinski} is based on Lemma $2.3$ and Theorem $2.6$ of \cite{Gray-invariants}, which describe the decomposition of the space of curvature tensors of Kähler type into $U(2)$ irreducibles. 
 \end{proof}

\begin{lemma}
On a $4$-manifold $M$, the Levi-Civita connection $D$ of a Kähler metric $(g, J, \Om)$ satisfies
\begin{align}\label{kks}
\begin{split}
2\K(D) &= \lap_{g}\sR_{g}  - \tfrac{1}{12}\sR_{g}^{2} + \tfrac{1}{2}|A^{-}|_{g}^{2} = L_{g}\sR_{g} + \tfrac{1}{12}\sR_{g}^{2} +  \tfrac{1}{2}|A^{-}|_{g}^{2},
\end{split}
\end{align}
where $L_{g} = \lap_{g} - \tfrac{1}{6}\sR_{g}$ is the conformal Laplacian. If $M$ is compact, then
\begin{align}\label{intk4}
\int_{M}\K(D)\,vol_{g} = -12\pi^{2}\sig(M).
\end{align}
\end{lemma}
\begin{proof}
A routine calculation using \eqref{sda} and the $J$-invariance of $A^{+}$ shows that  $6|A^{+}|^{2}_{g} = \sR_{g}^{2}$. When $2n = 4$, \eqref{kkahlern} becomes
\begin{align}\label{kks2}
\begin{split}
2\K(D) &= (\lap_{g} - \tfrac{1}{6}\sR_{g})\sR_{g} + \tfrac{1}{2}|A|_{g}^{2}= L_{g}\sR_{g}+ \tfrac{1}{2}|A|_{g}^{2}.
\end{split}
\end{align}
Substituting $6|A^{+}|^{2}_{g} = \sR_{g}^{2}$ in \eqref{kks2} gives \eqref{kks}. If $M$ is compact, taking $2n = 4$ in \eqref{intk} yields \eqref{intk4}.
\end{proof}

The Fubini-Study metric on the complex projective plane is Kähler Einstein and self-dual, so by \eqref{kks}, it is moment constant.

Let $\Sigma$ be a compact orientable surface equipped with a hyperbolic metric of constant scalar curvature $-2$, and let $S$ be the two-sphere of constant scalar curvature $2$. Then the product $S \times \Sigma$ with the product metric is a locally conformally flat Kähler manifold with scalar curvature $0$ and signature $0$. More generally, since a locally conformally flat Kähler $4$-manifold $M$ is self-dual, it has vanishing scalar curvature, and so by \eqref{kks} has $\K(D) = 0$ and signature $0$. 

By a K3 surface is meant a complex dimension $2$ complex manifold with trivial canonical line bundle and vanishing first Betti number. It can be shown that all K3 surfaces are diffeomorphic to a quartic hypersurface in the complex projective plane; in particular a K3 surface is simply-connected. See section VIII of \cite{Barth-Hulek-Peters-VandeVen} for background. By Yau's Theorem, every K3 surface admits a Ricci-flat Kähler metric. Theorem \ref{k3theorem} shows that a Ricci-flat Kähler metric on a K3 surface is not critical symplectic. The preliminary Lemma \ref{asdk3lemma} is needed in its proof.

\begin{lemma}\label{asdk3lemma}
For a Ricci-flat Kähler metric $(g, J, \Om)$ on a K3 surface $(M, J)$, any anti-self-dual two-form must vanish somewhere on $M$.
\end{lemma}

\begin{proof}
Let $\al_{ij}$ be a nowhere-vanishing anti-self-dual two form on $M$. A contradiction will be obtained. That $\al_{ij}$ be anti-self-dual means $J_{i}\,^{p}J_{j}\,^{q}\al_{pq} = \al_{ij}$ and $\al_{p}\,^{p} = 0$. Normalize $\al_{ij}$ so that $|\al|^{2}_{g} = 4$. 
Contracting the identity $\al \wedge \al = -\al \wedge \star \al = -\tfrac{1}{4}|\al|^{2}_{g}\Om\wedge \Om = -\Om\wedge \Om$ with $\Om^{ij}$ and using that $\al_{p}\,^{p} = 0$ yields $\al_{i}\,^{p}\al_{pj} = \Om_{ij}$. From this it follows that the endomorphism $K_{i}\,^{j} = -J^{jp}\al_{ip}$ satisfies $K_{i}\,^{p}K_{p}\,^{j} = -\delta_{i}\,^{j}$ and $J_{p}\,^{j}K_{i}\,^{p} = -\al_{i}\,^{j} = K_{p}\,^{j}J_{i}\,^{p}$, so is an almost complex structure that commutes with $J$ (equivalently, $K$ is an almost complex structure anti-self-adjoint with respect to $g$). Because $K$ commutes with $J$, the canonical orientation induced on $M$ by $K$ is opposite that induced by $J$. Let $\bar{M}$ denote $M$ with this opposite orientation. Since sections of $\ext^{-}$ are self-dual with respect to the orientation induced by $K$, and the rank three bundle of self-dual two forms (with respect to $K$) is isomorphic to the sum of a trivial line bundle and the canonical line bundle (with respect to $K$), the square $c_{1}^{2}(\bar{M}, K)$ of the first Chern class $c_{1}(\bar{M}, K)$ of the canonical line bundle of $\bar{M}$ with respect to $K$ equals the first Pontryagin class $\pon_{1}(\ext^{-}) = \pon_{1}(\bar{M}) + 2c_{2}(\bar{M}, K) = 2\chi(\bar{M}) + 3\sig(\bar{M})$ of $\ext^{-}$, so satisfies $c_{1}^{2}(\bar{M}, K) = -3\sig(M) + 2\chi(M) = 96$, since $\chi(M) = 24$ and $\sig(M) = -16$. Since $96$ is not a square, this is a contradiction.  (This argument is modeled on a similar one in section $3$ of \cite{LeBrun-survey}.)
\end{proof}
Theorem \ref{k3theorem} is proved now, by obtaining a contradiction with Lemma \ref{asdk3lemma}.
\begin{proof}[Proof of Theorem \ref{k3theorem}]
Let $D$ be the Levi-Civita connection of a Ricci-flat Kähler metric $(g, J, \Om)$ on a K3 surface. It will be shown $D$ is not critical symplectic.
Since there are no holomorphic vector fields on a K3 surface (see \cite{Barth-Hulek-Peters-VandeVen}), if $D$ is critical symplectic then, by Lemma \ref{kahlerlemma}, it is moment constant. Hence it suffices to show that $D$ cannot be moment constant. 

It follows from \eqref{kks} that $4\K(D) = |A^{-}|^{2}$. Hence the Levi-Civita connection of a Ricci-flat Kähler metric on a K3 surface is moment constant if and only if $|A^{-}|^{2}$ is constant.  If this constant were zero, then $A^{-}$ would vanish, so $g$ would be flat, and the universal cover of $M$ would be $\rea^{4}$, which is false. Hence it can be supposed that $|A^{-}|^{2}$ is a nonzero constant. This will be shown to be impossible (the argument that follows was suggested to the author by Claude LeBrun). Let $\det(A^{-})$ denote the determinant of the endomorphism $\al_{ij} \to -\tfrac{1}{2}\al_{ab}g^{pa}g^{qb}A^{-}_{pqij}$ of the bundle $\ext^{-}$ of anti-self-dual $2$-forms.
From the Weitzenböck formula 
\begin{align}\label{weitzenbock}
\lap_{g}|A^{-}|_{g}^{2} = 2|DA^{-}|^{2} + R_{g}|A^{-}|^{2}_{g} - 144 \det(A^{-}) + 4g^{ia}g^{jb}g^{kc}g^{dl}g^{pq}A^{-}_{abcd}D_{i}D_{q}A^{-}_{kljp}, 
\end{align}
(see, for example, section $4$ of \cite{Bourguignon})
it follows that for a Ricci-flat Kähler metric there holds $\lap_{g}|A^{-}|_{g}^{2} = 2|DA^{-}|^{2} - 144 \det(A^{-})$. (Recall that here the norms are those given by complete contraction, so differ by factors of $2$ and $4$ from those found in many references; in particular, $|A^{-}|^{2}$ is four times the sum of the squares of the eigenvalues of $A^{-}$ viewed as an endomorphism of $\ext^{-}$.) The Ricci-flat condition is used to conclude that $g^{pq}D_{i}D_{q}A^{-}_{kljp}= 0$; this follows from the differential Bianchi identity (see \cite{Besse}, chapter $16$).
Since $|A^{-}|^{2}\neq 0$, $0 < |DA^{-}|^{2} = 72\det(A^{-})$, and so $A^{-}$ has exactly one positive eigenvalue. Since $M$ is simply-connected there is a nowhere vanishing section of $\ext^{-}$ that is an eigenvector $\al_{ij}$ of $A^{-}$ corresponding to the positive eigenvalue. This contradicts Lemma \ref{asdk3lemma}. 
\end{proof}

\begin{corollary}
The Levi-Civita connection of a Ricci-flat Kähler metric on a compact $4$-manifold $M$ is critical symplectic if and only if it is flat, in which case $M$ must be a torus. 
\end{corollary}
\begin{proof}
By \eqref{gaussbonnet4d} and \eqref{hirzebruch}, equality holds in the Hitchin-Thorpe inequality, and the result follows by the characterization of this case given in \cite{Hitchin-compactfour}. 
\end{proof}

\begin{remark}
By Theorem A of \cite{LeBrun-Maskit}, the connected sum $M_{k} = \cp^{2}\#k\bcp^{2}$ admits a scalar-flat anti-self-dual Kähler metric. Since $\sig(M_{k}) = 1-k$ and $\chi(M_{k}) = 3 + k$, $3\sig(M_{k}) + 2\chi(M_{k}) = 9 -k$. When $k \geq 10$, this last quantity is negative, so by the Hitchin-Thorpe inequality $M_{k}$ admits no Einstein metric; in particular the scalar-flat Kähler metrics on it are not Ricci flat. The argument proving Theorem \ref{k3theorem} would work for these manifolds without changes, with the exception of the step using the Weitzenböck formula \eqref{weitzenbock}; if the anti-self-dual Weyl tensor is not harmonic, the conclusion about its eigenvalues does not follow. In fact, by the main theorem of \cite{Apostolov-Calderbank-Gauduchon}, a scalar-flat Kähler surface with harmonic anti-self-dual Weyl tensor is either Ricci flat or locally a product of constant curvature Riemann surfaces, so a scalar-flat Kähler metric on $M_{k}$ does not have harmonic Weyl tensor for $k \geq 10$.
\end{remark}

\begin{lemma}\label{yamlemma}
Let $M$ be an oriented compact $4$-manifold with signature $\sig(M) \geq 0$. Suppose the Levi-Civita connection $D$ of a Kähler metric $(g, J, \Om)$ inducing the given orientation and having Yamabe invariant $\yam(g)$ is moment constant.
\begin{enumerate}
\item\label{ruled0} If $\yam(g) > 0$ then $\sR_{g} > 0$.
\item\label{ruled1} If $\yam(g) = 0$ then $(M, g)$ is either flat or locally isometric to a Kähler product of two Riemann surfaces of opposite constant scalar curvatures.
\end{enumerate}
\end{lemma}
\begin{proof}
By Lemma $1.2$ of M. Gursky's \cite{Gursky} if the scalar curvature $R_{g}$ of a Riemannian metric on a compact $4$-manifold satisfies $L_{g}\sR_{g} \leq 0$ then $\sR_{g} > 0$ if $\yam(g) > 0$ and $\sR_{g}$ is identically zero if $\yam(g) = 0$. If $M$ is compact, then, by \eqref{intk4}, $\K(D)\vol_{g}(M) = -12\pi^{2}\sig(M)$. Together with \eqref{kks} this yields
\begin{align}\label{kksc}
\begin{split}
0 & = \lap_{g}\sR_{g}  - \tfrac{1}{12}\sR_{g}^{2} + \tfrac{1}{2}|A^{-}|_{g}^{2}  + \tfrac{24\pi^{2}\sig(M)}{\vol_{g}(M)} = L_{g}\sR_{g} + \tfrac{1}{12}R_{g}^{2} +  \tfrac{1}{2}|A^{-}|_{g}^{2}  + \tfrac{24\pi^{2}\sig(M)}{\vol_{g}(M)}.
\end{split}
\end{align}
Because $\sig(M) \geq 0$, \eqref{kksc} yields that $L_{g}\sR_{g} \leq 0$, and so Gursky's lemma applies. In the case $\yam(g) = 0$, so that $\sR_{g}$ is identically zero, then, by \eqref{kksc} and the nonnegativity of $\sig(M)$, $A^{-}$ vanishes and $\sig(M) = 0$. In particular $g$ is conformally flat. By the main theorem of \cite{Tanno} if $M$ is not flat then it is locally isometric to a product of Riemann surfaces of opposite curvatures.
\end{proof}

\section{Remarks on the moment \texorpdfstring{$4$}{4form}-form and Einstein-like conditions}\label{concludingsection}
It has been argued that moment constant and critical symplectic connections are analogues of constant scalar curvature and extremal Kähler metrics. To complete the analogy there should be identified classes of symplectic connections that correspond to Kähler Einstein metrics. This section records some remarks in this direction.

\subsection{The moment $4$-form}\label{momentsection}
The \textit{moment form} $\mf(\nabla)$ of $\nabla \in \symcon(M, \Om)$ is the closed $4$-form 
\begin{align}\label{mfdefined}
\mf  = \mf(\nabla) =  -2\pi^{2}\pon_{1} - \tfrac{1}{4(n-1)}d\rf \wedge \Om.
\end{align}
It satisfies:
\begin{enumerate}
\item\label{mfcon1} $\mf(\phi^{\ast}\nabla) = \phi^{\ast}\mf(\nabla)$ for $\phi \in \Symplecto(M, \Om)$.
\item\label{mfcon2} $-\tfrac{1}{2\pi^{2}}\mf(\nabla)$ represents the first Pontryagin class of $M$.
\item\label{mfcon3} In the Lefschetz decomposition of $\mf(\nabla)$ into its primitive parts, the zeroth order part equals the moment map $\K(\nabla)$. In particular, by \eqref{ponom},
\begin{align}\label{kcohom}
 -4 \mf \wedge \Omk{(n-2)} 
= d\rf \wedge \Omk{(n-1)} + 8\pi^{2}\pon_{1}\wedge \Omk{(n-2)}  = -2\K(\nabla)\Omk{n}.
\end{align}
\item\label{mfcon4} $\mf(\nabla)$ vanishes if $\nabla$ is projectively flat (see Lemma \ref{projflatlemma}). Precisely, by \eqref{gi52n}, \eqref{pon}, and \eqref{mfdefined} the moment form $\mf(\nabla)$ of a projectively flat $\nabla \in \symcon(M, \Om)$ satisfies $4(1-n)\mf(\nabla) = d\rf \wedge \Om = 0$.
\end{enumerate}
Item \eqref{mfcon2} is analogous to the statement that a multiple of the Ricci form of a Kähler manifold represents the manifold's first Chern class, while \eqref{mfcon3} is analogous to the statement that zeroth part of the Lefschetz decomposition of the Ricci form of a Kähler metric is the metric's scalar curvature. 

By these observations $\K(\nabla)$ and $\mf(\nabla)$ seem to be appropriate analogues for symplectic connections of the scalar curvature and Ricci form of a Kähler metric. The Kähler Einstein condition is equivalent to the statement that the Ricci form is harmonic, and, since the Ricci form is always closed, the condition is really that the Ricci form be coclosed; equivalently, the primitive part of the Ricci form vanishes. This suggests that an analogue for symplectic connections of the Kähler Einstein condition can be formulated in terms of the primitive parts of the moment form $\mf(\nabla)$. Since the Lefschetz decomposition of $\mf(\nabla)$ has parts of orders $0$, $2$, and $4$, there are various possibilities. 
An example of a natural question suggested by these remarks is: \textit{if the first Pontryagin class $\pon_{1}(M)$ of a symplectic manifold $(M, \Om)$ satisfies $-2\pi^{2}\pon_{1}(M) = [\la \Om \wedge \Om]$ for some $\la \in \rea$, must there exist $\nabla \in \symcon(M, \Om)$ such that $\mf(\nabla) = \la \Om \wedge \Om$?} Other similar questions can be posed, but the natural context for such questions appears to be some form of symplectic cohomology, such as the primitive cohomologies discussed in \cite{Tseng-Yau, Tseng-Yau-II}, rather than the ordinary de Rham cohomology. Although this requires too much background to discuss further here, it is planned to discuss these possibilities in detail in future work.
Just as constant curvature metrics on surfaces play a special role in the theory of Kähler Einstein metrics, as a degenerate case of the definition requiring a separate treatment, the conditions on symplectic connections in the $4$-dimensional case require special handling.

\subsection{Symplectic Bach tensors}\label{symplecticbachsection}
Given $\nabla \in \symcon(M, \Om)$, define an operator $\W:\Ga(S^{3}(\ctm)) \to \Ga(\weylmod(\ctm, \Om))$ as the first variation $\vr_{\Pi}W(\nabla) = \tfrac{d}{dt}\big|_{t = 0}W(\nabla + t\Pi)$ of the symplectic Weyl tensor of $\nabla$ in the direction of $\Pi$, viewing $\Pi$ as an element of $T_{\nabla}\symcon(M, \Om)$. From \eqref{dnwdefined} and \eqref{weylvar} it follows that $\W(\Pi) = \dnw\Pi$, where $\dnw$ is the operator defined in \eqref{dnwdefined}.
The symplectically adjoint operator $\W^{\ast}:\Ga(\weylmod(\ctm, \Om)) \to \Ga(S^{3}(\ctm))$ is defined, via the integration pairing, by 
\begin{align}
\int_{M}\W(\Pi)_{ijkl}A^{ijkl}\,\Omk{n} = -\int_{M}\Pi_{ijk}\W^{\ast}(A)^{ijk}\,\Omk{n}. 
\end{align}
When $A_{ijkl}$ is compactly supported, straightforward integration by parts shows that
\begin{align}
\W^{\ast}(A)_{ijk} = -2\nabla^{p}A_{p(ijk)} - \tfrac{4}{n+1}\nabla_{(i}A^{p}\,_{jk)p},
\end{align}
which is taken as the definition of $\W^{\ast}$ in general. In particular it makes sense to apply $\W^{\ast}$ to $W_{ijkl}$.
Differentiating \eqref{symplecticweyl} yields 
\begin{align}
\label{divw}\begin{split}
2(n+1)\nabla^{p}W_{pijk}  & = (2n+1)\nabla_{i}R_{jk} - 3\nabla_{(i}R_{jk)} + \Om_{i(j}\rf_{k)},
\end{split}
\end{align}
and from \eqref{divw} there results
\begin{align}\label{wastw}
\W^{\ast}(W)_{ijk} = \tfrac{2(n-1)}{n+1}\sd^{\ast} \ric_{ijk}.
\end{align}
Applying $\sd$ to \eqref{wastw}, using \eqref{sdsdast}, and simplifying the result yields
\begin{align}\label{sdwastw}
\sd \W^{\ast}(W)_{ij} = \tfrac{2(n-1)}{n+1}\sd\sd^{\ast}\ric_{ij} = \tfrac{4(n-1)}{3(n+1)}\nabla^{p}\nabla_{(i}R_{j)p}.
\end{align}
On a compact $4$-manifold the critical points of the squared $L^{2}$-norm of the conformal Weyl tensor of a Riemannian metric are given by the vanishing of a trace-free symmetric $2$-tensor called the Bach tensor (see \cite{Pedersen-Swann} or \cite{Graham-Hirachi-obstruction}). Alternatively, the Bach tensor is the image of the conformal Weyl tensor under the metric adjoint of the second order differential operator given by the linearization of the operator associating to a conformal structures its conformal Weyl tensor (see \cite{Calderbank-Diemer}). For $\nabla \in \symcon(M, \Om)$ this motivates regarding the tensors $\W^{\ast}(W)$ and $\sd \W^{\ast}(W)$ as analogues of the Bach tensor, as their construction is formally parallel to that of the Bach tensor via linearization. 
The vanishing of these tensors gives the equations for the critical points of the functional $\int_{M} W_{ijkl}W^{ijkl}\,\Omk{n}$ with respect to arbitrary variations of $\Pi$ and gauge variations of $\Pi$, respectively, and so the analogy with the Bach tensor is also consistent with the variational derivation of the Bach tensor. By the identity \eqref{wastw}, the critical points of $\int_{M} W_{ijkl}W^{ijkl}\,\Omk{n}$ with respect to arbitrary variations are preferred connections; as was explained in the introduction, this is because the all functionals quadratic in the curvature of a symplectic connection have the same critical points.

In dimension $4$ the metric Bach tensor vanishes for conformally Einstein metrics as well as half conformally flat metrics. If the analogy is taken seriously, this suggests that that the preferred and gauge preferred conditions are reasonable symplectic analogues of the Einstein and conformal Einstein conditions for metrics. At any rate, it gives additional motivation for studying these conditions.

\bibliographystyle{amsplain}
\def\polhk#1{\setbox0=\hbox{#1}{\ooalign{\hidewidth
  \lower1.5ex\hbox{`}\hidewidth\crcr\unhbox0}}} \def\cprime{$'$}
  \def\cprime{$'$} \def\cprime{$'$}
  \def\polhk#1{\setbox0=\hbox{#1}{\ooalign{\hidewidth
  \lower1.5ex\hbox{`}\hidewidth\crcr\unhbox0}}} \def\cprime{$'$}
  \def\cprime{$'$} \def\cprime{$'$} \def\cprime{$'$} \def\cprime{$'$}
  \def\cprime{$'$} \def\polhk#1{\setbox0=\hbox{#1}{\ooalign{\hidewidth
  \lower1.5ex\hbox{`}\hidewidth\crcr\unhbox0}}} \def\cprime{$'$}
  \def\Dbar{\leavevmode\lower.6ex\hbox to 0pt{\hskip-.23ex \accent"16\hss}D}
  \def\cprime{$'$} \def\cprime{$'$} \def\cprime{$'$} \def\cprime{$'$}
  \def\cprime{$'$} \def\cprime{$'$} \def\cprime{$'$} \def\cprime{$'$}
  \def\cprime{$'$} \def\cprime{$'$} \def\cprime{$'$} \def\cprime{$'$}
  \def\dbar{\leavevmode\hbox to 0pt{\hskip.2ex \accent"16\hss}d}
  \def\cprime{$'$} \def\cprime{$'$} \def\cprime{$'$} \def\cprime{$'$}
  \def\cprime{$'$} \def\cprime{$'$} \def\cprime{$'$} \def\cprime{$'$}
  \def\cprime{$'$} \def\cprime{$'$} \def\cprime{$'$} \def\cprime{$'$}
  \def\cprime{$'$} \def\cprime{$'$} \def\cprime{$'$} \def\cprime{$'$}
  \def\cprime{$'$} \def\cprime{$'$} \def\cprime{$'$} \def\cprime{$'$}
  \def\cprime{$'$} \def\cprime{$'$} \def\cprime{$'$} \def\cprime{$'$}
  \def\cprime{$'$} \def\cprime{$'$} \def\cprime{$'$} \def\cprime{$'$}
  \def\cprime{$'$} \def\cprime{$'$} \def\cprime{$'$} \def\cprime{$'$}
  \def\cprime{$'$} \def\cprime{$'$} \def\cprime{$'$} \def\cprime{$'$}
\providecommand{\bysame}{\leavevmode\hbox to3em{\hrulefill}\thinspace}
\providecommand{\MR}{\relax\ifhmode\unskip\space\fi MR }
\providecommand{\MRhref}[2]{%
  \href{http://www.ams.org/mathscinet-getitem?mr=#1}{#2}
}
\providecommand{\href}[2]{#2}

\end{document}